\documentclass[11pt]{amsart}

\usepackage{amsmath}
\usepackage{amsthm}
\usepackage{hyperref}
\usepackage{xcolor}
\hypersetup{
	colorlinks,
	linkcolor={blue!80!black},
	citecolor={green!60!black},
	urlcolor={blue!80!black}
}
\usepackage[margin=1.0in]{geometry}
\usepackage[T1]{fontenc}
\usepackage{MnSymbol}

\numberwithin{equation}{section}

\newtheorem{prop}{Proposition}
\newtheorem{lemma}[prop]{Lemma}

\newtheorem{thm}[prop]{Theorem}
\newtheorem{cor}[prop]{Corollary}

\numberwithin{prop}{section}

\theoremstyle{definition}
\newtheorem{defn}[prop]{Definition}

\newtheorem{rmk}[prop]{Remark}
\newtheorem{qtn}[prop]{Question}

\DeclareSymbolFont{script}{U}{eus}{m}{n}
\DeclareSymbolFontAlphabet{\mathscr}{script}
\DeclareMathSymbol{\Wedge}{0}{script}{"5E}
\DeclareMathAlphabet{\mathrmsl}{OT1}{cmr}{m}{sl}

\newcommand{\del}{\partial}
\newcommand{\dt}{\tfrac{\partial}{\partial t}}

\newcommand{\gs}{\sigma}

\newcommand{\poiss}{{\mathrm \pi}}

\newcommand{\gk}{\kappa}

\newcommand{\gw}{\omega}
\newcommand{\ga}{\alpha}

\newcommand{\Jj}{J}
\newcommand{\Ii}{I}
\newcommand{\sj}{-}
\newcommand{\msj}{+}
\newcommand{\J}{{\mathbb J}}
\newcommand{\half}{\tfrac{1}{2}}
\newcommand{\nonhalf}{}
\newcommand{\twice}{2}
\newcommand{\nontwice}{}
\newcommand{\cc}{\frac{1}{4}}
\newcommand{\ccc}{ }
\newcommand{\CCF}{{C^{\infty}_0(M, {dV_F})}}

\newcommand{\GK}{\mathcal{GK}}
\newcommand{\KK}{\mathcal K}
\newcommand{\AK}{\mathcal{AK}}
\newcommand{\AGK}{\mathcal{AGK}}
\newcommand{\AC}{\mathcal{AC}}
\newcommand{\II}{{\mathbb I}}
\newcommand{\JJ}{{\mathbb J}}
\newcommand{\bxi}{\boldsymbol{\xi}}
\newcommand{\til}[1]{\widetilde{#1}}

\renewcommand{\bar}[1]{\overline{#1}}

\renewcommand{\i}{\sqrt{-1}}
\newcommand{\Gu}{\gamma} 

\renewcommand{\part}{\del}

\newcommand{\Scal}{\mathrm{Scal}}
\newcommand{\Gscal}{\mathrm{Gscal}}

\newcommand{\R}{{\mathbb R}}
\newcommand{\C}{{\mathbb C}}

\newcommand{\T}{{\mathbb T}}
\newcommand{\tor}{{\mathfrak t}}

\newcommand{\Lab}{{\mathrm L}}
\newcommand{\Pol}{{\mathrm P}}
\newcommand{\ang}{\theta}
\newcommand{\kom}{\omega}
\newcommand{\Hess}{{\mathrm {Hess}}}
\newcommand{\Aut}{{\mathrm {Aut}}}
\newcommand{\hred}{\mathfrak{h}_{\mathrm{red}}}

\newcommand{\la}{\langle}
\newcommand{\ra}{\rangle}

\DeclareMathOperator{\tr}{tr}
\DeclareMathOperator{\Ker}{Ker}

\DeclareMathOperator{\grad}{grad}
\DeclareMathOperator{\Vol}{Vol}

\DeclareMathOperator{\End}{End}

\DeclareMathOperator{\Ham}{Ham}

\begin{document}

\title[Geometry of the space of generalized K\"ahler structures]{The Riemannian and symplectic geometry of the space of generalized K\"ahler structures}

\begin{abstract} On a compact complex manifold $(M, J)$ endowed with a holomorphic Poisson tensor $\poiss_J$ and a deRham class $\alpha\in H^2(M, \R)$, we study the space of generalized K\"ahler (GK) structures defined by a symplectic form $F\in \alpha$ and whose holomorphic Poisson tensor is $\poiss_J$.  We define a notion of generalized K\"ahler class of such structures, and use the moment map framework of Boulanger~\cite{boulanger} and Goto~\cite{Gotomoment} to extend the Calabi program to GK geometry.  We obtain generalizations of the Futaki--Mabuchi extremal vector field~\cite{Futaki-Mabuchi} and Calabi--Lichnerowicz--Matsushima result~\cite{calabi2, Lichnerowicz, Matsushima} for the Lie algebra of the group of  automorphisms of $(M, J, \poiss_J)$.  We define a closed $1$-form on a GK class, which yields a generalization of the Mabuchi energy and thus a variational characterization of GK structures of constant scalar curvature.  Next we introduce a formal Riemannian metric on a given GK class, generalizing the fundamental construction of Mabuchi--Semmes--Donaldson~\cite{Mabuchi,Semmes,donaldson-GIT}.  We show that this metric has nonpositive sectional curvature, and that the Mabuchi energy is convex along geodesics, leading to a conditional uniqueness result for constant scalar curvature GK structures.  We finally examine the toric case, proving the uniqueness of extremal generalized K\"ahler structures and showing that their existence is obstructed by the uniform relative K-stability of the corresponding Delzant polytope.  Using the resolution of the Yau--Tian--Donaldson conjecture in the toric case by Chen--Cheng~\cite{CC} and He~\cite{He}, we show in some settings that this condition suffices for existence and thus construct new examples.
\end{abstract}

\date{February 22, 2023}

\author{Vestislav Apostolov}
\address{V.\,Apostolov\\ D{\'e}partement de Math{\'e}matiques\\ UQAM \\
 and \\ Institute of Mathematics and Informatics\\ Bulgarian Academy of Sciences}
\email{\href{mailto:apostolov.vestislav@uqam.ca}{apostolov.vestislav@uqam.ca}}

\author{Jeffrey Streets}
\address{J.\,Streets \\ Rowland Hall\\
        University of California\\
        Irvine, CA 92617}
\email{\href{mailto:jstreets@uci.edu}{jstreets@uci.edu}}

\author{Yury Ustinovskiy}
\address{Y.\,Ustinovskiy \\ Chandler-Ullmann Hall\\ Lehigh University, Bethlehem, PA, 18015}
\email{\href{mailto:yuu221@lehigh.edu}{yuu221@lehigh.edu}}

\maketitle

\section{Introduction}

E.\,Calabi~\cite{calabi} initiated a far reaching  program of finding, on a given compact K\"ahler manifold $(M, \Jj)$, a canonical representative of the space ${\KK}_{\alpha}$ of K\"ahler metrics belonging to a fixed deRham class $\alpha \in H^2(M, \R)$.  He proposed as a candidate of such  representative the notion of  \emph{extremal K\"ahler metric}, i.e., one whose scalar curvature ${\Scal}_{\omega}$ defines a Hamiltonian vector field $\chi = -\omega^{-1}(d{\Scal}_{\omega})$ satisfying ${\mathcal L}_{\chi} \Jj=0$.  This problem unifies the existence problems for constant scalar curvature (csc) and K\"ahler--Einstein metrics and represents one of the most active areas of research in K\"ahler geometry during the  last half-century.  The central conjecture in the field, still open in full generality, is the \emph{Yau--Tian--Donaldson} (YTD) conjecture.  It states, broadly speaking, that the full obstruction for ${\KK}_{\alpha}$ to admit an extremal K\"ahler metric can be expressed in terms of a complex-algebraic notion of stability of $(M, \Jj, \alpha)$~\cite{donaldsonJDG-02, Sz, Tian,Yau}.  This correspondence, if established, will have further deep implications for the definition of well-behaved moduli spaces of K\"ahler manifolds~\cite{Dervan-Naumann, Fujiki-Schumacher}. 

Only two years after Calabi's seminal paper appeared, an extension of K\"ahler geometry emerged from studies in $(2,2)$ supersymmetric quantum field theory in physics~\cite{Gates}.  These geometric structures were later rediscovered, and given the name of  \emph{generalized K\"ahler (GK) structures}, in the context of Hitchin's generalized geometry program \cite{Gualtieri-PhD,Gualtieri-CMP,HitchinGCY}.  In the ensuing decades it has become clear that GK geometry is a deeply structured extension of K\"ahler geometry with novel implications for complex, symplectic and Poisson geometry.  In this paper we motivate a natural extension of the Calabi program to the setting of GK structures compatible with a given holomorphic Poisson tensor $\poiss_J$  and a deRham class $\alpha\in H^2(M, \R)$.

\subsection{Generalized K\"ahler structures of symplectic type}

To begin we first recall the classical \emph{biHermitian} definition of generalized K\"ahler structure~\cite{Gates}.  Here a GK structure consists of a quadruple $(g, b, I, J)$ of a Riemannian metric $g$ with compatible integrable complex structures $I$, $J$, and a $2$-form $b$,  satisfying
\begin{align*}
- d^c_{\Ii} \gw_{\Ii} = db = d^c_{\Jj} \gw_{\Jj},
\end{align*}
where $\gw_I=gI, \, \gw_J=gJ$ are the fundamental $2$-forms of $(g, I)$ and $(g, J)$.
The case when $I=J$ and $b=0$  gives rise to a K\"ahler metric $(g, \Jj)$.  A key point, observed by Hitchin~\cite{HitchinPoisson}, is that for any GK structure the tensor
\[  \poiss := \tfrac{1}{2} [I, J]g^{-1}\]
is a bivector which defines a real Poisson structure whereas  the complex bivectors
 \[ \poiss_{\Jj} := \poiss - \i \Jj \poiss , \qquad \poiss_{\Ii} := \poiss - \i \Ii\poiss  \]
are \emph{holomorphic Poisson structures} defined respectively on $(M, \Jj)$ and $(M, \Ii)$.
We will be interested in the special case of GK structures of \emph{symplectic type}.  In this setting we fix a complex manifold $(M, J)$ and consider a symplectic form $F$ on $M$ which \emph{tames} $J$, i.e.,such that
\[-F\Jj = g + b, \]
where $g$ is a Riemannian metric and $b$ is a $2$-form on $M$.  Given this, we define a second $g$-compatible almost complex structure by
\[ \Ii =  - F^{-1} \Jj^* F.\]
If $\Ii$ is also integrable, then $(g, b, I, J)$ is GK as defined above.  Note that for a fixed complex structure $\Jj$, the entire quadruple $(g, b, I, J)$ is encoded by $F$, while the integrability of $\Ii$ places a further nonlinear, first-order differential condition on $F$.  We denote by  ${\GK}_{\poiss, \alpha}$ the space of such symplectic type generalized K\"ahler structures, compatible with the holomorphic Poisson manifold $(M, \Jj, \poiss_{\Jj})$ and satisfying $[F] = \alpha$.  In the case when $\poiss_J=0$, the space  ${\GK}_{0, \alpha}$ is just the space ${\KK}_{\alpha}$ of K\"ahler metics in a fixed cohomology class.

A fundamental issue in extending the Calabi program to this setting is the nonlinear structure of ${\GK}_{\poiss, \alpha}$.  In the K\"ahler case, using the $dd^c_{\Jj}$-lemma, the space ${\KK}_{\alpha}$ is a convex-linear, Fr\'echet  manifold modeled on $C^{\infty}(M,\R)/\R$.  Such a description is no longer globally possible in the GK case due to the integrability condition on $\Ii$ (cf.\,\cite{BGZ,LindstromGKpotential} for results on \emph{local} generalized K\"ahler potentials).  Nonetheless, a natural notion of \emph{generalized K\"ahler class} has now emerged \cite{BGZ,GiS,Gualtieri-Hamiltonian}, which here consists of GK structures defined by deforming an element $F_0\in {\GK}_{\poiss, \alpha}$ by a smooth path of functions $\phi_t \in  C^{\infty}(M,\R)/\R$ in the following nonlinear way~\cite{Gualtieri-Hamiltonian}:
\[
	F_{\phi_t} := F_0 + \int_{0}^t dd^c_{\Ii_s} \phi_s \, ds, \qquad \Ii_s := (\Phi_s)_* I_0 (\Phi_s)_*^{-1},
\]
where $\Phi_t$ is the isotopy of diffeomorphisms corresponding to the time dependent vector field $\sj\nonhalf\poiss(d\phi_t)$. It turns out that if $F_{\phi_t}$ tames $\Jj$ (an open condition), then  $(F_{\phi_t}, \Jj)$ gives rise to an element of ${\GK}_{\poiss, \alpha}$ with $\Ii_t$ as defined above. In Proposition~\ref{p:gk_class} we prove that any GK structure $F\in\GK_{\pi,\alpha}$ in the $C^\infty$ path-connected component of $F_0$ can be obtained by the above deformation. In what follows, we denote this path-connected component by $\GK_{\poiss,\alpha}^0(F_0)$, or simply $\GK_{\poiss,\alpha}^0$ if the underlying base point is clear.  Thus the space $\GK_{\poiss,\alpha}^0$ will be referred to as  a \emph{generalized K\"ahler class} and we show in Lemma~\ref{l:commuting} that $\GK_{\poiss,\alpha}^0$ is an integrable (formal) submanifold of a distribution of vector fields on a formal Fr\'echet manifold.
Our point of view in this paper is that  $\GK_{\poiss,\alpha}^0$ is the right substitute of the K\"ahler class ${\KK}_{\alpha}$ in the symplectic type generalized K\"ahler setting.

 \subsection{Scalar curvature as moment map}

The second problem of extending the Calabi problem to the space $\GK_{\poiss,\alpha}^0$ stems from the fact that  in the generalized K\"ahler setting, there is no connection preserving all the structure, and thus there is no obvious way to define a scalar curvature.  Such a definition was first proposed by Boulanger~\cite{boulanger}, building on an unpublished work of Gauduchon~\cite{gauduchon-GIT}.  It uses the approach developed  by Fujiki~\cite{fujiki-GIT} and Donaldson~\cite{donaldson-GIT} in the K\"ahler case, who recast the Calabi program as a formal GIT problem of finding zeroes of a momentum map.  In their set-up, the ``manifold'' is the space ${\AK}_{\omega_0}$ (a formal Fr\'echet manifold) of all almost complex structures  $\Jj$ on $M$ compatible with a fixed symplectic form $\omega_0$, acted upon by the group ${\rm Ham}(M,\omega_0)$ of $\omega_0$-Hamiltonian diffeomorphisms.
It turns out that  ${\AK}_{\omega_0}$ admits a formal K\"ahler structure $({\bf \Omega}, {\bf \Jj})$ such that ${\rm Ham}(M, \omega_0)$ acts in a  Hamiltonian way with momentum map  identified,   at any integrable almost complex structure $\Jj \in {\AK}_{\omega_0}$,  with
\[
    \langle {\boldsymbol \mu}(\Jj), f \rangle = -\int_M {\Scal}_{(\omega_0, \Jj)}  f dV_{\omega_0}, \qquad f\in C^{\infty}(M,\R), \,\  \int_M f dV_{\omega_0}=0,
\]
where ${\Scal}_{(\omega_0, \Jj)}$ is the scalar curvature of the corresponding K\"ahler structure $(\omega_0, \Jj)$ and the Lie algebra of ${\rm Ham}(M, \omega_0)$ is identified with the vector space of zero mean smooth functions endowed with  the Poisson bracket with respect to $\omega_0$. Using Moser's lemma, ${\KK}_{[\omega_0]}$ can be mapped to a subset of  ${\AK}_{\omega_0}$.  As
observed in \cite{donaldson-GIT},  the image of  ${\KK}_{[\omega_0]}$  belongs to the  \emph{complexified orbit} of ${\rm Ham}(M, \omega_0)$  and is transversal to the orbits of ${\rm Ham}(M, \omega_0)$ inside this complexified orbit. Thus, the Calabi problem becomes the  familiar GIT problem of finding zeros  $\Jj$ of $\boldsymbol{\mu}$  in a given complexified orbit of the group action.  

In~\cite{boulanger}, the Donaldson--Fujiki setting is extended to the space ${\AGK}_{F_0}$ of \emph{almost generalized K\"ahler structures of symplectic type}  compatible with a fixed symplectic form  $F_0$, i.e.,the space of all almost complex structures $\Jj$ on $M$ such that $F_0$ \emph{tames} $\Jj$.    It is shown in \cite{boulanger} that ${\AGK}_{F_0}$ admits a formal K\"ahler structure such that ${\AK}_{F_0}$ is a formal K\"ahler submanifold;  when $(M, F_0, \T)$ is a compact toric manifold,  \cite{boulanger} also  proves that the equivariant part ${\rm Ham}^{\T}(M, F_0)$ acts in a Hamiltonian way on the invariant part ${\AGK}^{\T}_{F_0}$.  Furthermore,  in the $4$-dimensional toric case,  \cite{boulanger}  identifies the corresponding momentum map $\boldsymbol{\mu}(\Jj)$  for certain generalized K\"ahler structures $\Jj\in{\AGK}^{\T}_{F_0}$ with the smooth function
\begin{equation}\label{e:GK-scal-0}
{\Gscal}_{(F_0, \Jj)} :=  {\Scal}_g - \frac{1}{12}|db|^2_g + 2\Delta_g \Psi - |d\Psi|_g^2, \qquad \Psi = -\log\frac{dV_{F_0}}{dV_g}, \end{equation}
where ${\Scal}_g$ denotes the scalar curvature of $g$,  $\Delta_g = -d^*d$ is the corresponding Laplace operator, and $dV_{F_0}$ and $dV_g$ are  the volume forms associated to $F_0$ and $g$, respectively.  Surprisingly, this function precisely corresponds to the density of the string effective action for $H = db$ \cite{Polchinski}, and is thus a natural candidate for scalar curvature on physical grounds.

The approach of \cite{boulanger} was generalized by Goto~\cite{Gotomoment} who established without any restrictions on $(M, F_0)$ that ${\rm Ham}(M, F_0)$ acts in a Hamiltonian way on ${\AGK}_{F_0}$.  The associated moment map was taken to be the definition of scalar curvature, and Goto furthermore gave an expression for this curvature in terms of the underlying generalized complex structures $(\JJ, \II)$ on $TM\oplus T^*M$, which are associated to the biHermitian data $(g, b, I, J)$ via Gualtieri's map~\cite{Gualtieri-PhD}.  Despite this formula, there is no straightforward way to express Goto's scalar curvature in terms of the underlying biHermitian geometry, as it relies on local sections of the underlying generalized canonical bundles.  Our first main result resolves the apparent ambiguities between the different approaches to scalar curvature and confirms that the moment map of \cite{Gotomoment} in the general integrable case is given by~\eqref{e:GK-scal-0}:
\begin{thm}[Theorem~\ref{t:gscal_equiv}]
    Goto's moment map $\boldsymbol{\mu}$ computed at a generalized K\"ahler structure $J$ is given by the formula \eqref{e:GK-scal-0}.
\end{thm}
\noindent The proof relies on the local nondegenerate approximation technique introduced in \cite{AFSU}.  We thus refer to the function ${\Gscal}_{(F, \Jj)}$ associated to a symplectic type generalized K\"ahler structure $(F, \Jj)$ via \eqref{e:GK-scal-0} as the \emph{generalized K\"ahler scalar curvature} of  $(F, \Jj)$.

\subsection{Variational formulations and obstructions}

Based on the discussion of the previous subsection, Goto's moment map suggests natural generalizations of the extremal and csc metrics in K\"ahler geometry.  In particular, a generalized K\"ahler structure for which the vector field $\chi:=-F^{-1}\left(d{\Gscal}_{(F, \Jj)}\right)$ preserves $\Jj$ will be called \emph{extremal}, and a generalized K\"ahler structure for which  ${\Gscal}(F, \Jj) = \mathrm{const}$ will be called a \emph{cscGK} structure.  As an initial fundamental step in understanding the existence and uniqueness of extremal GK structures, we show in Lemma~\ref{l:complex-orbit}
 that the space $\GK_{\poiss,\alpha}^0$ is mapped via Moser's lemma into a formal ``complexified orbit'' for the action of $\Ham(M, F_0)$ on ${\AGK}_{F_0}$, being transversal to the $\Ham(M, F_0)$-orbits inside it. This is precisely as in the description \cite{donaldson-GIT} in the K\"ahler setting.  Thus we have the familiar GIT setup and we naturally arrive at the following:

\medskip

\noindent
 {\bf Generalized K\"ahler Calabi Program.}  {\it
 Express the obstructions to the existence of extremal generalized K\"ahler structures in $\GK^{0}_{\poiss, \alpha}$ in the form of a complex algebraic notion of stability of $(M, \Jj, \poiss_{\Jj}, \alpha)$.  When they exist show that such structures are unique up to the action of the connected component of the identity ${\Aut}_0(\Jj, \poiss_{\Jj})$ of the automorphism group ${\rm Aut}(\Jj, \poiss_{\Jj})$ of $(M, J, \poiss_{\Jj})$.}

\medskip

As a direct consequence of the GIT framework we provide an alternative characterization of extremal metrics as critical points of a Calabi functional $\mathbf{Ca}$ on $\GK_{\poiss,\alpha}$, see Definition~\ref{d:Calabi} and Proposition~\ref{p:Calabi}. Next, in Section~\ref{ss:Mabuchi} we introduce a proxy for the Mabuchi energy by defining a  closed 1-form $\boldsymbol{\tau}$ on $\GK_{\poiss,\alpha}$ which vanishes at cscGK structures.  A primitive of this 1-form would provide an analogue of Mabuchi's $K$-energy, which is a central object in the YTD conjecture concerning the K\"ahler Calabi Problem. 

Next we prove a structure theorem for the complex Lie group ${\Aut}(J, \poiss_J)$ of automorphisms of the complex Poisson manifold $(M, J, \poiss_J)$, similar to the well-known results by Matsushima~\cite{Matsushima} in the K\"ahler--Einstein case, Lichnerowicz~\cite{Lichnerowicz} in the constant scalar curvature K\"ahler case, and Calabi~\cite{calabi2} in the extremal K\"ahler case.  Such results are obtained by Goto~\cite{GotoLichne} for the Lie algebra $\mathfrak{g_0}$ of reduced automorphisms of one of the generalized complex structures associated to a symplectic type cscGK structure.  In the special case when the first Betti number of $M$ is zero, and $(F,J)$ is a small Poisson deformation of a K\"ahler structure, Goto proved that $\mathfrak g_0$ is the Lie algebra of holomorphic vector fields preserving the underlying Poisson tensor.
Compared to this, our result is obtained as a direct corollary of our formal GIT framework, which has the advantages of treating immediately the case of extremal GK structures, and being directly formulated in terms of the complex Poisson manifold $(M, J, \poiss_J)$ and its \emph{reduced} automorphism group $\mathrm{Aut}_{\mathrm{red}}(J,\poiss_J)$ (see Definitions~\ref{d:hred} and~\ref{d:gred})  without any extra technical assumptions, thus sharpening \cite[Theorems 6.5 \& 9.11]{GotoLichne}. 

\begin{thm}[Calabi-Lichnerowicz-Matsushima obstruction, Theorem~\ref{t:Calabi-Lichne-Matsushima}]
Suppose $F \in {\GK}_{\poiss, \alpha}$ is an extremal generalized K\"ahler structure with the holomorphic extremal vector field
\[
\chi = -F^{-1}(d\Gscal_{(F, J)}).
\]
Denote by ${\Aut}_{\rm red}(J, \poiss_J)^{\chi}\subset {\Aut}_{\rm red}(J, \poiss_J)$ the connected subgroup preserving $\chi$. Then the group 
\[
K_0:=({\Aut}_{\rm red}(J, \poiss_J)\cap \Ham(M,F))_0
\]
is a maximal compact connected subgroup of ${\Aut}_{\rm red}(J, \poiss_J)$ and ${\Aut}_{\rm red}(J, \poiss_J)^{\chi}=K_0^\C$.
In particular, $(F, J)$  must be invariant under a maximal real torus in ${\Aut}_{\rm red}(J, \poiss_J)$ containing the one-parameter subgroup $\exp(t\chi)$. If $F$ is cscGK,  then $\chi=0$ and  ${\Aut}_{\rm red}(J, \poiss_J)= K_0^{\C}$.
\end{thm}

In Section~\ref{ss:extremal} we provide an intrinsic description of the extremal vector field of a potential extremal GK structure in $\GK_{\poiss,\alpha}$. It provides an efficient obstruction for the existence of cscGK metrics in a given GK class and is instrumental in our treatment of the toric case. In the K\"ahler case this description was obtained by Futaki-Mabuchi~\cite{Futaki-Mabuchi}.

\begin{thm}[{Extremal vector field, Theorem \ref{t:GK-extremal}}]
    Given a torus $\T\subset \mathrm{Aut}_{\mathrm{red}}(J,\poiss_J)$ let $\GK_{\poiss,\alpha}^\T$ be the space of $\T$-invariant generalized K\"ahler structures. Then, for any  $F_0\in \GK_{\poiss,\alpha}^\T$, $\T\subset \mathrm{Ham}(M,F_0)$ and we denote by 
    $\Pi_{F_0}\colon C^\infty(M,\R)\mapsto C^\infty(M,\R)$ the $L^2(M,dV_F)$ projection onto the space of $F_0$-Hamiltonians of $\T$.
    Moreover, the vector field
    \begin{equation}\label{e:extremal_vf_i}
        \chi=-F^{-1}_0(d\Pi_{F_0}(\Gscal_{(F_0,J)}))
    \end{equation}
    is independent of the choice of a $\T$-invariant symplectic form $F_0\in \GK_{\poiss,\alpha}^\T$.  If, furthermore, $\T$ is a maximal torus in $\mathrm{Aut}_{\mathrm{red}}(J,\poiss_J)$,  then $F\in\GK_{\poiss,\alpha}^\T$ is an extremal structure if and only if
    \[
        \Gscal_{(F,J)}-\Pi_F(\Gscal_{(F,J)})=0.
    \]
    In particular,  the underlying extremal vector field $-F^{-1}(d\Gscal_{(F,J)})$ is necessarily given by~\eqref{e:extremal_vf_i}.
\end{thm}

Using the Mabuchi 1-form $\boldsymbol{\tau}$ on $\GK^0_{\poiss, \alpha}$,  we define a generalized K\"ahler analogue of the classical obstruction for the existence of cscK metrics in K\"ahler geometry~--- the Futaki character of the Lie algebra $\hred(J,\poiss_J)$ of $\mathrm{Aut}_{\mathrm{red}}(J,\poiss_J)$.

\begin{thm}[Futaki character, {Theorem~\ref{t:futaki-character}}]
    Let $F\in\GK_{\poiss,\alpha}$ be a generalized K\"ahler structure on $(M,J,\poiss_J)$.  Define a linear map  $\mathcal F_{(F,J)}\colon \hred(J,\poiss_J)\to \R$ by
    \begin{equation}\label{e:futaki-character_intro}
        \mathcal F_{(F,J)}(X)=\int_M \psi\Gscal_{(F,J)}dV_F ,\quad X=F^{-1}(d\phi+Id\psi).
    \end{equation}
    Then $\mathcal F_{(F,J)}$ is independent of $F\in\GK_{\poiss,\alpha}^0$ and vanishes on the commutator $[\hred(J,\poiss_J),\hred(J,\poiss_J)]$.
    In particular $\mathcal F_{(F,J)}$ is a character of $\hred(J,\poiss_J)$ and is identically zero if $\GK_{\poiss,\alpha}^0$ admits a cscGK metric.
\end{thm}

\begin{rmk} A conceptually distinct notion of Futaki invariant has previously appeared in the context of generalized geometry, providing obstructions to the existence of solutions of the Hull-Strominger system \cite{GFM}.
\end{rmk}

\subsection{Formal metric structure and uniqueness}

In the K\"ahler setting, the GIT framework is naturally completed by a formal Riemannian metric, known as the Mabuchi--Semmes--Donaldson metric \cite{Mabuchi, Semmes, donaldson-GIT}.  This metric formally gives ${\mathcal K}_{\ga}$ the structure of a symmetric space of nonpositive curvature.  Furthermore, the Mabuchi $K$-energy, whose critical points are the constant scalar curvature K\"ahler metrics, is convex along geodesics of ${\mathcal K}_{\ga}$, a key point leading to uniqueness of csc K\"ahler metrics \cite{Chen, CT, BB}.  We show that this formal picture extends to the symplectic type GK setting.

To begin, the tangent space ${\bf T}_F\left(\GK_{\poiss,\alpha}\right)$ of $\GK_{\poiss,\alpha}$ at a point $F$ can be identified with
\[
{\bf T}_F\left(\GK_{\poiss,\alpha}^0\right)  \simeq C^{\infty}(M,\R)/\R.
\]
We define a formal Riemannian metric at a point $F\in \GK_{\poiss,\alpha}$ by
\[
	\big\llangle \phi_1, \phi_2 \big\rrangle_F :=
	\int_M \phi_1 \phi_2 \frac{F^n}{n!},
	\qquad \phi_i \in \CCF,\ i=1,2,
\]
where $\CCF$ is the space of functions with zero average against the symplectic volume form $dV_F$.  We first give a generalization of the formal symmetric space structure established by Mabuchi in the K\"ahler case:
\begin{thm}[{Theorem~\ref{p:Mabuchi-curvature}}] The curvature tensor $\mathcal R$ of $\llangle \cdot, \cdot \rrangle$  at a point $F\in\GK_{\poiss,\alpha}$ is given by
\[\left({\mathcal R}_{[{\phi_1}], [{\phi_2}]} [{\phi_3}]\right)_F =-\Big\{\big\{\phi_1,\phi_2\big\}_F, \phi_3\Big\}_F,  \]
where $\{\cdot, \cdot \}_{F}$ denotes the Poisson bracket of functions with respect to the symplectic form $F$ and $[\phi_i]\in C^{\infty}(M,\R)/\R, \, i=1,2,3$ are tangent vectors in ${\bf T}_{F}\left(\GK_{\poiss,\alpha}^0\right)$. In particular, the sectional curvature of $\big\llangle \cdot, \cdot \big\rrangle$ is everywhere nonpositive.
\end{thm}

The formal Riemannian metric naturally determines a notion of geodesic.  We observe that a special class of geodesics are generated by the flow $\exp(-tJY)$ of any  $F_0$-Hamiltonian Killing vector field $Y$ of $(F_0, \Jj)$ (cf.\,Proposition~\ref{p:killing-geodesics}).  Analogous to the geodesic convexity of the Mabuchi energy in the K\"ahler setting, we show that the Mabuchi $1$-form $\boldsymbol{\tau}$ increases along a smooth geodesic in $\GK_{\poiss,\alpha}$, and is identically zero along a geodesic precisely when the latter is induced by a Hamiltonian Killing field as above.  This leads to a formal proof of uniqueness:

\begin{cor}[{Corollary \ref{c:conditional-uniqueness}}]
Suppose $F_0, F_1 \in \GK_{\poiss,\alpha}$ are cscGK structures connected by a smooth geodesic $F_t$. Then there exists $Y\in\hred(J,\poiss_J)$ such that $F_t=\Phi_t^* F_0$, where $\Phi_t=\exp(- tJY)\in \mathrm{Aut}_{\mathrm{red}}(J,\poiss_J)$.
\end{cor}

\noindent Turning this into a genuine proof of uniqueness requires developing the theory of geodesics in $\mathcal {GK}_{\pi,\ga}$.  Such curves do not readily reduce to a Monge--Amp\`ere equation, as exploited in the construction of (weak) geodesics in ~\cite{Chen}.  However, there is a generalization of the Semmes construction \cite{Semmes} expressing the geodesic equation as a prescribed volume form problem for a natural family of symplectic forms on an augmented spacetime track (cf.\,Remark \ref{r:Semmes}).

\subsection{The toric case}

In the final portion of the paper we consider the case  when $(M, J, \T^{\C})$ is a smooth projective toric variety and $\poiss_J$ a $\T^{\C}$-invariant Poisson tensor.  In this case, the $\T$-invariant symplectic type generalized K\"ahler structures have been studied in \cite{ASU,W2}  where they were described in terms of a smooth convex function defined on the interior of  the Delzant polytope of $(M, F, \T)$ and a bivector $\poiss_J \in \Wedge^2({\rm Lie}(\T^{\C}))$. This is analogous to the Abreu--Guillemin description~\cite{Abreu, guillemin}  of toric K\"ahler structures. Building on  this theory, we are able to obtain an almost complete picture for the generalized K\"ahler Calabi problem on a toric variety:

\begin{thm} Suppose $(M, J, \poiss_\Jj, \T^{\C})$ is a toric projective holomorphic Poisson manifold and $F_0$ a $\T$-invariant $\poiss_\Jj$-compatible symplectic type generalized K\"ahler structure  in the deRham class $\alpha$. Then:

\begin{enumerate}
    \item[(a)] (Corollary~\ref{c:T-reduction}) Any extremal generalized K\"ahler structure in $\mathcal {GK}_{\pi,\ga}$ is isometric to a $\T$-invariant such structure.
    \item[(b)] (Proposition~\ref{p:contractible} and Theorem~\ref{t:uniqueness-toric}) The space of $\T$-invariant generalized K\"ahler structures compatible with $\poiss$ and $\alpha=[F_0]$ is path connected: $({\GK}_{\poiss,\alpha}^\T)^{0} = {\GK}_{\poiss,\alpha}^\T$.  Furthermore, any two extremal generalized K\"ahler structures in $\GK_{\poiss,\alpha}^0$ are isometric by an element of $\Aut_0(\Jj, \poiss_{\Jj})$.
    \item[(c)] (Theorem~\ref{c:easy}) If $\GK_{\poiss,\alpha}^\T$ admits an extremal generalized K\"ahler structure then the Delzant polytope of $\alpha$ is uniform relative K-stable.
    \item[(d)] (Theorems~\ref{t:hard} and~\ref{t:LS}) If the Delzant polytope of $(M, J, \alpha, \T^{\C})$ is uniform relative K-stable, and $\poiss_{\Jj} \in \Wedge^{2}\left({\rm Lie}(\T^{C})\right)$, then there exists an $\varepsilon >0$, such that ${\GK}_{t\poiss, \alpha}$ admits an extremal generalized K\"ahler metric  for all $t \in \R, \, |t|< \varepsilon$.
\end{enumerate}
\end{thm}

The statement (a) follows formally from the Matsushima-type obstruction of Theorem \ref{t:Calabi-Lichne-Matsushima}.  The path-connectedness in (b) follows from the Abreu--Guillemin description of toric generalized K\"ahler structures obtained in \cite{boulanger, W2} and reviewed in Section~\ref{s:toric} below.  The uniqueness relies on establishing the smooth geodesic convexity of $\GK_{\poiss,\alpha}^{\T}$.  In the toric K\"ahler case, this due to  Guan~\cite{Guan} who showed that the geodesics of ${\KK}^{\T}_{\alpha}$ become linear segments in the Abreu--Guillemin description.  In Proposition~\ref{l:geodesic-toric} we show this is still the case in the generalized K\"ahler case, but only for a special class of Poisson tensors $\poiss_{\Jj}$.  To obtain the full claim, we argue in Corollary~\ref{c:A-independence} that the existence problem for an extremal generalized K\"ahler structure can be reduced to this subclass of Poisson tensors.  The statement in (c) makes a direct connection with the YTD conjecture which is now established for toric varieties due to the work of  X.X.\,Chen and J.\,Cheng~\cite{CC} in the constant scalar curvature case and its extension by W.\,He~\cite{He} to the extremal case. The condition of  uniform relative K-stability  of the polytope was introduced by Donaldson~\cite{donaldsonJDG-02} and is now known~\cite{hisamoto} that it is  equivalent to the uniform relative K-stability  on toric test configurations of  the corresponding polarized toric variety. Our proof of (c) adapts the original arguments~\cite{CLS} in the toric K\"ahler case  to the generalized K\"ahler context,  by replacing  the hessian of the K\"ahler potential with a suitable positive-definite  smooth symmetric-matrix valued function on the Delzant polytope, coming from the Abreu--Guillemin description. We also notice that (c) provides many constructible examples of toric varieties which do not admit extremal generalized K\"ahler structures. The final point (d)  follows from the existence in $\alpha$ of a toric extremal K\"ahler metric~\cite{CC, He} and  an  adaptation of the LeBrun--Simanca openness  result~\cite{LS} where instead of varying the K\"ahler class we vary the Poisson tensor.  This  yields new examples of cscGK structures even on ${\bf CP}^2$, as the existence results obtained in \cite{boulanger, Gotomoment, GotoLichne} apply only to a special class of toric Poisson tensors.

%
%
\bigskip
\noindent
\textbf{Acknowledgements.}   The first named author was supported in part  by a Connect Talent fellowship of the Region des Pays de la Loire (France) and an NSERC Discovery Grant.  The second named author was supported by a Simons Fellowship and by the NSF via DMS-2203536. The authors are grateful to R.\,Goto for his interest and remarks.

\section{Generalized K\"ahler structures of symplectic type} \label{s:defns}
\subsection{Conventions}

For a Hermitian manifold $(M, g, J)$ we consider
the Riemannian metric $g$ as a field of  isomorphisms
\[g: TM \to T^*M, \qquad X \to g(X, \cdot), \qquad X\in TM, \]
with inverse $g^{-1} : T^*M \to TM$.  This leads to a definition of inner product on $T^*M$ and a compatible almost complex structure
\[ J:= g J g^{-1}. \]
Notice that if $J^*$ is the induced almost complex  structure  on $T^*M$ by $(J^* \alpha)(X):= \alpha (JX)$, then
\[ J = - J^*.\]
More generally, we will implicitly consider a field of bilinear forms $B\in C^{\infty}(T^*M \otimes T^*M)$ on $TM$ (such as a Riemannian metric $g$ or a symplectic form $F$ on $M$) as a field of endomorphisms
\[ B: TM \to T^*M, \qquad X\to B(X, \cdot).\]
 In case $B$ is nondegenerate we denote by $B^{-1}$  the inverse endomorphism. Similarly, a section of $\Pi \in C^{\infty}(TM\otimes TM)$ will be viewed as a field of endomorphisms
 \[ \Pi: T^*M \to TM, \qquad \alpha \to \Pi(\alpha, \cdot).\]   We denote by $\langle \cdot, \cdot \rangle$ the natural pairing between $TM$ and $T^*M$. If $B\in C^{\infty}(T^*M \otimes T^*M)$ and $\Pi \in C^{\infty}(TM\otimes TM)$ is skew (i.e.,a field of bivectors), then we have
\[ \left\langle \Pi, B \right\rangle = -\tr \left(\Pi\circ B \right). \]
For a $2$-form $\psi$, we will  let
\[\psi^{[k]} := \frac{1}{k!}\psi^{k}.\]
Our convention for the trace  $\tr_F(\psi)$ of a $2$-form $\psi$ with respect to  a symplectic form $F$ is then 
\begin{equation*} \tr_F (\psi) := \frac{\psi \wedge F^{[n-1]}}{F^{[n]}} = \tfrac{1}{2} \tr\left( F^{-1} \psi \right). \end{equation*}  
The volume form of a symplectic structure is
\[dV_F := F^{[n]}, \]
and the Riemannian volume form of the Hermitian structure $(g, J)$ is
\[ dV_g = \omega_J ^{[n]}, \qquad \omega_J := gJ.\]
Furthermore, given a symplectic form $F$ we let $\CCF$denote the space of \textit{normalized} smooth functions with zero average relative to $dV_F$:
\[
\CCF=\left\{\phi\in C^\infty(M,\R)\ |\ \int_M\phi dV_F=0\right\}.
\]

\subsection{Symplectic type generalized K\"ahler structures} \label{NGKdef}

In this subsection we recall the biHermitian formulation of generalized K\"ahler
geometry, and the basic properties of the associated Poisson structures.

\begin{defn} \label{GKdef} Given a smooth manifold $M$ endowed with a closed $3$-form $H_0$, we say that $(g, b, I, J)$
is a \emph{generalized K\"ahler structure (GK structure)} if $I$ and $J$ are
integrable complex structures, $b$ is a $2$-form, $g$ is Riemannian metric compatible with both $I$ and $J$, and
furthermore
\begin{align*}
 d^c_\Ii \gw_\Ii = H_0 + db = - d^c_\Jj \gw_\Jj.
\end{align*}
Associated to this structure we define the tensor
\begin{align*}
 \poiss := \tfrac{1}{2} [I,J]g^{-1} \in \Wedge^2 (TM).
\end{align*}
By \cite{HitchinPoisson, AGG}, $\poiss$ is the real part of holomorphic $(2,0)$-bivectors
with respect to  $I$ and $J$, and we define
\begin{gather} \label{e:PoissonJ}
\begin{split}
\pi_J := \pi - \i J \pi, \qquad \pi_I := \pi - \i I \pi.
\end{split}
\end{gather}
\end{defn}

\begin{defn}\label{symplectic typeGKdef} Given a complex manifold $(M, \Jj)$, suppose there exists a symplectic form $F$ such that
\begin{equation*}
 g := -(F\Jj)^{\rm sym},  \qquad  \Ii :=  -F^{-1} \Jj^* F \end{equation*}
define respectively a Riemannian metric $g$  and an integrable almost complex structure $\Ii$.  Then, setting
\begin{equation*}
H_0=0, \qquad b := -(FJ)^{\rm skew}, \end{equation*}
the data $(g, b, I, J)$ satisfies \eqref{GKdef} and is referred to as \emph{symplectic type generalized K\"ahler structure} on $(M, \Jj)$.  Since $(F,J)$ algebraically determine the tuple $(g,b,I,J)$, by abuse of notation we will often refer to $(F,J)$ as a symplectic type generalized K\"ahler structure.
\end{defn}

\begin{rmk}\label{r:LA} Elementary linear algebra yields that for a symplectic type GK structure one has
\begin{equation}\label{F-g-relation}
F(I+J)= -2g, \end{equation}
so that $\det(I + J)\neq 0$ on $M$.  Notice that $\Ii$ coincides with $\Jj$ (and is then automatically integrable) precisely when $F$ is of type $(1,1)$ with respect to $\Jj$, i.e.,defines a K\"ahler structure $(g, \Jj)$.
\end{rmk}

From the point of view of generalized geometry \cite{Gualtieri-PhD,Gualtieri-CMP} adapted in~\cite{Gotomoment,GotoLichne}, symplectic type generalized K\"ahler structures correspond to pairs of commuting \emph{generalized complex structures}. To make this connection more precise, we explicitly express the underlying (almost) generalized complex structures in terms of $(F,J)$.  The following proposition immediately follows from~\cite[\S 6.4]{Gualtieri-PhD}.

\begin{prop}\label{p:GK-interpretation}
    Let $(M,J)$ be an almost complex manifold. For a nondegenerate 2-form $F$ denote as above $I=-F^{-1}J^*F$, $g=-(FJ)^\mathrm{sym}$, $b=-(FJ)^\mathrm{skew}$. Then $(F,J)$ gives rise to a pair of commuting generalized almost complex structures $\JJ,\II\in\End(TM\oplus T^*M)$
    \begin{equation}\label{e:GK-interpretation}
        \qquad {\JJ} =  \JJ_{P, Q}=\left(\begin{array}{cc} P & Q F^{-1}\\ -FQ & -P^* \end{array}\right), \qquad  {\II} ={\mathbb J}_F=\left(\begin{array}{cc} 0 & - F^{-1} \\ F & 0 \end{array}\right),
    \end{equation}
    where $P, Q \in C^{\infty}(M, {\rm End}(TM))$ are given by
    \begin{equation*}
        P:= -2(I+J)^{-1}, \qquad  Q:= (J-I)(I+J)^{-1}.
    \end{equation*}
    Furthermore, $\II$ is integrable if and only if $F$ is closed, and $(\JJ,\II)$ are both integrable if and only $F$ is symplectic, and $(J,I)$ are integrable. Conversely any generalized almost complex structure $\JJ$ commuting with $\JJ_F$ is necessarily given by~\eqref{e:GK-interpretation} for some almost complex structures $J$.
\end{prop}

We note that fundamental works of Goto and Gualtieri \cite{Goto-JDG, Goto-AM, Gualtieri-Hamiltonian} show that on a compact \emph{K\"ahler} manifold endowed with a holomorphic Poisson tensor $\pi_J$, there always exist symplectic type GK structures with Poisson tensor $\pi_J$.  Furthermore there are natural deformation spaces, reviewed below, which naturally fix the holomorphic Poisson geometry.  We thus define the space of GK structures on such a background.

\begin{defn} A symplectic type generalized K\"ahler structure $F$ on  $(M, \Jj)$ is called $\poiss_{\Jj}$-\emph{compatible} if the holomorphic Poisson tensor associated to $(g, b, I, J)$ via (\ref{e:PoissonJ}) equals $\poiss_{\Jj}$. As $\poiss= \Re e(\poiss_{\Jj})$ determines $\poiss_{\Jj}$ on $(M, J)$, we let ${\GK}_{\poiss}$ denote
the space  of $\poiss_{\Jj}$-compatible symplectic type GK structures.  The deRham class $\alpha:=[F]\in H^2(M, \mathbb R)$ of a generalized K\"ahler structure in ${\GK}_{\poiss}$ will be called a \emph{compatible deRham class}. We denote by ${\GK}_{\poiss, \alpha}$ the space of $\poiss_{\Jj}$-compatible symplectic type generalized K\"ahler structures within a fixed $\poiss_{\Jj}$-compatible deRham class $\alpha$.
\end{defn}

\subsection{Algebraic identities on a symplectic type generalized K\"ahler manifold}\label{s:appendix}

Here we collect some elementary identities for a symplectic type GK manifold $(M, g, I, J, F)$ which are useful in the sequel.  Recall the basic expressions
\begin{equation}\label{a:basic}
F = -2g(I+J)^{-1}, \qquad F^{-1} = -\tfrac{1}{2}\left(I + J\right) g^{-1},
\end{equation}

We compute for any $1$-forms $\alpha, \beta$
\begin{equation}\label{a:deep1}
\tr_F(\alpha\wedge \beta)  = \tfrac{1}{2}\tr\left(F^{-1} (\alpha \wedge \beta)\right) =\tfrac{1}{2}\bigl\langle (I+ J)\alpha, \beta \big\rangle_g, \end{equation}
where  $\langle\cdot, \cdot \rangle_g$ denotes the inner product on $T^*M$ induced by $g$ and have used the fact that $I, J$ are skew with respect to $g$.
In particular, using \eqref{a:basic}, the $g$-orthogonality of $I$ and $J$, and the identity
\begin{equation}\label{a:identity}
I(I+J) = (I+ J) J,
\end{equation}
we get
\begin{equation}\label{a:Fvolumeidentity}
\tr_F(\alpha\wedge J\beta)  = - \tr_F(I\alpha\wedge \beta),
\end{equation}
or, equivalently,
\begin{equation}\label{f:Fvolumeidentity}
\alpha\wedge J\beta \wedge F^{[n-1]} = -I\alpha \wedge \beta \wedge F^{[n-1]}.
\end{equation}
Recall that the Poisson tensor $\poiss$ is given by
\[\poiss = \half [I, J]g^{-1} =  \half (I-J)(I+J)g^{-1}\]
Using \eqref{a:deep1},  the fact that $I$ and $J$ are skew, and \eqref{a:Fvolumeidentity}, we compute for any $1$-forms $\alpha, \beta$
\begin{equation}\label{a:poisson}
\begin{split}
\Big\langle \nonhalf \poiss, \alpha\wedge \beta \Big\rangle &=- \frac{1}{4}\tr\Big(\poiss \circ (\alpha\wedge \beta)\Big)\\
& = -\frac{1}{4}\Big( \big\langle \alpha, (I-J)(I+J)\beta\big\rangle_g-\big\langle (I-J)(I+J)\alpha, \beta\big\rangle_g\Big)\\
&=-\tfrac{1}{2}\big\langle (I+J) (I-J)\alpha, \beta \big\rangle =\tr_F\big((J-I) \alpha \wedge \beta \big).
\end{split}
\end{equation}
We next establish the following identity, which holds for any $1$-forms $\alpha_1, \alpha_2, \beta_1, \beta_2$
\begin{equation}\label{a:deep2}
\begin{split}
\tr \Big(F^{-1}(\beta_1 \wedge \beta_2) F^{-1} (\alpha_1\wedge \alpha_2)\Big) & = \tfrac{1}{2}\Big\langle (I+J)\alpha_1\wedge (I+J)\alpha_2, \beta_1 \wedge \beta_2 \Big\rangle_g \\
&= \tfrac{1}{2}\Big\langle \alpha_1\wedge \alpha_2,  (I+J) \beta_1 \wedge (I+J)\beta_2 \Big\rangle_g.
\end{split}
\end{equation}
To this end, we use \eqref{a:basic}, \eqref{a:identity} and the fact that $I$ and $J$ are skew with respect to $g$ to compute
\[
\begin{split}
& 4\tr \Big(F^{-1}(\beta_1 \wedge \beta_2) \circ F^{-1} (\alpha_1\wedge \alpha_2)\Big) \\
&=  \tr\Big(\big(\beta_1\otimes (I+J)\beta_2^{\sharp} - \beta_2 \otimes (I+J)\beta_1^{\sharp}\big)\circ \big(\alpha_1\otimes (I+J)\alpha_2^{\sharp} - \alpha_2\otimes (I+J)\alpha_1^{\sharp}\big)\Big)\\
&=\Big(\big\langle(I+J)\alpha_2, \beta_1\big\rangle_g\big\langle \alpha_1, (I+J)\beta_2\big\rangle_g - \big\langle \beta_1, (I+J)\alpha_1\big\rangle_g\big\langle \alpha_2, (I+J) \beta_2\big\rangle_g \\
& \hspace{0.7cm} -\big\langle (I+J)\beta_1, \alpha_1\big\rangle_g\big\langle \beta_2, (I+J)\alpha_2\big\rangle_g + \big\langle \beta_2, (I+J)\alpha_1\big\rangle_g \big\langle (I+J)\beta_1, \alpha_2\big\rangle_g\Big)\\
&=2\Big( \big\langle (I+J)\alpha_1, \beta_1\big\rangle_g \big\langle (I+J)\alpha_2, \beta_2\big\rangle_g -  \big\langle (I+J)\alpha_1, \beta_2\big\rangle_g \big\langle (I+J)\alpha_2, \beta_1\big\rangle_g\Big).
\end{split}
\]
A useful ramification of \eqref{a:deep2} is the formula
\begin{equation}\label{a:deep4}
\begin{split}
& \tr \Big(F^{-1}(\beta \wedge I \beta) \circ F^{-1} (\alpha\wedge J\alpha)\Big) =\tfrac{1}{2} \Big(\big\langle (I+J)\alpha, \beta\big\rangle_g^2  +  \big\langle (I+J)\alpha, I\beta\big\rangle_g^2 \Big).
\end{split}
\end{equation}
Notice finally that \eqref{a:deep2} and \eqref{a:basic} also yield
\begin{equation}\label{a:deep5}
\tr \Big(F^{-1}(I\beta_1 \wedge I\beta_2) F^{-1} (\alpha_1\wedge \alpha_2)\Big) =\tr \Big(F^{-1}(\beta_1 \wedge \beta_2) F^{-1} (J\alpha_1\wedge J\alpha_2)\Big). \end{equation}

\subsection{An integration by parts formula} We will frequently use in this paper the following basic integration by parts identity:

\begin{lemma} For $(F, J)$ a symplectic type generalized
K\"ahler structure one has
\begin{equation}\label{e:by-parts}
\begin{split}
\int_M \phi \tr_F (d I d \psi) F^{[n]} =&\ \int_M \psi \tr_F (d J d \phi) F^{[n]}.
\end{split}
\end{equation}
\begin{proof}
We use (\ref{f:Fvolumeidentity}) to compute 
\begin{equation*}
\begin{split}
\int_M \phi \tr_F (d I d \psi) F^{[n]} =&\ \int_M \phi d I d \psi \wedge F^{[n-1]} = - \int_M d \phi \wedge I d \psi \wedge F^{[n-1]}\\
=&\ \int_M J d \phi \wedge d \psi \wedge F^{[n-1]} = \int_M \psi d J d \phi \wedge F^{[n-1]}\\
=&\ \int_M \psi \tr_F (d J d \phi) F^{[n]}.
\end{split}
\end{equation*}
\end{proof}
\end{lemma}

\subsection{Generalized K\"ahler Hodge theory}

In this subsection we briefly review the Hodge theory on a generalized K\"ahler manifold focusing on the symplectic type case. We refer the reader to~\cite{Gualtieri-Hodge} and~\cite{CavalcantiBook} for a more detailed exposition.

We start with some linear-algebraic preliminaries. Elements of the \emph{generalized tangent bundle} $TM\oplus T^*M$ naturally act on the space of differential forms $\Wedge^*(M)$ via contraction and exterior product. This action extends to an action of the Clifford algebra $\mathrm{Cl}(TM\oplus T^*M,\ \la\cdot,\cdot\ra)$, where $\la\cdot,\cdot\ra$ is the usual pairing between $TM$ and $T^*M$. This  turns $\Wedge^*(M)$ into a Clifford module. The action on $TM\oplus T^*M$ of the commuting operators $\JJ$ and $\II$ from~\eqref{e:GK-interpretation} extends to an action on $\Wedge^*(M)$,  via the representation of the $\mathfrak{spin}(n,n)$ Lie algebra. We thus get a decomposition of $\Wedge^*(M)\otimes \C$ into the eigenspaces of $\JJ$ and $\II$:
\begin{equation}\label{e:eigenspace_forms}
	\Wedge^*(M)\otimes \C=\bigoplus_{-n\leq p,q\leq n} U^{p,q},
\end{equation}
where $U^{p,q}$ is the $(\i p,\i q)$ eigenspace of $(\JJ,\II)$.  One can show that $U^{p,q}=0$ unless $|p-q|,|p+q|\leq n$, and  $\bar{U^{p,q}}=U^{n-p,n-q}$. For a form $\xi\in\Wedge^*(M)\otimes \C$ we will denote by $\xi^{p,q}$ its component in $U^{p,q}$. More concretely, we can describe the spaces $U^{p,q}$ as follows. Let
\begin{equation}\label{e:eqigenspace_tangent}
	(TM\oplus T^*M)\otimes \C= L_+\oplus L_-\oplus \bar{L_+}\oplus \bar {L_-}
\end{equation}
be the decomposition of the complexified generalized tangent bundle into the eigenspaces of $(\JJ,\II)$:
\[
\JJ\big|_{L_+\oplus\bar{L_-}}=\i \mathrm{Id},\quad
\II\big|_{L_+\oplus L_-}=\i\mathrm{Id}.
\]
Specifically, using (\ref{e:GK-interpretation}) we have
\begin{equation}\label{e:L_pm}
	\begin{split}
		L_+&=\{v-\i F(v,\cdot)\ |\ v\in T^{1,0}_IM\}\\
		L_-&=\{v-\i F(v,\cdot)\ |\ v\in T^{0,1}_JM\}\\
	\end{split}
\end{equation}
Then one can compute (see~\cite{CavalcantiBook})
\begin{equation}\label{e:upq}
U^{0,n}=\mathrm{span}(e^{\i F}), \qquad U^{a-b,n-a-b}=\Wedge^{a}\bar{L_-}\cdot (\Wedge^b\bar{L_+}\cdot U^{0,n}).
\end{equation}
Since $FI=-J^*F$, and $v\cdot e^{\i F}=\i F(v,\cdot)\wedge e^{\i F}$, it follows from~\eqref{e:upq} that
\begin{equation}\label{e:u_(1,n-1)}
\begin{split}
U^{1,n-1}&=\{\xi\wedge e^{\i F}\ |\ \xi\in \Wedge^{0,1}_I(M)\}
\\
U^{-1,n-1}&=\{\xi\wedge e^{\i F}\ |\ \xi\in \Wedge^{1,0}_J(M)\}.
\end{split}
\end{equation}

Now let us review the underlying differential complex of $(U^{*,*},d)$. The integrability assumptions of the generalized K\"ahler structure $(\JJ,\II)$ imply that $d : U^{*,*}\to U^{*,*}$ has four components of bidegrees $(\pm 1,\pm 1)$. We denote them respectively
\[
\delta_+: U^{p,q}\mapsto U^{p+1,q-1},\quad \delta_-: U^{p,q}\mapsto U^{p-1,q-1},\quad \bar{\delta_+}: U^{p,q}\mapsto U^{p-1,q+1},\quad \bar{\delta_-}: U^{p,q}\mapsto U^{p+1,q+1}
\]
so that $d=\delta_++\delta_-+\bar{\delta_+}+\bar{\delta_-}$.  Furthermore, there is a natural pairing $\Wedge^*(M) \times \Wedge^*(M)\to \Wedge^{2n}(M)$:
\[
(\xi,\eta):=[\xi\wedge\sigma(\eta)]_{\mathrm{top}},
\]
where $[\cdot]_{\mathrm{top}}$ is the top degree component, and $\sigma: \Wedge^*(M)\mapsto\Wedge^*(M)$ is the Clifford involution $\sigma(dx^1\wedge\dots\wedge dx^k):=dx^k\wedge\dots\wedge dx^1$. It is shown in~\cite{Gualtieri-Hodge, CavalcantiBook} that the real \textit{star} operator
\[
\star\colon U^{*,*}\to U^{*,*},\quad \star\big|_{U^{p,q}}=(\i)^{p+q}\mathrm{Id}
\]
defines a positive definite pairing $\Wedge^*(M)\times\Wedge^*(M)\to\R$
\[
G(\xi,\eta):=\int_M (\xi,\star\eta).
\]
Using this pairing and integration by parts we define the adjoint linear first order differential operator $d^*: \Wedge^*(M)\mapsto \Wedge^*(M)$
\[
G(d\xi,\eta)=G(\xi,d^*\eta).
\]
Note that $d^*$ differs from the Riemannian co-differential. The following result is the crucial fact of Hodge theory on $M$:

\begin{thm}[{Hodge identities, see \cite{Gualtieri-Hodge,CavalcantiBook}}]
	On a compact generalized K\"ahler manifold there is a decomposition $d^*=\delta_+^*+\delta_-^*+\bar{\delta_+}^*+\bar{\delta_-}^*$, where $\delta_\pm^*$ and $\bar{\delta_\pm}^*$ are the adjoint operators of $\delta_\pm$ and $\bar{\delta_\pm}$. These operators satisfy the following identities:
    \begin{equation*}
    \begin{split}
    \delta_+^*=\bar{\delta_+},\quad\delta_-^*=-\bar{\delta_-}\\
    4\Delta=\Delta_{\delta_+} = \Delta_{\delta_-}=\Delta_{\bar{\delta_+}}=\Delta_{\bar{\delta_-}},
    \end{split}
    \end{equation*}
where $\Delta=dd^*+d^*d$, $\Delta_{\delta_+}=\delta_+\delta_+^*+\delta_+^*\delta_+$, and the remaining Laplacians are defined analogously.
\end{thm}

This theorem has several important consequences. First, the cohomology of $U^{*,*}$ with respect to any of the differentials $d,\delta_+,\delta_-$ are naturally isomorphic to each other and to the space of \textit{harmonic forms}
\[
    \mathcal{H}^{p,q}=\{\xi\in U^{p,q}\ |\ \Delta\xi=0\}
\]
Interestingly, a form $\xi\in U^{p,q}$ of pure type is harmonic if and only if it is closed.
Furthermore, every form $\xi\in U^{p,q}$ has a decomposition
\[
    \xi=\xi_h+\Delta \eta,
\]
where $\xi_h\in \mathcal{H}^{p,q}$ is the harmonic part of $\xi$ and $\eta\in U^{p,q}$. Further we will need a more explicit description of $\mathcal H^{\pm 1,n-1}$. Using the identification~\eqref{e:u_(1,n-1)} we conclude that there are natural isomorphisms given by wedging with $e^{-\i F}$:
\begin{equation}\label{Hodge-isomorphism}
    \begin{split}
    \mathcal{H}^{1,n-1}&\simeq \{\xi\in \Wedge^{0,1}_I(M)\ |\ d\xi=0\}\\
    \mathcal{H}^{-1,n-1}&\simeq \{\xi\in\Wedge^{1,0}_J(M)\ |\ d\xi=0\}\\
    \mathcal{H}^{1,n-1}\oplus \mathcal{H}^{-1,n-1}&\simeq H^1(M,\C).
    \end{split}
\end{equation}

\begin{lemma}\label{l:h1_isomorphism}
    Let $(F,J)$ be a symplectic type generalized K\"ahler structure on a compact manifold~$M$. Given a 1-form $\xi\in \Wedge^1(M)\otimes\C$ we decompose it as
    \[
        \xi=\xi_++\xi_-,
        \quad
        \xi_+=I(I+J)^{-1}\xi+\i (I+J)^{-1}\xi,\quad \xi_-=J(I+J)^{-1}\xi-\i(I+J)^{-1}\xi,
    \]
    so that
    \[
        \xi\wedge e^{\i F}=\xi_+\wedge e^{\i F}+\xi_-\wedge e^{\i F}
    \]
    is the decomposition of $\xi\wedge e^{\i F}\in U^{1,n-1}\oplus U^{-1,n-1}$ by type. Then the maps
    \[
        \xi\mapsto (\xi\wedge e^{\i F})^{1,n-1}_h,\quad \xi\mapsto(\xi\wedge e^{\i F})^{-1,n-1}_h
    \]
    induce isomorphisms
    \[
        i_+: H^1(M,\R)\mapsto \mathcal H^{1,n-1},\quad
        i_-: H^1(M,\R)\mapsto \mathcal H^{-1,n-1}.
    \]
    Furthermore, any $[\xi]\in H^1(M,\R)$ can be represented by a closed, $d^c_J$-closed form $\xi_-$.
\end{lemma}
\begin{proof}
    First, we observe that the map $\xi\mapsto (\xi\wedge e^{\i F})^{1,n-1}_h$ yields a well-defined map on $H^1(M,\R)$. Indeed, for an exact 1-form $df$ we have $(df\wedge e^{\i F})^{1,n-1}_h=\left(\delta_+(f\wedge e^{\i F})\right)_h=0$. Next we observe that the map
    \[
        \Wedge^1(M)\otimes\C\mapsto U^{1,n-1}\oplus U^{-1,n-1},\quad \xi\mapsto \xi\wedge e^{\i F}
    \]
    establishes an isomorphism between $H^1(M,\C)$ and $H^{1,n-1}(U^{*,*},d)\oplus H^{1,n-1}(U^{*,*},d)\simeq \mathcal H^{1,n-1}\oplus \mathcal H^{-1,n-1}$, thus
    \[
        2\dim_\R H^1(M,\R)=\dim_\R \mathcal H^{1,n-1}+\dim_\R\mathcal H^{-1,n-1}.
    \]
    Therefore it suffices to prove that the maps $i_\pm$ are injective. Indeed, consider a closed real form $\xi\in\Wedge^{1}(M)$. Assume that $(\xi\wedge e^{\i F})^{1,n-1}_h=0$, so that the $\delta_+$-closed form $(\xi\wedge e^{\i F})^{1,n-1}$ is $\delta_+$-exact: $(\xi\wedge e^{\i F})^{1,n-1}=\delta_+(f e^{\i F})$ for some complex-valued function $f=\phi+\i\psi\in C^\infty(M,\C)$. Using the definition of $\xi_+$, and the explicit expression for $\delta_+$ we can rewrite it as
    \[
	\xi=d\psi+Jd\phi.
    \]
    Since $\xi$ is closed, we conclude that $dd^c_J\phi=0$. On a compact manifold $M$ the latter implies that $\phi=\mathrm{const}$ by the maximum principle. Thus $\xi=d\psi$ represents the trivial cohomology class, so the map $i_+$ is injective. Similarly $i_-$ is also injective, and by the dimension count both must by isomorphisms.

    To prove that last claim, we observe that the map
    \[
        \mathcal{H}^{-1,n-1}\to H^1(M,\R),\quad \xi_-\wedge e^{\i F}\mapsto [\Re(\xi_-)]
    \]
    between two vector spaces of equal dimensions is also injective for similar reasons as for $i_-$. Indeed, if $\Re(\xi_-)=d\phi$ is exact, then $\phi$ is $dd_J^c$-closed, which is only possible when $\phi$ is a constant. Therefore, the above map is an isomorphism, and every element in $H^1(M,\R)$ has a closed, $d^c_J$-closed representative.
\end{proof}

A further key consequence of the Hodge identities is a $\delta_+\delta_-$-lemma for $(U^{p,q},d)$:

\begin{lemma}[{$\delta_+\delta_-$ and $\delta_+\bar{\delta_-}$-lemma \cite{Gualtieri-Hodge,CavalcantiBook}}]
    On a compact generalized K\"ahler manifold
    \[
        \begin{split}
        \mathrm{Im}(\delta_+)\cap\Ker(\delta_-)
        &=\Ker(\delta_+)\cap\mathrm{Im}(\delta_-)
        =\mathrm{Im}(\delta_+\delta_-)\\
        \mathrm{Im}(\delta_+)\cap\Ker(\bar\delta_-)
        &=\Ker(\delta_+)\cap\mathrm{Im}(\bar\delta_-)
        =\mathrm{Im}(\delta_+\bar\delta_-).
        \end{split}
    \]
 \end{lemma}

\begin{cor}\label{c:holo_closed}
    On a compact generalized K\"ahler manifold of symplectic type $(M, F, J)$ any holomorphic $p$-form is  closed.
\end{cor}
\begin{proof}
    First, we observe that from the identities~\eqref{e:L_pm} and~\eqref{e:upq}  we get an isomorphism
    \begin{equation}\label{e:lambda_{p,0}}
        \bigoplus_p\Wedge^{p,0}_J(M)\simeq \bigoplus_p U^{-p,n-p}\quad \xi\mapsto \xi\wedge e^{\i F}.
    \end{equation}
    Let $\xi\in \Wedge_J^{p,0}(M)$ be a holomorphic form. Then $d\xi=\i\del_J\xi\in \Wedge_J^{p+1,0}(M)$ and we claim that this form vanishes. Indeed, from isomorphism~\eqref{e:lambda_{p,0}} we conclude that
    \[
        \xi\wedge e^{\i F},d\xi\wedge e^{\i F}\in  \bigoplus_p U^{-p,n-p}
    \]
    which implies that $\xi\wedge e^{\i F}$ is $\delta_+$-closed. Consider the form $\delta_-(\xi\wedge e^{\i F})$. This form is $\delta_-$-exact and $\delta_+$-closed, therefore by the $\delta_+\delta_-$-lemma it has to be in the image of
    \[
        \delta_+\delta_-: \bigoplus_p U^{-p,n-p-2}\to \bigoplus_p U^{-p,n-p}.
    \]
    As the former space is trivial,  $\delta_-  (\xi\wedge e^{\i F})$ is zero. Similarly, applying the $\delta_+\bar{\delta_-}$-lemma we conclude that $\bar{\delta_-}(\xi\wedge e^{\i F})=0$. Since for dimensional reasons we also have $\bar{\delta_+}(\xi\wedge e^{\i F})=0$, it follows that $\xi\wedge e^{\i F}$ is closed which is equivalent to $\xi$ being closed.
\end{proof}

\subsection{Hamiltonian symplectic type generalized K\"ahler  deformations}\label{s:HFD}

Let $(M,\omega_0, \Jj)$ be a compact K\"ahler manifold.  A key feature of K\"ahler geometry is the possibility to deform $\omega_0$ by smooth functions as follows: given $\varphi\in C^\infty(M,\R)/\R$ we can define a new (1,1)-form
\[
\omega_\varphi:=\omega_0+dd^c_\Jj\varphi,
\]
and as long as $\omega_\varphi$ is positive definite $(M,\omega_\varphi, \Jj)$ is again a K\"ahler manifold. This procedure is reversible,  due to  the $dd^c_{\Jj}$-lemma: a  K\"ahler metric $\omega$ on $(M, J)$ is of the form $ \omega= \omega_\varphi$  for some smooth function $\varphi$ (uniquely defined  up to an additive constant) iff  $\omega \in \alpha= [\omega_0] \in H^{2}(M,\R)$.
 In particular, the space ${\KK}_\alpha$  of all K\"ahler metrics within a given deRham class $\alpha$  has a structure of a  Fr\'echet manifold modeled on the vector space  $C^\infty(M,\R)/\R$:
for any $\phi \in C^\infty(M,\R)/\R$,  there is an infinitesimal deformation $\omega_t$ of $\omega$ in  ${\KK}_\alpha$, such that
\[
\dt \omega_t=dd^c_\Jj \phi.
\]
There is a similar construction in generalized K\"ahler geometry, which was first presented in the context of generalized K\"ahler structures with nondegenerate Poisson tensor on 4-dimensional manifolds in~\cite[\S 4.2]{AGG},  where it is attributed to Joyce (see also~\cite{hi-07}). Gualtieri in \cite[\S 7]{gu-10} defined, more generally,  Hamiltonian deformations of  symplectic type generalized K\"ahler manifolds whereas \cite{GiS} defined a version adapted to the general biHermitian case.  We recall the construction of \cite{gu-10}:

\begin{thm}[{\cite{gu-10}}]\label{t:cstr:flow}
    Let $(M,\Jj, F_0)$ be a compact symplectic type generalized K\"ahler manifold with the second complex structure $I_0$ and real Poisson tensor $\poiss$. Let $\phi_t\in C^\infty(M,\R), \, t\in (-\varepsilon, \varepsilon)$  be a one-parameter family of smooth functions on $M$,   $X_{\phi_t}:=\sj\nonhalf\poiss(d\phi_t)$  the time dependent $(\sj \nonhalf\poiss)$-Hamiltonian vector field,  and $\Phi_t,\, \Phi_0={\rm Id}$ the corresponding isotopy of diffeomorphisms.  Define
\begin{equation} \label{f:flowconst}
	\begin{split}
F_{t}&:= F_0 +\int_0^t \left(dd^c_{I_s}\phi_s\right) ds, \qquad I_s = \Phi_s \cdot I_0 := (\Phi_{s})_* I_0 (\Phi_{s})_*^{-1}.
	\end{split}
	\end{equation}
Then $F_{t}$ defines a one-parameter family of  symplectic type generalized K\"ahler structures on $(M, J)$ as long as $(F_t)_J^{1,1}>0$. Furthermore,
$I_t = - F_t^{-1} J^*  F_t= \Phi_t \cdot I_0$.
\end{thm}
\begin{defn}[Hamiltonian deformations of symplectic type generalized K\"ahler structures]\label{d:flow_construction} The family of symplectic type generalized K\"ahler structures $F_t$  on $(M, \Jj, \poiss_{\Jj})$ given by Theorem~\ref{t:cstr:flow} for  some $\phi_t\in C^\infty(M,\R)$ will be referred to as a \emph{Hamiltonian deformation} of $F_0$.
Under this deformation
\begin{equation}\label{e:flow_evolution}
	\begin{split}
	\dt \Jj&=0, \qquad \dt \Ii =-\mathcal L_{X_{\phi}} \Ii= \sj\nonhalf\poiss \circ (dd^c_{\Ii}\phi),  \qquad \dt F= dd^c_\Ii\phi, \qquad  \dt \poiss  = 0.
	\end{split}
	\end{equation}
It follows that the Hamiltonian deformations are $\poiss_{\Jj}$-compatible and preserve the given symplectic generalized K\"ahler class $\alpha=[F_0]$, i.e., $F_t \in {\GK}_{\poiss, \alpha}$.
\end{defn}

\subsection{Formal manifold structure}\label{s:AGK-poiss}

According to \cite{BGZ,gu-10}, the space ${\GK}_{\poiss}$ is parametrized by the closed $2$-forms $F$ which tame $\Jj$ and satisfy the algebraic (zero order) identity
\begin{equation}\label{algebraic}
F \Jj + \Jj^* F \sj \nonhalf F\circ \poiss \circ F=0.\end{equation}
In fact, the algebraic condition \eqref{algebraic} is equivalent to
\begin{equation}\label{AGK}
\poiss =\half [I, J]g^{-1}, \qquad g:=-(F\Jj)^{\rm sym}, \qquad \Ii:= - F^{-1}\Jj^* F.
\end{equation}
Assuming this condition, when $F$ is closed, $\Ii$ defined as above is automatically \emph{integrable} and $F$ defines a symplectic type generalized K\"ahler structure.  This observation allows us to consider a weaker space of structures satisfying the algebraic constraints but not the integrability condition, which is useful from an analytic perspective.

\begin{defn} \label{d:AGKpi} Given a complex manifold $(M, J)$ with real Poisson tensor $\pi$, let
\begin{align*}
\AGK_{\poiss} = \{ F \in \Wedge^2 (M) \ |\ F J + J^* F - F \circ \pi \circ F = 0 \}.
\end{align*}
\end{defn}

\begin{rmk} By differentiating the defining relation along a path $F_s$ we see that
\[ 0 = \dot{F}\Jj +  \Jj^*{\dot F} \sj \nonhalf \left(\dot{F}  \circ \poiss \circ F  + F \circ \poiss \circ \dot{F}\right) = 2 (\dot F \Ii)^{(2,0)+ (0, 2)}_{\Ii},\]
where we have used that $I - J = - \pi \circ F$.  In particular, the tangent space of ${\AGK}_{\poiss}$ at $F$ is:
\begin{equation}\label{tangent-identification}
 {\bf T}_{F} \left({\AGK}_{\poiss}\right) = \Wedge^{1,1}_{\Ii} (M), \qquad  \Ii:= - F^{-1}\Jj^* F.
 \end{equation}
 \end{rmk}

\begin{defn} \label{d:fundamental-field} For any smooth (time independent) function $\phi \in C^{\infty}(M,\R)$, we define a vector field ${\bf X}_{\phi}$ on ${\AGK}_{\poiss}$ given by
\begin{equation}\label{e:fundamental-field}
{\bf X}_{\phi} (F) :=  (dd^c_\Ii \phi)^{1,1}_{\Ii}, \qquad \Ii := -F^{-1} \Jj^* F, \end{equation}
and  call it a \emph{fundamental vector field} associated to $\phi$.  We denote by ${\bf D}$ the distribution generated by the fundamental vector fields on ${\AGK}_{\poiss}$.
\end{defn}
Given a $\phi \in C^{\infty}(M,\R)$,  by \eqref{e:flow_evolution}, for any $F \in {\GK}_{\poiss}$  the Hamiltonian flow construction with the (time independent) function $\phi$ and starting at $F$ produces a smooth path $F_t$   representing ${\bf X}_{\phi}(F)$; it follows that  $F_t$ is an integral curve of ${\bf X}_\phi$.
We have the following consequence of this fact:
\begin{lemma}\label{l:commuting} The subset ${\GK}_{\poiss}\subset {\AGK}_{\poiss}$ is a (formal) integrable submanifold of ${\bf D}$. In particular, it is  ${\bf D}$-invariant and,  when restricted to ${\GK}_{\poiss}$, ${\bf D}$ is involutive. More precisely, we have on ${\GK}_{\poiss}$,
\[ [{\bf X}_{\phi}, {\bf X}_{\psi}] = - {\bf X}_{\{\phi, \psi\}_{\poiss}}, \]
where $\{\phi, \psi\}_{\poiss}:= \poiss(d\phi, d\psi)$ is the $\poiss$-Poisson bracket.
\end{lemma}
\begin{proof} The first part follows from the facts that ${\GK}_{\poiss}$ is ${\bf D}$-invariant (as we have already checked) and at each point $F\in {\GK}_{\poiss}$, by \eqref{e:flow_evolution}, ${\bf D}_F$ coincides with the sub-space of  tangent vectors in ${\bf T}_F\left({\AGK}_{\poiss}\right)$ generated by smooth paths in ${\GK}_{\poiss}$, i.e.
\[ {\bf T}_F\left({\GK}_{\poiss}\right) = {\bf D}_F \subset {\bf T}_F\left({\AGK}_{\poiss}\right).\]
We now compute the vector field bracket of ${\bf X}_{\phi}$ and ${\bf X}_{\psi}$. At each point $F$, we let $F_t$ denote the Hamiltonian deformation with respect to $\phi$ and $F_s$ the Hamiltonian deformation  with respect to $\psi$. We also denote by $\Phi_t$ and $\Psi_s$ the flows of the Poisson vector fields $X_{\phi}:= \sj\nonhalf \poiss(d\phi)$ and $X_{\psi}:= \sj\nonhalf\poiss(d\psi)$ on $M$. Noting that $F_t$ and $F_s$ are respectively integral curves of ${\bf X}_\phi$  and ${\bf X}_\psi$, we compute
\[ \begin{split}
 [{\bf X}_{\phi}, {\bf X}_{\psi}](F_0) &= \left. \frac{d}{dt} \right |_{t=0} \left(\frac{d}{ds}_{|_s=0} \Big(\big(F_t\big)_s\Big)_{-t}\right) \\
 &=\left. \frac{d}{dt} \right |_{t=0} \frac{d}{ds}_{|_s=0}\left(\int_0^t \Big(dd^c_{\Phi_r \cdot \Ii_0} \phi\Big) dr + \int_0^s \Big(dd^c_{\Psi_p\Phi_t \cdot \Ii_0} \psi \Big)dp - \int_0^t \Big(dd^c_{\Phi_{-q}\Psi_s \Phi_t\cdot \Ii_0} \phi \Big) dq \right) \\
        &= d\left(-\Big({\mathcal L}_{X_{\phi}} \Ii^*_0\Big)(d\psi) + \Big({\mathcal L}_{X_{\psi}} \Ii^*_0\Big)(d\phi)\right) \\
        &= d\left(\imath_{X_{\psi}}\Big(dd^c_{\Ii_0} \phi\Big) - \imath_{X_{\phi}}\Big(dd^c_{\Ii_0} \psi\Big)\right) \\
        & = - \left({\mathcal L}_{X_{\phi}}\Big(dd^c_{\Ii_0} \psi\Big) - {\mathcal L}_{X_{\psi}}\Big(dd^c_{\Ii_0} \phi\Big)\right) \\
        & = -\twice dd_{\Ii_0}^c\{\phi, \psi\}_{\poiss} + d\left(\big({\mathcal L}_{X_{\phi}}I^*_0\big)(d\psi) - \big({\mathcal L}_{X_{\psi}}{\Ii^*_0}\big)(d\phi)\right),
        \end{split}\]
        where we have used formulae \eqref{e:flow_evolution} to pass from the third line to the forth. Comparing the third line with the last line in the above equalities, we conclude
\[ [{\bf X}_{\phi}, {\bf X}_{\psi}](F_0)  = \sj\nonhalf dd^c_{\Ii_0} \{\phi, \psi\}_{\poiss} = -{\bf X}_{\{\phi, \psi\}_{\poiss}} (F_0).\]
\end{proof}

Note that at this point we have defined the natural class of Hamiltonian deformations in Definition \ref{d:flow_construction}, but on the other hand it is natural to consider the space of GK structures in $\mathcal GK_{\pi,\ga}$ with fixed cohomological background data.  We next show that these constructions are the same, in particular showing that any smooth path $F_t\in \GK_{\poiss,\alpha}$ is given by the construction of Theorem~\ref{t:cstr:flow}.

\begin{prop}\label{p:gk_class}
	Let $F_t\in \GK_{\poiss,\alpha}$ be a smooth path of generalized K\"ahler structures of symplectic type.  Denote the underlying second complex structure by $I_t$. Then \[
	\frac{dF_t}{dt}=dd^c_{I_t}\phi_t
	\]
	for some $\phi_t\in C^\infty(M,\R)$.
\end{prop}
\begin{proof}
    Since $[F_t]\in H^2(M,\R)$ is fixed, by~\eqref{tangent-identification} we conclude that
    \[
        \dot F:=\frac{dF_t}{dt}=d\xi\in\Wedge^{1,1}_{I_t}(M),\quad \xi\in\Wedge^1(M)
    \]
    is an exact form of type $(1,1)$ with respect to $I_t$. The differential form $\xi\wedge e^{\i F}$ belongs to the space $U^{1,n-1}\oplus U^{-1,n-1}$ so that,  by virtue of Lemma~\ref{l:h1_isomorphism},  we can decompose $\xi=\xi_++\xi_-$ with  $\xi_+\in \Wedge_I^{0,1}(M),$  $\xi_-\in\Wedge_J^{1,0}(M)$,  and $\xi_{\pm}\wedge e^{\i F}\in U^{\pm 1,n-1}$.

    We claim that the $U^{-2,n-2}$-component of $d(\xi\wedge e^{\i F})=\dot F\wedge e^{\i F}$ vanishes. Indeed, for any two elements
    \[
        v-\i F(v,\cdot),w-\i F(w,\cdot)\in L_+,\quad v,w\in T_{I_t}^{1,0}M
    \]
    we have the vanishing of the Clifford action
    \[
        (v-\i F(v,\cdot))\cdot (w-\i F(w,\cdot))\cdot (\dot F\wedge e^{\i F})=2\dot F(w,v)\wedge e^{\i F}=0
    \]
    since $\dot F$ is of  type $(1,1)$ with respect to $I_t$. This implies that the element $\xi_-\wedge e^{\i F}$
    is $\delta_-$-closed. By Lemma~\ref{l:h1_isomorphism} the natural map $i_-: H^{1}(M,\R)\to \mathcal{H}^{-1,n-1}$, $[\xi]\mapsto (\xi\wedge e^{\i F})^{-1,n-1}_h$ is an isomorphism, so we can find a closed 1-form $\eta$ such that $((\xi-\eta)\wedge e^{\i F})^{-1,n-1}_h=0$. Therefore, by the $\delta_-$-Hodge decomposition we conclude that
    \[
        ((\xi-\eta)\wedge e^{\i F})^{-1,n-1}=\delta_-((\phi+\i \psi) e^{\i F}),
    \]
    where $\phi+\i \psi$ is some complex-valued function. This means that $(d\phi+\i d\psi)_-=(\xi-\eta)_-$ which yields $(\xi-\eta)=\Re(d\phi+\i d\psi+I_t(d\phi+\i\psi))$ so that
    \[
    \dot F=d\xi=d(\xi-\eta)=dd^c_{I_t}\phi
    \]
    as claimed.
\end{proof}

\begin{rmk}\label{d:O}
    On a compact manifold $dd_I^c\phi$ is nonzero unless $\phi\in C^\infty(M,\R)$ is a constant. Thus in view of Proposition~\ref{p:gk_class} we have an identification of the tangent space to $\GK_{\poiss,\alpha}$
    \[
    {\bf T}_{F}\left(\GK_{\poiss,\alpha}\right)\simeq C^\infty(M,\R)/\R
    \]
    at any $F\in \GK_{\poiss,\alpha}$.
\end{rmk}

\section{The generalized K\"ahler scalar curvature as a momentum map}

Central to our understanding of the YTD conjecture in the K\"ahler setting is the GIT formulation due to Fujiki-Donaldson \cite{fujiki-GIT, donaldson-GIT}.  The fundamental point in this framework is that the moment map for the space of Hamiltonian diffeomorphisms of a fixed K\"ahler form acting on complex structures is given by the scalar curvature.  In attempting to extend this circle of ideas to generalized K\"ahler geometry one is faced with the subtle issue that there is not an obvious choice of scalar curvature due to the lack of a connection preserving all structure.  Rather, the works of Boulanger-Goto \cite{boulanger,Gotomoment} take the point of view of defining a natural Hamiltonian action on the space of GK structures and use its momentum map to \emph{define} scalar curvature.  In the work of Goto \cite{Gotoscal} this does lead to a general definition of scalar curvature for GK structures, although one expressed implicitly in terms of local defining spinors determining the underlying generalized complex structures.  Furthermore, ideas from mathematical physics and generalized geometry (cf.\,\cite{GRFbook, Polchinski, Str-Scal}) suggest an explicit definition of scalar curvature for GK structures, and in special cases Boulanger and Goto have shown that this definition agrees with what arises from the moment map framework.  In this section we will review the moment map construction, and furthermore close this circle of ideas by showing that the a priori different definitions of scalar curvature agree for all symplectic type GK structures.  The key input is the nondegenerate perturbation technique introduced in \cite{AFSU}.

\subsection{Scalar curvature of GK structures}

We begin by explicitly stating our definition of scalar curvature of GK structures of symplectic type in terms of the biHermitian data.  It was already noted in Boulanger \cite{boulanger} that the moment map takes this explicit form in the toric setting, and this was inspirational for our work.

\begin{defn}\label{d:GK-scal} Let $(g, b, I, J, F)$ be a symplectic type generalized K\"ahler structure on $M$.  The \emph{generalized scalar curvature} is defined by
\begin{equation}\label{e:GK-scal}
    {\Gscal}_{(F, J)} :=  {\Scal}_g - \frac{1}{12}|db|^2_g + 2 \Delta_g \Psi - |d\Psi|^2_g , \qquad \Psi = -\log\frac{dV_{F}}{dV_g},
\end{equation}
where $\Scal_{g}$ is the scalar curvature of $g$, $\Delta_g= -d^*d$ is the Laplacian, and $dV_F=F^{[n]}, \, dV_g = \omega_I^{[n]}=\omega_J^{[n]}$ are respectively the symplectic and Riemannian volume forms.
\end{defn}

In~\cite{Gotomoment} Goto constructed a smooth function
\[
\Gscal_{(F,J)}^{\mathrm{Goto}}\in C^\infty(M,\R)
\]
associated to a symplectic form $F$ and
a generalized almost complex structure $\JJ\in\End(TM\oplus T^*M)$ which via the Gualtieri map~\eqref{e:GK-interpretation} is equivalent to an \emph{almost} complex structure $J$ tamed by $F$. Goto's motivation for defining this function originates from the formal momentum map picture, and we discuss it later. For now we record several important features of $\Gscal_{(F,J)}^{\mathrm{Goto}}$, which will be used throughout the paper.
\begin{enumerate}
    \item $\Gscal_{(F,J)}^{\mathrm{Goto}}$ at a point $x\in M$ depends algebraically on the second jets of $F_x$ and $J_x$, is additive with respect to the Cartesian product, and vanishes on flat (linear) structures
    \item A generalized \emph{almost} complex structure $\J$ given by~\eqref{e:GK-interpretation} determines a \emph{canonical}  complex line bundle $K_{\J}$ 
    and there exists a representative $\rho_{(F,J)}\in \Wedge^2(M)$ of $2\pi c_1(K_{\J}^{-1})$ such that
    \begin{equation}\label{e:gscal_trace}
        \Gscal_{(F,J)}^{\mathrm{Goto}}
        =\frac{\rho_{(F,J)}\wedge F^{[n-1]}}{F^{[n]}},
    \end{equation}
    see~\cite[Def.\,5.3 \& Prop.\,5.5]{Gotomoment}.
    \item If $(F,J)$ defines a genuine generalized K\"ahler structure, i.e., both complex structures $J$ and $I=-F^{-1}J^*F$ are integrable and $I-J$ is invertible, then
    \begin{equation}\label{e:gscal_nondeg}
        \Gscal_{(F,J)}^{\mathrm{Goto}}=
        \tr_{F}(d(Fg^{-1}d\Phi)),\quad \Phi=\frac{1}{2}\log{\det(I-J)}-\frac{1}{2}\log{\det(I+J)},
    \end{equation}
 see~\cite[Prop.\,10.4]{Gotomoment}.
\end{enumerate}

\begin{thm}\label{t:gscal_equiv}
    If $(F,J)$ is a symplectic type generalized K\"ahler structure on $M$, then the generalized scalar curvature given by~\eqref{e:GK-scal} coincides with Goto's scalar curvature:
    \[
        \Gscal_{(F,J)}
        =\Gscal_{(F,J)}^{\mathrm{Goto}}.
    \]
\end{thm}
\begin{proof}
    Our goal is to give an explicit biHermitian expression for $\Gscal_{(F,J)}^{\mathrm{Goto}}$. We use the result of~\cite[\S 3]{AFSU}, where we proved that possibly after taking the product with a flat factor $(\C, g_{\mathrm{flat}})$ any GK structure locally on an open dense set can be approximated in any $C^{k,\alpha}$ norm by a GK structure with invertible $I\pm J$. Let $(F^l,J^l)$ be such a sequence of locally defined GK structures converging to a given $(F_0,J_0)$ in the $C^{2,\alpha}$ norm. Since $\Gscal^{\mathrm{Goto}}$ depends continuously on the second jet of the underlying biHermitian data, we have $\Gscal_{(F^l,J^l)}^{\mathrm{Goto}}\mapsto \Gscal_{(F_0,J_0)}^{\mathrm{Goto}}$ as $l\to \infty$.
    
    Let $(F,J)$ be any member of the sequence $(F^l,J^l)$. Using~\eqref{e:gscal_nondeg} and the identities $\theta_I=I\delta^g I$, $\theta_J=J\delta^g J$ for the Lee forms we compute:
    \begin{equation}\label{e:gscal_pf1}
    \begin{split}
        \tr_{F}(d(Fg^{-1}d\Phi))&=\sum_{i=1}^{2n}\big\la(I+J)\nabla_{e_i}((I+J)^{-1}g^{-1}d\Phi),e_i\big\ra_g
        \\
        &=\Delta \Phi-\sum_{i=1}^{2n}\la (\nabla_{e_i}(I+J))(I+J)^{-1} g^{-1}d\Phi,e_i\ra_g
        \\
        &=\Delta \Phi-\la (I+J)^{-1}g^{-1}d\Phi, 	\delta^gI+\delta^gJ\ra\\&=
        \Delta\Phi + \la (I+J)^{-1}g^{-1}d\Phi, 	I\theta_I+J\theta_J\ra\\&=\Delta\Phi -\frac{1}{2}\la d\log\det(I+J),d\Phi\ra_g,
    \end{split}
    \end{equation}
    where in the last step we applied the identity $I\theta_I+J\theta_J=\frac{1}{2}(I+J)d\log\det(I+J)$, which holds on any symplectic type generalized K\"ahler manifold see~\cite[Prop.\,4.3]{AFSU}.

	It remains to combine the identity~\eqref{e:gscal_pf1} with a formula
 	\begin{equation}\label{e:gscal_pf_formula}
		\Scal_g-\frac{1}{12}|db|^2_g=-\frac{1}{2}\Delta_g(\log\det(I+J)+\log\det(I-J))+\frac{1}{4}\la d\log\det(I+J),d\log\det(I-J)\ra_g,
    \end{equation}
    see~\cite[Lemma\,4.5]{AFSU}, to conclude that
    \begin{equation}\label{e:gscal_pf2}
        \begin{split}
        \Gscal_{(F,J)}^{\mathrm{Goto}}=
        &\tr_{F}(d(Fg^{-1}d\Phi))
            =\Scal_g-\frac{1}{12}|db|^2_g+2\Delta_g\Psi-|d\Psi|^2_g,
        \end{split}
    \end{equation}
    where $\Psi=\frac{1}{2}\log\det(I+J)=-\log\frac{dV_F}{dV_g}$.
    By~\eqref{e:gscal_pf2} along the sequence $(F^l,J^l)$ we have $\Gscal_{(F^l,J^l)}^{\mathrm{Goto}}=\Scal_{g^l}-\frac{1}{12}|db^l|^2_{g^l}+2\Delta_{g^l}\Psi^l-|d\Psi^l|^2_{g^l}$. Passing to the $C^{2,\alpha}$ limit $(F^l,J^l)\to (F_0, J_0)$ we conclude that
    \[
    \Gscal_{(F_0,J_0)}^{\mathrm{Goto}}=\Scal_{g_0}-\frac{1}{12}|db_0|^2_{g_0}+2\Delta_{g_0}\Psi_0-|d \Psi_0|^2_{g_0}=\Gscal_{(F_0,J_0)},\quad \Psi_0=-\log\frac{dV_{F_0}}{dV_{g_0}}
    \]
    in general, without assuming that $I_0-J_0$ is invertible.
\end{proof}
In the view of the above proposition, we will denote Goto's scalar curvature simply by $\Gscal_{(F,J)}$ bearing in mind that it is given by~\eqref{e:GK-scal} if $(F,J)$ defines a generalized K\"ahler structure.
\begin{rmk} 
Goto \cite{Gotoscal} further introduced a generalized scalar curvature associated to an arbitrary generalized K\"ahler structure $(\JJ,\II)$ on $(M,H_0)$, $H_0\in\Wedge^3(M)$ and a volume form $d\mu_f=e^{-f}dV_g$.  We will show in forthcoming work that similarly to Proposition~\ref{t:gscal_equiv} this quantity can be computed for $(g, b, I, J)$ by an analogous formula
\[
\Gscal^f_{(\JJ,\II)}=\Scal_g-\frac{1}{12}|H|^2_g+2\Delta_g f-|df|^2_g,
\]
where $H = H_0 + db$.
\end{rmk}

\subsection{Geometry of the space of almost generalized K\"ahler structures}

On a compact smooth oriented manifold $M$, we  consider the space  ${\AC}$  of (oriented) almost complex structures $J$, endowed with the Fr\'echet topology of smooth sections of $T^*M\otimes TM$. Thus the tangent space of ${\AC}$   at a point $J$ can be identified with the vector space of smooth sections  $\dot J$ of $T^*M\otimes TM$ satisfying $J\dot J = - \dot J J$.  Similarly, on a given compact symplectic manifold $(M, F)$, we consider the subspace  ${\AGK}_F \subset {\AC}$ of  $F$-tamed almost complex structures, i.e.
\[{\AGK}_F:= \left\{ J \in {\AC} \, \Big| \, F_p(v, Jv) >0,\, \,  \forall \, 0\neq v \in T_pM, \, \forall \, p\in M \right\}. \]
As the $F$-taming condition is open in the $C^{\infty}$ topology, ${\AGK}_F$ is an open subspace of ${\AC}$, with the same tangent space at $J\in {\AGK}_F$. For each element $J \in {\AGK}_F$ , we denote by $I:= - F^{-1} J^* F^{-1}\in {\AGK}_F$ its $F$-conjugate.  Writing $-FJ = g + b$ where $g$ is the symmetric part and $b$ is the skew-part of $-FJ$, the taming condition means  that $g$ is a Riemannian metric on $M$ and $b$ is a $2$-form.  It is easy to check that both $J$ and its $F$-conjugate $I$ are $g$-orthogonal, so we have a quadruple $(g, b, I, J)$ which  gives rise to a symplectic type generalized K\"ahler structure if both $I$ and $J$ are integrable. Irrespective of the integrability of $I$ and $J$, we will refer to $(g, b, I, J)$
as  an \emph{almost generalized K\"ahler structure}. For any such structure, we still have
\[ \det(I+J) \neq 0, \qquad  F= -2g (I+J)^{-1}, \qquad b= g (J-I)(I+J)^{-1} = -g(I+J)^{-1} (J-I).\]
Letting $P:= (I+J)^{-1}$ and $Q:=(J-I) (I+J)^{-1}$,  we can  still consider the endomorphisms $\JJ=\JJ_{P,Q}$ and  $ \II={\mathbb J}_F$ of $TM \oplus T^*M$,  introduced by \eqref{e:GK-interpretation};  $(\JJ, \II)$ give rise to a commuting pair of almost complex structures on $TM \oplus T^*M$ (called \emph{generalized almost complex} structures),   and a  positive definite bilinear form $-\left\langle \JJ \II \cdot, \cdot \right\rangle$ (where $\langle \cdot, \cdot \rangle$ is the natural symmetric product  on $TM \oplus T^*M$). Unlike $\II=\JJ_{F}$, the endomorphism $\JJ_{P,Q}$ will not be in general an \emph{integrable} generalized almost complex structure, but $(\JJ, \II)$ still gives rise to what we will refer to as a  $F$-compatible almost generalized K\"ahler structure.  We introduce the space
\[ \AGK_{\J_F}:=\left\{\JJ \in {\rm O}(TM\oplus T^*M, \la\cdot,\cdot\ra)\ |\ \J^2=-\mathrm{Id},\ \J_F\J=\J_F\J, \ -\la\J\J_F\cdot,\cdot\ra>0\right\}. \]
By Proposition~\ref{p:GK-interpretation} the correspondence $J\mapsto \JJ=\JJ_{P,Q}$ given by~\eqref{e:GK-interpretation}  provides an isomorphism of formal Fr\'echet manifolds
\[
\Gu : \AGK_F\to \AGK_{\J_F}.
\]
These Fr\'echet manifolds have the following tangent spaces at $F$ (respectively $\JJ_F$)
\[
\begin{split}
    {\bf T}_J \left(\AGK_F\right) &=\{\dot{J}\in \End(TM)\ |\ J\dot{J}=-\dot{J}J \},\\
    {\bf T}_{\JJ}\left(\AGK_{\J_F}\right)&=\{\dot{\J}\in \mathfrak{o}(TM\oplus T^*M,\la\cdot,\cdot\ra)\ |\ \J\dot{\J}=-\dot{\J}\J,\  [\J_F,\dot \J]=0\}.
\end{split}
\]
Both ${\bf T}_J\left(\AGK_F\right)$ and ${\bf T}_\JJ\left(\AGK_{\J_F}\right)$ admit a formal almost complex structure,  via the left multiplications by $J$ and  $\J$, respectively. It turns out that $\Gu$ preserves these almost complex structures:

\begin{lemma}\label{l:complex}
The isomorphism of Fr\'echet manifolds $\Gu: \AGK_F\mapsto \AGK_{\J_F}$ preserves the underlying almost complex structures. Namely, if $\dot J$ is a tangent vector at $J\in {\AGK}_F$, then
\[
\J\, d\Gu(\dot J)=d\Gu(J\dot J).
\]
\end{lemma}

\begin{proof} Along the map
    \[
    \Gu: J\mapsto
    \left(
        \begin{matrix}
        P & QF^{-1} \\ -FQ & -P^*
        \end{matrix}
    \right), \quad P=-2(I+J)^{-1}, Q=(J-I)(I+J)^{-1}=-(I+J)^{-1}(J-I)
    \]
    we have
    \[
    d\Gu(\dot J)=
    \left(
        \begin{matrix}
            \dot{P} & \dot{Q}F^{-1} \\ -F\dot{Q} & -\dot{P}^*
        \end{matrix}
        \right),
    \]
    where
    \[
    \dot{P}=2(I+J)^{-1}(\dot{J}+\dot{I})(I+J)^{-1},\quad \dot{Q}=2(I+J)^{-1}(J\dot J-I\dot I)(I+J)^{-1}.
    \]
    Furthermore, since $I=-F^{-1}JF^*$, we find
    \[
    \dot I=(I+J)\dot{J}(I+J)^{-1},
    \]
    \[
    I\dot I=(I+J)J\dot{J}(I+J)^{-1},
    \]
    hence
    \[
    \begin{split}
        \J d\Gu(\dot{J})=
        \left(
        \begin{matrix}
            P & QF^{-1} \\ -FQ & -P^*
        \end{matrix}\right)
        \left(
        \begin{matrix}
            \dot{P} & \dot{Q}F^{-1} \\ -F\dot{Q} & -\dot{P}^*
        \end{matrix}
        \right)=
        \left(
        \begin{matrix}
            P\dot{P}-Q\dot Q & (P\dot{Q}+Q\dot P)F^{-1} \\ -F(P\dot{Q}+Q\dot P) & -(P\dot{P}-Q\dot Q)^*
        \end{matrix}
        \right).
    \end{split}
    \]
    We thus can compute
    \[
    \begin{split}
    P\dot P-Q\dot Q&=-4(I+J)^{-2}(\dot J+\dot I)(I+J)^{-1}-2(J-I)(I+J)^{-2}(J\dot J-I\dot I)(I+J)^{-1}\\&=
    2(I+J)^{-2}\left(
        -2(\dot J+\dot I)-(J-I)(J\dot J-I\dot I)
    \right)(I+J)^{-1}\\&=
    2(I+J)^{-2}(-\dot J-\dot I+JI\dot I+IJ\dot J)(I+J)^{-1}\\&=
    2(I+J)^{-1}(J\dot J+I\dot I)(I+J)^{-1}.
    \end{split}
    \]
    \[
    \begin{split}
    P\dot Q+Q\dot P&=-4(I+J)^{-2}(J\dot J-I\dot I)(I+J)^{-1}+2(J-I)(I+J)^{-2}(\dot J+\dot I)(I+J)^{-1}\\
    &=2(I+J)^{-2}(-2(J\dot J-I\dot I)+(J-I)(\dot J+\dot I))(I+J)^{-1}\\
    &=2(I+J)^{-1}(-\dot J+\dot I)(I+J)^{-1}
    \end{split}
    \]
    Comparing the resulting expressions with the entries of the matrix $d\Gu(J\dot J)$ we conclude that
    \[
    \J d\Gu(\dot J)=d\Gu(J\dot J)
    \]
    as claimed.
\end{proof}
Because of Lemma~\ref{l:complex},  we will abuse notations slightly and  tacitly identify $\AGK_F$ and $\AGK_{\J_F}$ via the map $\Gu$. We will denote the underlying almost complex structure given by the left multiplication with $J$ on ${\bf T}_{J}\left(\AGK_F\right)$ by ${\bf J}$.

\begin{rmk}
The $F$-taming condition  yields that any two elements $J_0, J \in {\AGK}_F$ are commensurable, i.e., $\det(J_0 + J) \neq 0$. Therefore one can apply the  Cayley transform with base-point $J_0$
\[
J\mapsto (J+J_0)^{-1}(J_0-J),
\]
to endow ${\AGK}_F$ with a structure of an open contractible subset ${\bf U}$ in a  Fr\'echet complex vector space. It then follows from Lemma~\ref{l:complex} and the identification in \cite[App.~B]{gauduchon-book} of the induced complex structure on ${\bf U}$ with the complex multiplication by $i$ that ${\bf J}$ is integrable.
\end{rmk}

Following Goto~\cite{Gotomoment} and Gauduchon~\cite{gauduchon-GIT} we define a formal K\"ahler structure on $(\AGK_F, {\bf J})$. We use the identification $\Gu: \AGK_F\to \AGK_{\J_F}$ throughout.
\begin{defn}[Formal K\"ahler structure on $\AGK_F$]\label{d:Omega}
    Let $\dot J_1,\dot J_2\in {\bf T}_J\AGK_F$ be two tangent vectors at $J\in \AGK_F$. Denote by $\dot\J_1=d\Gu(\dot J_1)$, $\dot\J_2=d\Gu(\dot J_2)$ the corresponding tangent vectors at $\J:=\Gu(J)\in \AGK_{\J_F}$. One defines a 2-form
    \begin{equation}\label{e:Omega}
        \boldsymbol{\Omega}(\dot J_1,\dot J_2)= \cc \int_M \tr(\J\dot\J_1\dot\J_2)F^{[n]}.
    \end{equation}
    It is shown in~\cite{Gotomoment} that $(\AGK_F,\bf J,\boldsymbol{\Omega})$ gives rise to a formal K\"ahler manifold with the underlying Riemannian metric
    \[
    \boldsymbol{g}(\dot J_1,\dot J_2)=\cc\int_M \tr(\dot\J_1\dot\J_2)F^{[n]}.
    \]
\end{defn}

One can trace through the identification $\Gu: \AGK_F\to \AGK_{\J_F}$ similarly to the proof of Lemma~\ref{l:complex} and get a more explicit expression for $\boldsymbol{\Omega}$ and $\boldsymbol{g}$ which is useful in finding the linearization of scalar curvature. We leave the details to the reader. 

\begin{lemma}\label{l:symplectic} The formal symplectic $\boldsymbol{\Omega}$ form and Riemannian metric $\boldsymbol{g}$ are given by
\[
\begin{split}
\boldsymbol{\Omega}_J(\dot J_1, \dot J_2) &= \ccc\int_M  \tr\left((I+J)^{-2}\big[(J\dot J_1 + I \dot I_1)(I+J)^{-1}(\dot J_2 + \dot I_2) + (\dot J_1 - \dot I_1) (I+ J)^{-1}(J \dot J_2 - I \dot I_2)\big] \right) F^{[n]}; \\
\boldsymbol{g}_J(\dot J_1, \dot J_2)&= \ccc\int_M  \tr\left((I+J)^{-2}\big[(\dot J_1 + \dot I_1)(I+J)^{-2}(\dot J_2 + \dot I_2) - (J\dot J_1 - I\dot I_1) (I+ J)^{-2}(J \dot J_2 - I \dot I_2)\big] \right) F^{[n]}.
\end{split}\]
\end{lemma}

\subsection{The momentum map}

Let ${\rm Ham}(M,F)$ be the group of Hamiltonian diffeomorphisms of $(M,F)$  whose Lie algebra $\mathfrak{ham}(M,F)$ is identified with the space $\CCF$ of zero mean smooth functions  on $(M, F)$ by
\[ \CCF \ni f \mapsto  -F^{-1}(df) \in \mathfrak{ham}(M,F), \]
thus giving rise to the usual Lie algebra isomorphism
\[
\mathfrak{ham}(M,F)\simeq \{f\in \CCF, \{\cdot,\cdot\}_F\},
\]
where $\{f,g\}_F=F^{-1}(df,dg)$ is the $F^{-1}$-Poisson pairing.
The group ${\rm Ham}(M,F)$ naturally acts on $\AGK_F$ by the induced action on the underlying almost complex structures $J$:
\[  \Phi \cdot  J : = (\Phi_*) J (\Phi_*)^{-1}, \qquad \Phi \in {\rm Ham}(M,F), \]
leading to a Lie algebra representation
\[
\CCF \ni f \mapsto {\bf Y}_f(J):= {\mathcal L}_{F^{-1}(df)} J \in {\bf T}_{J}\left(\AGK_F\right).
\]
From the point of view of generalized complex structures $(\JJ, \J_F)$ this action fixes $\J_F$ and pulls back $\J\in \End(TM\oplus T^*M)$. The crucial result of Goto~\cite{Gotomoment} (see also \cite{boulanger} for the toric case) is that this action on the K\"ahler Fr\'echet manifold $(\AGK_F,\boldsymbol{\Omega},{\bf J})$ is Hamiltonian.  Proposition \ref{t:gscal_equiv} above complements this with a concrete geometric interpretation of the momentum map in terms of the biHermitian geometry.

\begin{thm}[{\cite{Gotomoment}}]\label{t:GKscal-moment-map} The action of ${\rm Ham}(M, F)$ on $({\AGK}_F, {\bf \Omega}, {\bf J})$ is Hamiltonian. Specifically, for each $J\in {\AGK}_F$ there exists a function $\Gscal_{(F,J)}$ such that the pairing
\[
\langle \boldsymbol{\mu}(J), f \rangle := -\int_M \Gscal_{(F,J)} f F^{[n]}, \qquad f\in \CCF,
\]
is a momentum map for the action, i.e., for any infinitesimal variation $\dot J=\frac{d}{dt}\Big|_{t=0}J_t$, $J_t\in {\AGK}_F$ and any smooth function $f\in \CCF$ we have
\[
{\bf\Omega}({\mathcal L}_{F^{-1}(df)} J, \dot J) = \int_M \left(\frac{d}{dt}\Gscal_{(F, J_t)}\right) f  F^{[n]}.
\]
\end{thm}

We note that the setup in~\cite{Gotomoment} is slightly more general than stated here, as it allows for a symplectic-type generalized complex structure $\II$ twisted by an additional \textit{$B$-field transform}. However, in the special case when the $B$-field transform is trivial, the group acting on $J$ (or equivalently on $\JJ$) reduces to $\Ham(M, F)\subset \mathrm{Diff}(M)$.

\begin{rmk}[Average value of $\Gscal$] \label{r:avescal}
    Using~\eqref{e:gscal_trace} we observe that for any $J\in \AGK_{F}$ the average value of $\Gscal$
    \begin{equation}\label{e:gscal_average}
    \overline\mu:=\frac{\int_M \Gscal_{(F,J)}F^{[n]}}{\int_M F^{[n]}}=\frac{\int_M \rho_{(F,J)}\wedge F^{[n-1]}}{{\int_M F^{[n]}}}=2\pi \frac{\la c_1(K_\JJ^{-1})\cdot [F]^{n-1}, [M]\ra}{\la [F]^n,[M]\ra}
    \end{equation}
    is a topological constant depending only on $c_1(K_\JJ)\in H^2(M,\mathbb Z)$, $[F]\in H^2(M,\R)$.
\end{rmk}

\begin{defn}\label{d:extremal} We say that a symplectic type generalized K\"ahler structure $(F, J)$ is of \emph{constant  generalized scalar curvature}, and abbreviate \emph{cscGK}, if $\Gscal_{(F, J)}=\overline{\mu}$ is a constant function. We say that $(F, J)$ is an \emph{extremal} generalized K\"ahler structure if the vector field $\chi:= - F^{-1}\left(d\Gscal_{(F, J)}\right)$ preserves  $J$, i.e., ${\mathcal L}_{\chi} J= 0$. Vector $\chi$ will be referred to as the \emph{extremal vector field}.
\end{defn}
These definitions are motivated by the following immediate corollary of Theorem~\ref{t:GKscal-moment-map} (see also  \cite[\S 9]{GotoLichne}):

\begin{cor}\label{c:symplectic-Gscal} Let $(F, J)$ be a symplectic type generalized K\"ahler structure defined on a compact manifold $M$. Then,
\begin{itemize}
\item $(F, J)$ is cscGK iff $J$ is a zero of the momentum map $\boldsymbol{\mu}$ of ${\AGK}_{F}$.
\item $(F, J)$ is extremal iff $J$ is a critical point of the norm functional  $J \to ||\boldsymbol{\mu}(J)||^2$ on ${\AGK}_{F}$,
 where the norm is  the $L^2(M, F)$ norm  $||\boldsymbol{\mu}(J)||^2:= \int_M \left(\Gscal_{(F, J)} -\overline{\mu}\right)^2 F^{[n]}$.
 \end{itemize}
 \end{cor}
 \begin{proof}
    It is immediate by Theorem~\ref{t:GKscal-moment-map} that $(F,J)$ is cscGK if and only if $J$ is a zero of the moment map $\boldsymbol{\mu}\colon \AGK_F\mapsto (C_0^\infty(M,\R))^*$.
    
    For the second part we compute the variation of $||\boldsymbol{\mu}(J)||^2 := \int_M (\Gscal_{(F, J)} - \overline{\mu})^2 dV_F$ along the corresponding Hamiltonian vector field $\chi:=-F^{-1}(d\Gscal_{(F, J)})$.  By Theorem~\ref{t:GKscal-moment-map}, for any smooth path $J_t \in {\AGK}_{F}$ tangent to $\dot J$ at $J$, we have
    \[
        \frac{d}{dt} ||\boldsymbol{\mu}(J_t)||^2 =  2\boldsymbol{\Omega}_J\left(-{\mathcal L}_{\chi} J, \dot J \right).
    \]
    As  $\boldsymbol{\Omega}_J$ is nondegenerate,  $J\in {\AGK}_{F}$ is a critical point of $||\boldsymbol{\mu}(J)||^2$ iff $\chi$ preserves $J$.
    \end{proof}

\section{The Calabi program for symplectic type GK structures}\label{s:calabi_variational}

Motivated by Calabi's program in K\"ahler geometry, which seeks extremal and constant scalar curvature K\"ahler metrics in ${\KK}_{\alpha}$, the discussion in the previous section naturally leads one to ask a similar question in the symplectic type generalized K\"ahler context.  First in this section we show that the space $\GK^0_{\poiss,\alpha}$ is the formal complexified orbit for $\Ham(M, F)$, further justifying it as a natural generalization of $\KK_{\ga}$.  We then give variational characterizations of extremal and cscGK structures in terms of a Calabi functional and Mabuchi energy, respectively.  Using this structure we establish the Calabi--Lichnerowicz--Matsushima-type obstruction, the existence of extremal vector field, and Futaki character discussed in the introduction.  We end by computing the linearization of the scalar curvature, yielding a natural generalization of the Lichnerowicz operator in this setting.

\subsection{The complexified orbits for the action of \texorpdfstring{${\rm Ham}(M, F)$}{Ham(M,F)}}\label{s:Moser's trick}

We start by extending a key observation from \cite{donaldson-GIT} to the symplectic type generalized K\"ahler setting.
\begin{lemma}\label{l:complex-orbit}
    Let $(F_t, J_0)$ be a Hamiltonian flow deformation of a symplectic type generalized K\"ahler structure $(F_0, J_0)$,  with respect to a time dependent smooth function $\phi_t$. Let $\Phi_t \in {\rm Diff}(M), \, \Phi_0 ={\rm Id}$ be the isotopy of diffeomorphisms corresponding to the time dependent vector field $Z_t := - F_t^{-1}(I_td\phi_t)$.  Let
    \[
        J_t := \Phi_t^{-1} \cdot J_0:= (\Phi_t)_*^{-1} J_0 (\Phi_t)_*, \qquad \psi_t : = \Phi_t^*(\phi_t), \qquad  Y_{\psi_t} := -F_0^{-1}(d\psi_t).
    \]
    Then  for any $t$
    \[
    \Phi_t^*(F_t)= F_0, \qquad \dot J_t = J_t \left({\mathcal L}_{Y_{\psi_t}} J_t\right).
    \]
\end{lemma}
\begin{proof} Using that  along the Hamiltonian flow deformation $\dot F  = dd^c_I \phi$,  $Z_t= - F_t^{-1}(I_td\phi_t)$ defines a Moser isotopy, i.e, $\Phi_t^*(F_t)=F_0$. Furthermore, we compute
\[ Z_t = - F_t^{-1}(d^c_{I_t}\phi_t)= \tfrac{1}{2}(I_t+J_0) I_t(d\phi_t) = \tfrac{1}{2}J_0(I_t+J_0)(d\phi_t)= -J_0 F_t^{-1}(d\phi_t).\]
Letting $J_t =  (\Phi_t)_*^{-1} J_0 (\Phi_{t})_{*}$,   we get
\[ \frac{d}{dt} J_t= \mathcal{L}_{(\Phi_t)^{-1} \cdot Z_t} J_t  = {\mathcal L}_{J_t Y_{\psi_t}} J_t= J_t\left({\mathcal L}_{Y_{\psi_t}} J\right)
\]
as claimed.
\end{proof}

\begin{rmk} By virtue of Lemma~\ref{l:complex}, it follows that $F_t\in \GK_{\poiss,\alpha}$ is mapped via $\Phi_t$ into a curve inside the formal ``complexified orbit'' of ${\rm Ham}(M, F_0)$ in  $({\AGK}_{F_0}, {\bf J})$, which is transversal to the  ${\rm Ham}(M, F_0)$ orbit of $(F_0, J_0)$.
\end{rmk}

\subsection{The generalized K\"ahler Calabi functional}

We now consider an extension of the Calabi functional to our notion of a generalized K\"ahler class $\GK_{\poiss,\alpha}$, which gives a further variational characterization of extremal metrics. This functional is related to the the square norm of the momentum map studied in \cite[\S 9]{GotoLichne} (see also Corollary~\ref{c:symplectic-Gscal}), where the notion of extremal generalized K\"ahler structures was introduced.

\begin{defn}[Generalized K\"ahler Calabi functional]\label{d:Calabi} We define the Calabi functional by
\[ {\bf Ca}(F)= \int_M {\Gscal}_{(F, J)}^2 F^{[n]},  \qquad F \in \GK_{\poiss,\alpha}.\]
\end{defn}
\begin{prop}\label{p:Calabi} $F \in \GK_{\poiss,\alpha}$ is a critical point of ${\bf Ca}$ if and only if  $F$ is extremal.
\end{prop}
\begin{proof} This is similar to the proof of the second part of Corollary~\ref{c:symplectic-Gscal}.  Let $F_t$ a Hamiltonian flow starting at $F$ and corresponding to $\phi_t$.  Using a Moser isotopy $\Phi_t$ as in Lemma~\ref{l:complex-orbit} we can translate the above property with respect to a path $(F, J_t) \in {\AGK}_{F}$, and then compute
\[
\begin{split}
\left. \frac{d}{dt} \right |_{t=0} \int_M {\Gscal}_{(F, J_t)}^2 F^{[n]} &= \left. \frac{d}{dt} \right |_{t=0} \int_M \left({\Gscal}_{(F, J_t)}- \overline{\mu}\right)^2 F^{[n]}\\
      & = 2 \boldsymbol{\Omega}\left( -{\mathcal L}_{\chi} J, \dot J\right) \\
   &= -2 \boldsymbol{\Omega}\left( {\mathcal L}_{\chi} J, J {\mathcal L}_Y J\right),
   \end{split}
\]
where $\chi:= -F^{-1}(d\Gscal_{(F, J_0)})$ and $Y= -F^{-1}(d\phi_0)$. In the above equalities, we first used that the averaged generalized scalar curvature is constant in $t$, see Remark~\ref{r:avescal}, while the third line follows from the computation in Corollary \ref{c:symplectic-Gscal}, and for the last line we have used Lemma~\ref{l:complex-orbit}.  The result follows by specializing the above computation with $\phi_t := \Gscal_{(F_t, J_0)}$ we conclude that $\boldsymbol{g}(\mathcal L_{\chi}J,\mathcal L_{\chi}J)=0$ so that $\mathcal L_{\chi}J=0$. \end{proof}

\subsection{The generalized K\"ahler  Mabuchi functional}\label{ss:Mabuchi}

In \cite{Mabuchi} Mabuchi further introduced his ``$K$-energy'' which is essential to the proof of uniqueness of cscK metrics.  Formally it can be thought of as extending the moment map to the complexified orbit of the action of Hamiltonian diffeomorphisms of a K\"ahler form.  A similar phenomenon exists here, although in general we only obtain a one-form corresponding to the differential of the Mabuchi energy.  Specifically, we introduce a $1$-form $\boldsymbol{\tau}$ on  $\GK_{\poiss,\alpha}$  by its value at fundamental vector fields ${\bf X}_{\phi}\in {\bf T}_F\GK_{\poiss,\alpha}$, $\phi\in \CCF$, see \eqref{e:fundamental-field}:
\[
\boldsymbol{\tau}_{F}({\bf X}_{\phi}):= - \int_M  \phi\,{\Gscal}_{(F, J)} \, F^{[n]}, \qquad F \in \GK_{\poiss,\alpha}.
\]

\begin{prop}\label{p:Mabuchi} The $1$-form $\boldsymbol{\tau}$ is closed.
\end{prop}
\begin{proof} Using Lemma~\ref{l:commuting}, we need to check that
\[
    {\bf X}_{\psi}\left(\boldsymbol{\tau}({\bf X}_{\phi})\right) - {\bf X}_{\phi}\left(\boldsymbol{\tau}({\bf X}_{\psi})\right) - \boldsymbol{\tau}({\bf X}_{\{\phi, \psi\}_{\nonhalf\poiss}}) =0.
\]
Let us examine the first term at a point $F\in \GK_{\poiss,\alpha}^0$.
\[ {\bf X}_{\psi}\left(\boldsymbol{\tau}({\bf X}_{\phi})\right) = -\frac{d}{dt} \int_M \phi\,{\Gscal}_{(F_t, J)} F_t^{[n]}, \]
where $F_t$ is a Hamiltonian generalized K\"ahler deformation of $F$ generated by $\psi$.  We use a Moser isotopy $\Phi_t$ corresponding to the vector field $Z_{t} =-F^{-1}(I_td\psi)= -J F^{-1}(d\psi)$, as in the proof of Lemma~\ref{l:complex-orbit}, and pull back by $\Phi_t$ the integrand of the above expression to yield
\[
    {\bf X}_{\psi}\left(\boldsymbol{\tau}({\bf X}_{\phi})\right) = -\frac{d}{dt} \int_M \phi_t\,{\Gscal}_{(F, J_t)} F^{[n]},
\]
where $J_t$  and $\phi_t$ are the pull-backs of $J_0=J$ and $\phi_0 = \phi$ under $\Phi_t$.  We also put $\psi_t =\Phi_t^*(\psi)$. Using Theorem~\ref{t:GKscal-moment-map}] and Lemma~\ref{l:complex-orbit} we compute
\[
\begin{split}
-\frac{d}{dt} \int_M \phi_t\,{\Gscal}_{(F, J_t)} F^{[n]}  =&\ \boldsymbol{\Omega}({\bf Y}_{\phi_t}, {\bf J}{\bf Y}_{\psi_t}) - \int_M {\Gscal}_{(F, J_t)} d\phi_t (Z_t)   F^{[n]}  \\
=&\ \boldsymbol{\Omega}({\bf Y}_{\phi_t}, {\bf J}{\bf Y}_{\psi_t})  + \tfrac{1}{2} \int_M {\Gscal}_{(F, J_t)} \left\langle d\phi_t,   J_t (I_t+ J_t) d\psi_t\right\rangle_{g_t}   F^{[n]}   \\
=&\ \boldsymbol{\Omega}({\bf Y}_{\phi_t}, {\bf J}{\bf Y}_{\psi_t})  - \tfrac{1}{2} \int_M {\Gscal}_{(F, J_t)} \left\langle d\phi_t,  d\psi_t\right\rangle_{g_t}   F^{[n]} \\
&\ 
+\tfrac{1}{2} \int_M {\Gscal}_{(F, J_t)} \left\langle d\phi_t,   J_t I_t d\psi_t\right\rangle_{g_t}   F^{[n]}.
\end{split}
\]
Noting that the two terms at  the third line are symmetric in $\phi$ and $\psi$, we thus get from the above
\[
\begin{split}
{\bf X}_{\psi}\left(\boldsymbol{\tau}({\bf X}_{\phi})\right) - {\bf X}_{\phi}\left(\boldsymbol{\tau}({\bf X}_{\psi})\right) =& -\tfrac{1}{2} \int_M {\Gscal}_{(F_0, J_t)} \langle d\phi_t, [I_t, J_t] d\psi_t \rangle_{g_t} F_0^{[n]} \\
=  &  -\tfrac{1}{2} \int_M {\Gscal}_{(F_0, J_t)} d\phi_t (\poiss_t(d\psi_t)) F_0^{[n]} \\
= & \int_M {\Gscal}_{(F_0, J_t)} \{\phi_t, \psi_t\}_{\nonhalf\poiss} F_0^{[n]}.
 \end{split}\]
Pulling back via $\Phi_t^{-1}$, we get the claim. \end{proof}
 As  $\GK_{\poiss,\alpha}$ may have a  nontrivial topology, it is not immediately clear whether  
 \[
    \boldsymbol{\tau} = {\bf d} {\bf M}
\] for some functional ${\bf M} : \GK_{\poiss,\alpha} \to \R$.
 
\begin{defn}[Generalized K\"ahler Mabuchi functional]\label{d:Mabuchi} We define the \emph{Mabuchi functional} ${\bf M}_{F_0}$ to be the primitive of $\boldsymbol{\tau}$ (if it exists) satisfying
${\mathbf M}_{F_0}(F_0)=0$.
\end{defn}

\subsection{The generalized K\"ahler Calabi--Lichnerowicz--Matsushima Theorem}

Here we prove a structural result on the group of Poisson automorphisms of an extremal GK structure, extending the classical results of Calabi-Lichnerowicz-Matsushima.

\subsubsection{The reduced automorphism group}

Let ${\Aut}_0(J, \poiss_J)$ be the connected component of the (complex) automorphism group of $J$ and $\poiss_J$. We denote by
\[
\mathfrak{h}(J,\poiss_J):=\mathrm{Lie}({\Aut}_0(J, \poiss_J))
\]
its Lie algebra, which is  a $J$-invariant subalgebra of the algebra of real holomorphic vector fields on $(M, J)$.  We first observe several basic structural results concerning $\mathfrak{h}(F,J)$, extending the theory in the K\"ahler case.

Let us denote by
\[
\mathcal{H}^{1}_J:=\left\{ \xi \in \Wedge^1(M) \, \Big| \, d\xi = d^c_J\xi =0\right\},
\]
the space of $d$ and $d^c_J$-closed $1$-forms on $(M, J)$:  $\mathcal{H}^{1}_J$ is just the underlying real vector space of the space  of holomorphic $1$-forms on $(M, J)$, which are automatically closed by Corollary~\ref{c:holo_closed}.

Given $X\in \mathfrak{h}(J,\poiss_J)$,  we consider an infinitesimal variation $\mathcal L_X F$ of $F\in \GK_{\poiss,\alpha}$. As in the proof of Proposition~\ref{p:gk_class} we conclude that the $(-1,n-1)$-degree component of $F(X)\wedge e^{\i F}$ can be decomposed as 
\[
(F(X)\wedge e^{\i F})^{(-1,n-1)}=(F(X)\wedge e^{\i F})^{(-1,n-1)}_h+\delta_-\left(\big(\psi + \i\phi\big) e^{\i F}\right),
\]
$(F(X)\wedge e^{\i F})^{(-1,n-1)}_h=\eta^\C_X\wedge e^{\i F}$ for a closed 1-form $\eta^\C_X\in\Lambda^{1,0}_J(M)$, see \eqref{Hodge-isomorphism}. Equivalently,
\begin{equation}\label{e:X-decomposition}
    F(X)=(I+J)\eta_X+d\phi+Id\psi,
\end{equation}
where $\eta_X\in \mathcal{H}^1_J$.

We note that for any $\eta\in \mathcal{H}^1_J$ and any $X\in \mathfrak{h}(J,\poiss_J)$
\[
\mathcal L_{X}\eta=0.
\]
Indeed, as any element $\eta\in \mathcal{H}^1_J$ is uniquely determined by its cohomology class $[\eta]\in H^1(M, \R)$, and since an infinitesimal symmetry of $J$ preserves both the cohomology classes of closed forms and the space $\mathcal{H}^1_J$, the Lie derivative $\mathcal L_X\eta$ must vanish. This observation implies that $\eta(X)$ is a constant.
We thus get a real linear map
\[
\tau \, :  \, \mathfrak{h}(J,\poiss_J)\mapsto \left(\mathcal{H}^1_J\right)^*,\quad \tau(X)(\eta)=\eta(X).
\]
Since for $X,Y\in\mathfrak{h}(J,\poiss_J)$ we have $\eta(X)=\mathrm{const}$ and $\eta(Y)=\mathrm{const}$, and $\eta$ is closed, it  follows that $\eta([X,Y])=0$. Therefore $\tau$ is a Lie algebra homomorphism  from $\mathfrak{h}(J,\poiss_J)$  to the abelian Lie algebra $\left(\mathcal{H}^1_J\right)^*$, and its kernel
\begin{equation}\label{e:hred_ker}
\hred(J,\poiss_J):=\Ker(\tau) < \mathfrak{h}(J,\poiss_J)
\end{equation}
is an ideal. We have an alternative description of $\hred(J,\poiss_J)$:
\begin{lemma}
    The Lie subalgebra $\hred(J,\poiss_J)\subset \mathfrak{h}(J,\poiss_J)$ consists of all elements $X\in\mathfrak{h}(J,\poiss_J)$ which can be represented as
    \begin{equation}\label{e:X_reduced}
    X=F^{-1}(d\varphi)+F^{-1}(Id\psi),\quad \phi,\psi\in C^\infty(M,\R).
    \end{equation}
\end{lemma}
\begin{proof}
    First we prove that $X$ of the form~\eqref{e:X_reduced} lies in $\Ker \tau$.  Indeed, for any 1-form $\eta\in \mathcal{H}^1_J$ we compute
    \[
    \begin{split}
    \int_M \eta(X)F^{[n]}= &\ \int_M \tr_F(\eta\wedge (d\varphi+Id\psi))F^{[n]}\\
    =&\ \int_M \eta\wedge d\varphi\wedge F^{[n-1]}-
    \int_M J\eta\wedge d\psi\wedge F^{[n-1]}=0,
    \end{split}
    \]
    where we used that $\tr_F(\eta\wedge Id\psi)=-\tr_F(J\eta\wedge d\psi)$ (see \eqref{a:Fvolumeidentity}), and that the forms $\eta, J\eta$ and $F$ are closed. Since $\eta(X)$ is constant on $M$, this implies that $\eta(X)=0$, so that $X$ belongs to the kernel of $\tau$.

    Now conversely assume that $X\in\mathfrak{h}(J,\poiss_J)$ is any element with the corresponding decomposition~\eqref{e:X-decomposition}. Using the computation in the first part of the proof, we know have
    \[
    \begin{split}
        \eta_X(X)=\frac{1}{V}\int_M \tr_F(\eta_X\wedge (I+J)\eta_X)F^{[n]}=\frac{1}{2V}\int_M \left|(I+J)\eta_X\right|_g^2  F^{[n]}\geq 0,
    \end{split}
    \]
    where $V=\int_M F^{[n]}$ and $g$ is the Riemannian metric determined by $(F, J)$. The latter expression is nonzero unless $\eta_X=0$, as claimed.
\end{proof}

Motivated by the previous lemma and the classical K\"ahler setup,  we make the following definition:

\begin{defn}\label{d:hred}
    The Lie algebra of~\emph{reduced automorphisms} of $(J,\poiss_J)$ is defined as
    \[
    \hred(J,\poiss_J):=\{X\in\mathfrak{h}(J,\poiss_J)\ |\ X=F^{-1}(d\varphi+Id\psi),\quad \phi,\psi\in C^\infty(M,\R)\}
    \]
    which can be equivalently rewritten as
    \[
    \hred(J,\poiss_J)=\{X\in\mathfrak{h}(J,\poiss_J)\ |\ X=F^{-1}(d\varphi)+JF^{-1}(d\psi),\quad \phi,\psi\in C^\infty(M,\R)\},
    \]
    since $I^*=-FJF^{-1}$.
\end{defn}

\begin{rmk}
   By the second equality in Definition~\ref{d:hred},   $\hred(J,\poiss_J)$ is invariant under $J$, so that it can be endowed with a complex Lie algebra structure.
    It is clear from the identification~\eqref{e:hred_ker} that $\hred(J,\poiss_J)$ is independent of $F\in\GK_{\poiss}$. Following the argument of LeBrun and Simanca \cite{LS}, one can alternatively characterize $\hred(J,\poiss_J)< \mathfrak{h}(J,\poiss_J)$ as the ideal  of holomorphic vector fields  preserving $\poiss_J$ and whose zero set is nonempty.
\end{rmk}

\begin{rmk}[Comparison with the approach of Goto] 
    To put our definition into context let us compare it to the approach in~\cite[\S 5.1]{GotoLichne}. Using the formalism of  generalized geometry,  Goto defines a complex Lie algebra $\mathfrak{g}_0\subset \bar{L_{\JJ}}\subset \End(TM\oplus T^*M)\otimes\C$, where $\bar{L_{\JJ}}$ is the $(-\i)$-eigenspace of $\JJ:=\JJ_{P,Q}$ (see~\eqref{e:GK-interpretation}). The Lie bracket on $\mathfrak g_0$ is given by the \emph{Dorfman bracket} (see~\cite[Ch.\,2]{GRFbook}) on the sections of $\End(TM\oplus T^*M)\otimes\C$. The underlying \emph{real} Lie algebra $\Re(\mathfrak{g}_0)$ is
    \[
        \Re(\mathfrak{g}_0)=\left\{e\in TM\oplus T^*M\ |\ e=F^{-1}(d\phi)+\JJ F^{-1}(d\psi),\quad L^{\mathrm{Dor}}_e\JJ=0, \,  \phi,\psi\in C^\infty(M)\right\},
    \]
    where $L^{\mathrm{Dor}}_e$ is the infinitesimal action of an element $e\in TM\oplus T^*M$ on $\JJ$ given by the Dorfman bracket. The Lie algebra structure on $\Re(\mathfrak g_0)$ is equivalently given by the $F$-Poisson bracket on the corresponding smooth complex-valued functions $\phi + i \psi$. One can show that $e:=F^{-1}(d\phi)+\JJ F^{-1}(d\psi)$ is in $\Re(\mathfrak{g}_0)$ if and only if $X:=F^{-1}(d\phi)+JF^{-1}(d\psi)\in\hred(J,\poiss_J)$. This gives rise to a natural isomorphism of real Lie algebras. 
    \[
    \hred(J,\poiss_J)\simeq \mathfrak \Re(\mathfrak{g}_0),
    \]
    respecting the underlying complex structures and the $F$-poisson bracket on $\{\phi + i \psi\}$. 
\end{rmk}

As in the K\"ahler case (see e.g., \cite{gauduchon-book}), the Lie algebra homomorphism $\tau \, : \,  \mathfrak{h}(J,\poiss_J)\mapsto \left(\mathcal{H}^1_J\right)^*$ can be integrated to a Lie group morphism
\[
\hat \tau : \Aut_0(J,\poiss_J)\to \left(\mathcal{H}^1_J\right)^*/\Gamma,
\]
where $\Gamma=H_1(M,\mathbb Z)$.

\begin{defn}[Reduced automorphisms of $(M,J,\poiss_J)$]\label{d:gred}
    The Lie  group $\Aut_{\mathrm{red}}(J,\poiss_J)$ of~\emph{reduced automorphisms} of $(J,\poiss_J)$ is  the connected component of identity of the kernel of $\hat \tau$ inside $\Aut_0(J,\poiss_J)$. Thus $\Aut_{\mathrm{red}}(J,\poiss_J) \subset \Aut_0(J,\poiss_J)$ is a closed subgroup with Lie algebra $\hred(J,\poiss_J)$.
\end{defn}

\subsubsection{The reduced automorphism group of an extremal generalized K\"ahler manifold}

We note that the infinitesimal action of $X = F^{-1} (d \phi + I d \psi) \in\hred(J,\poiss_J)$ on $F$ is $\mathcal L_XF=dd^c_I\psi$, which vanishes if and only if $\psi=\mathrm{const}$. Thus we get a Lie subalgebra
\begin{equation}\label{lie-algebra-k}
\mathfrak{k}(F,J)=\{X\in\hred(J,\poiss_J)\ |\ \mathcal L_XF=0\}= \{ X \in \mathfrak{h}(J,\poiss_J)\ |\ X=F^{-1}(d\varphi)\} < \hred(J,\poiss_J).
\end{equation}
We observe that $\mathfrak{k}(F,J)$ is the Lie algebra of the \emph{compact} group $K_0$~--- the connected component of the identity of
\[
K:=\mathrm{Isom}(g)\cap \mathrm{Aut}(I)\cap \mathrm{Aut}_{\mathrm{red}}(J,\poiss_J)=\mathrm{Ham}(M,F) \cap \mathrm{Aut}_{\mathrm{red}}(J,\poiss_J).
\]
Note that, since $\Aut(I)$ and ${\Aut}_{\mathrm{red}}(J,\poiss_J)$ are closed Lie subgroups of $\mathrm{Diff}(M)$ and $\mathrm{Isom}(g)$ is compact, it follows that $K$ is a compact Lie group.

\begin{thm}\label{t:Calabi-Lichne-Matsushima}
Suppose $F \in {\GK}_{\poiss, \alpha}$ is an extremal generalized K\"ahler structure with the holomorphic extremal vector field
\[
\chi = -F^{-1}(d\Gscal_{(F, J)}).
\]
Denote by ${\Aut}_{\rm red}(J, \poiss_J)^{\chi}\subset {\Aut}_{\rm red}(J, \poiss_J)$ the connected subgroup preserving $\chi$. Then the group 
\[
K_0:=({\Aut}_{\rm red}(J, \poiss_J)\cap \Ham(M,F))_0
\]
is a maximal compact connected subgroup of ${\Aut}_{\rm red}(J, \poiss_J)$ and ${\Aut}_{\rm red}(J, \poiss_J)^{\chi}=K_0^\C$.
In particular, $(F, J)$  must be invariant under a maximal real torus in ${\Aut}_{\rm red}(J, \poiss_J)$ containing the one-parameter subgroup $\exp(t\chi)$. If $F$ is cscGK,  then $\chi=0$ and  ${\Aut}_{\rm red}(J, \poiss_J)= K_0^{\C}$.
\end{thm}

\begin{proof}
    We follow the recent treatment \cite{LSW} which deduces the above properties by formal arguments  from the moment map picture.  Another relevant reference for this approach is \cite{Wang}. We want to apply \cite[Thm.3.3]{LSW}. In our setting, we consider the (holomorphic) action of ${\rm Ham}(M, F)$ on $({\AGK}_F, {\bf J})$ endowed with the formal K\"ahler structure ${\bf \Omega}$,  and with momentum map $\boldsymbol{\mu}\colon\AGK_F\to (\CCF)^*$:
    \[
        \langle \boldsymbol{\mu}(J), \phi \rangle:= - \int_M \Gscal_{(F, J)} \phi F^{[n]}, \qquad J \in {\AGK}_F, \, \phi \in \CCF.
    \]
    Recall that we identify the Lie algebra $\mathfrak{ham}(M, F)$ of Hamiltonian vector fields with the space of normalized smooth functions $\CCF$ equipped with the $F$-Poisson bracket $\{\cdot, \cdot\}_F$, through the standard Lie algebra isomorphism
    \[ 
        \CCF \ni \phi \mapsto  Y_{\phi}:= -F^{-1}(d\phi)
        \in  \mathfrak{ham}(M, F).
    \]
    We further consider the ${\rm ad}$-invariant inner product $\llangle \cdot, \cdot \rrangle$ on $\mathfrak{ham}(M, F)$, defined by
    \[
        \llangle \phi_1, \phi_2 \rrangle := \int_M \phi_1 \phi_2 F^{[n]}, \qquad \phi_1, \phi_2 \in \CCF \simeq \mathfrak{ham}(M, F).
    \]
    The real Lie algebra $\left(\CCF, \{\cdot, \cdot \}_F\right) \simeq \mathfrak{ham}(M, J)$ can be complexified, and is then denoted by
    \[
        \mathfrak{ham}(M, F)^{\C}=
        \{-F^{-1}(d\phi)-JF^{-1}(d\psi) \}\simeq \{ \phi + i \psi\},\quad \phi, \psi \in \CCF.
    \]
     There is an ``infinitesimal action'' of  $\mathfrak{ham}(M, F)^{\C}$  on ${\AGK}_F$, such that $\phi + i \psi$ is mapped to the vector field $J \to {\mathcal L}_{F^{-1}(d\phi)} J + {\bf J} {\mathcal L}_{F^{-1}(d\psi)} J \in {\bf T}_J \left({\AGK}_F\right)$.  Suppose that $J \in {\AGK}_F$ is a generalized K\"ahler structure, i.e., $J$ and $I$ are both integrable. We claim then that the stabilizer of $J$ in  $\mathfrak{ham}(M, F)^{\C}$ is $\hred(J,\poiss_J)$. Indeed,  in this case ${\mathcal L}_{F^{-1}(d\phi)} J + {\bf J} {\mathcal L}_{F^{-1}(d\psi)} J= - \mathcal{L}_{X} J$ with  $X=-F^{-1}(d\phi)-JF^{-1}(d\psi)$. If $X=-F^{-1}(d\phi)-JF^{-1}(d\psi)$ preserves $J$, then $\mathcal{L}_X F=-dd^c_I\psi$ is of $I$-type $(1,1)$, so that by~\eqref{algebraic} and \eqref{tangent-identification} $\mathcal{L}_X\poiss=0$. Therefore $X\in\mathfrak{h}(J,\poiss_J)$ and being of the form as in Definition~\ref{d:hred} it lies in $\hred(J,\poiss_J)$.

    We now consider the functional $||\boldsymbol{\mu}(J)||^2$ on the formal K\"ahler Fr\`echet manifold $({\AGK}_F, \bf{\Omega}, \bf{J})$ as in Corollary~\ref{c:symplectic-Gscal}, which is the square-norm of the momentum map $\boldsymbol{\mu}$ with respect to $\llangle \cdot, \cdot \rrangle$. Let $J \in {\AGK}_F$ be a generalized K\"ahler structure which is a critical point of $||\boldsymbol{\mu}(J)||^2$, i.e., an extremal generalized K\"ahler structure by Corollary~\ref{c:symplectic-Gscal}. Then the extremal vector field
    \[
        \chi := -F^{-1}(d\Gscal_{(F, J)})    \]
    belongs to the center of $\mathfrak{k}(F, J)$. By the generalized Calabi--Lichnerowicz--Matsushima decomposition (see \cite[Thm.3.3]{LSW}) applied to the convex functional $||\boldsymbol{\mu}||^2$ on
    \[
        \left({\AGK}_F, {\bf \Omega}, {\bf J}, \mathfrak{ham}(M, F), \llangle \cdot, \cdot \rrangle\right),
    \]
    there is a semidirect splitting $\hred(J,\poiss_J) = \hred(J,\poiss_J)^{\chi} + \mathfrak{s}(J,\poiss_J)$ where
    the centralizer $\hred(J,\poiss_J)^{\chi}$ of $\chi$ is a reductive algebra satisfying
    \begin{equation*}
        \hred(J,\poiss_J)^{\chi} = \mathfrak{k}(F, J) \otimes \C, 
    \end{equation*}
    whereas $\mathfrak{s}(J,\poiss_J)$ is a solvable ideal. This shows that $\hred(J,\poiss_J)^{\chi}$ is a maximal reductive Lie subalgebra of  $\hred(J, \poiss_J)$. The claims follow from this.
\end{proof}

\begin{rmk}\label{r:b_1=0} In the case when $b_1(M)=0$, we have that  $\mathcal{H}^1_J =\{0\}$ and then $\Aut_{\rm red}(J, \poiss_J) = \Aut_0(J, \poiss_J)$ is just the connected component of the identity of the group of Poisson automorphisms of $(J, \poiss_J)$. Furthermore, in this case $\mathfrak{k}(F, J)$ reduces to the Lie algebra of Killing fields of $(F, J)$,  so that $K_0$  is just the connected component of the isometry group of $(F, J)$.
\end{rmk}

\begin{rmk}\label{r:Hitchin} 
(a) Theorem~\ref{t:Calabi-Lichne-Matsushima} gives obstructions to the existence of extremal,  non-cscGK generalized K\"ahler structures. We illustrate this on the following example.  Let us consider the second Hirzebruch complex surface \[(M, J)={\bf F}_2:= {\bf P}({\mathcal O} \oplus {\mathcal O}(-2))\to {\bf CP}^1,\] endowed with its Liouville Poisson structure $\poiss^{L}_{J}$,  i.e., $\poiss^{L}_{J}$ is the inverse of the Liouville symplectic form on ${\mathcal O}(-2)=T^*{\bf CP}^1$.  Let $\omega$ be a K\"ahler  metric on $(M, J)$, defining a deRham class  $\alpha = [\omega]$. By \cite[Theorem~8.15]{Goto-AM}  or  \cite[Corollary~7.3]{Gualtieri-Hamiltonian}, one can deform $\omega$ in order to obtain on
$(M, J)$ a $t\poiss^{L}_J$-compatible symplectic type GK structure $F_t$,  for  $|t|<\varepsilon$ with $F_0=\omega$.  As $H^{0,1}(M, J)=0$, the arguments in the proof of \cite[Theorem~7.1]{Gualtieri-Hamiltonian} actually show that $F_t \in \alpha$. 
Notice that the automorphism group  of the base ${\bf CP}^1$ preserves $\poiss_J^L$ and, in fact,
$\Aut_0(J, \poiss_J^{L}) = {\rm PGL}(2, \C)$. On the other hand,  the $\C^*$-action on the fibers of ${\bf F}_2$ preserves $\alpha$ and scales $\poiss_J^{L}$, so by acting with an element of $\C^*$,  the above results yield a symplectic type generalized K\"ahler structure $F \in {\GK}_{\poiss, \alpha}(M, J)$. As ${\rm PGL}(2, \C)$ is a semi-simple group (and thus has a trivial center),  by Theorem~\ref{t:Calabi-Lichne-Matsushima} any extremal generalized structure in ${\GK}_{\poiss,\alpha}^0(M, J)$ must be cscGK. Theorem~\ref{t:Calabi-Lichne-Matsushima} gives no further obstructions for the existence of a cscGK metric in ${\GK}_{\poiss,\alpha}^0(M, J)$.  It is thus interesting to know whether or not ${\GK}_{\poiss,\alpha}^0(M, J)$ does admit a ${\rm PGL}(2, \C)$-invariant cscGK structure.  At the same time, it is well-known that ${\GK}_{0,\alpha}(M, J)=\KK_{\alpha}(M, J)$ does not admit a  K\"ahler metric of constant scalar curvature. In fact, it admits an extremal K\"ahler metric with  non-constant scalar curvature~\cite{calabi}.  

\smallskip

(b) Still considering the same example, an interesting phenomenon related to the connectedness of $\GK_{\poiss, \alpha}$ and the size of $\poiss$ appears. Let us denote by $I$ the second complex structure defined by $(F, J)$. It is shown  \cite{hi-07}  (see the Remark on page 6) that
\[
    (M, I) \simeq {\bf CP}^1 \times {\bf CP}^1.
\]
The corresponding holomorphic Poisson tensor $\poiss_I$ on $(M, I)$ has zeros of order $2$ along the diagonal $\Delta \subset {\bf CP}^1 \times {\bf CP}^1$. We find therefore $\Aut_0(I, \poiss_I)= {\rm PGL}(2, \C)$, which is also the stabilizer of $\Delta$ inside $\Aut_0(M,I)={\rm PGL}(2, \C) \times {\rm PGL}(2, \C)$. Notice that, unlike $(M, J, \poiss_{J}^{L})$, we cannot rescale $\poiss_I$ with  an element of  $\Aut_0(M, I)$. As $(M, I)=H^{0,2}(M,I)=0$,   $H^2(M, \R)= H^{1,1}(M, \R)$, $\alpha$ is a $(1,1)$-class on $(M, I)$. As $F\in \alpha$ tames $I$,  $\alpha$ satisfies the Nakai--Moishezon positivity condition~\cite{DP} and thus is a K\"ahler class, i.e., there exists a K\"ahler structure $\omega' \in \alpha$ on $(M, I)$. Using \cite[Theorem~8.15]{Goto-AM}  and  \cite[Corollary~7.3]{Gualtieri-Hamiltonian} on $(M, I, \omega')$, one gets symplectic type generalized K\"ahler structures $F'_t \in {\GK}_{t\poiss, \alpha}(M, I)$, defined for $|t|< \varepsilon'$.  Notice that the second complex structures $J_t'$ are now biholomorphic to $I$: this follows for instance by the fact that $H^1(M, T_I^{1,0}M)=0$, i.e., ${\bf CP}^1\times {\bf CP}^1$ is rigid. If we were able to obtain such deformations up to $t=1$, we will have  for $F':=F'_1$ that $F, F' \in {\GK}_{\poiss, \alpha}(M, I)$ but $F'$ and $F$ cannot be in the same connected component  ${\GK}^0_{\poiss, \alpha}(M, I)$ as $J'$ and $J$ are not biholomorphic.  \end{rmk}

\subsection{The generalized K\"ahler extremal vector field}\label{ss:extremal}

The momentum map interpretation of the generalized K\"ahler scalar curvature leads to an alternative intrinsic description of the extremal vector field.

\begin{prop} \label{p:futkai_symplectic} Let $G\subset {\rm Ham}(M, F)$ be a compact subgroup with Lie algebra $\mathfrak{g}$. Let $\AGK^G_F$ be the space of $G$-invariant structures in $\AGK_F$. For an element $a \in \mathfrak{g}$ let $\phi:=\phi_{a,F}\in \CCF$ be the corresponding Hamiltonian. Then the integral
\[
\int_M \phi_{a, F}\,\Gscal_{(F,J)}dV_F
\]
is independent of $J\in \AGK^G_F$. In other words, the $L^2(M,dV_F)$-projection 
$\Pi_F(\Gscal_{(F,J)})$ of $\Gscal_{(F,J)}$ on the space of Hamiltonian potentials $\{\phi_{a,F}\ |\ a\in\mathfrak{g}\}\subset \CCF$ is constant on $\AGK^G_F$. In particular the vector field $\chi\in\mathfrak{g}$
\[
\chi=-F^{-1}(d\Pi_F(\Gscal_{(F,J)}))
\]
is independent of $J\in\AGK_G^G$.
\end{prop}

\begin{proof} The claim follows from Theorem~\ref{t:GKscal-moment-map} along the lines of proof of Corollary~\ref{c:symplectic-Gscal}. Let $J_0,J_1\in \AGK_F^G$ be two invariant almost complex structure. Using the Cayley transform we can connect them via a path $J_t\in \AGK_F^G$:
\[
J_t=(1+tS)J_0(1+tS)^{-1},\quad S=(J_1+J_0)^{-1}(J_0-J_1).
\]
Consider a Hamiltonian vector field
$Y_{\phi} = -F^{-1}(d\phi)\in \mathfrak{g}$ on $(M, F)$. Let ${\bf Y}_{\phi} (J) = -{\mathcal L}_{Y_{\phi}} J$ be the induced fundamental vector field on ${\AGK}_{F}$. As ${\bf Y}_{\phi}$ preserves any $G$-invariant  element of ${\AGK}_F$,  we have ${\bf Y}_{\phi}(J)=0$ for $J\in\AGK_F^G$ and thus  $\boldsymbol{\Omega}_J({\bf Y}_{\phi}, \cdot) =0$. We apply this to a path $J_t\in \AGK_F^G$: by the definition of  momentum  map  we  have
\[
    \frac{d}{dt}\int_M \phi\, {\Gscal}_{(F, J_t)} dV_F = \boldsymbol{\Omega}_J({\bf Y}_{\phi}, \dot J) =0.
\]
This shows that the $L^2(M,dV_F)$ projection  ${\Gscal}_{(F, J_t)}$ onto the space of normalized Hamiltonian potentials of $\mathfrak{g}$ is constant.
\end{proof}

Now we will change the point of view by fixing the \textit{holomorphic} data $(J,\poiss_J)$ and varying $F\in\GK_{\pi,\alpha}$. Suppose $(F,J)$ is an extremal generalized K\"ahler structure. Then by Theorem~\ref{t:Calabi-Lichne-Matsushima} the extremal vector field $\chi=-F^{-1}(d\Gscal_{(F,J)})$ is contained in the Lie algebra $\tor$ of a maximal torus $\T\subset  \mathrm{Aut}_{\mathrm{red}}(J,\poiss_J)\cap\mathrm{Ham}(M,F)$. It turns out that by the means of Proposition~\ref{p:futkai_symplectic}, vector field $\chi$ can be defined intrinsically from any $\T$-invariant generalized K\"ahler structure $F_0$. In particular, $\chi$ being nonzero provides an obstruction for the existence of $\T$-invariant cscGK structures in $\GK_{\poiss,\alpha}$.

\begin{thm}[Extremal vector field]\label{t:GK-extremal}
    Given a torus $\T\subset \mathrm{Aut}_{\mathrm{red}}(J,\poiss_J)$ let $\GK_{\poiss,\alpha}^\T$ be the space of $\T$-invariant generalized K\"ahler structures. Then, for any  $F_0\in \GK_{\poiss,\alpha}^\T$, $\T\subset \mathrm{Ham}(M,F_0)$ and we denote by 
    $\Pi_{F_0}\colon C^\infty(M,\R)\mapsto C^\infty(M,\R)$ the $L^2(M,dV_F)$ projection onto the space of $F_0$-Hamiltonians of $\T$.
    Moreover, the vector field
    \begin{equation}\label{e:extremal_vf_i2}
        \chi=-F^{-1}_0(d\Pi_{F_0}(\Gscal_{(F_0,J)}))
    \end{equation}
    is independent of the choice of a $\T$-invariant symplectic form $F_0\in \GK_{\poiss,\alpha}^\T$.  If, furthermore, $\T$ is a maximal torus in $\mathrm{Aut}_{\mathrm{red}}(J,\poiss_J)$,  then $F\in\GK_{\poiss,\alpha}^\T$ is an extremal structure if and only if
    \[
        \Gscal_{(F,J)}-\Pi_F(\Gscal_{(F,J)})=0.
    \]
    In particular,  the underlying extremal vector field $-F^{-1}(d\Gscal_{(F,J)})$ is necessarily given by~\eqref{e:extremal_vf_i2}.
\end{thm}
\begin{proof}
    First we prove that $\T\subset \mathrm{Aut}_{\mathrm{red}}(J,\poiss_J)$ acts in a Hamiltonian fashion. Indeed, any vector $X\in\hred(J,\poiss_J)$ in its Lie algebra has a decomposition $X=F_0^{-1}(d\phi)+F_0^{-1}(Id\psi)$, see Definition~\ref{d:hred}. Since $F_0$ is $\T$-invariant, we have $\mathcal L_X F_0=0$, which implies $dd^c_I\psi=0$, so $\psi=\mathrm{const}$ and  $X$ is a Hamiltonian vector field.

    For the claim about the constancy of $\chi$ on $\GK_{\poiss,\alpha}^\T$, take another $F_1\in \GK_{\poiss,\alpha}^\T$. Then applying an equivariant Moser isotopy we find a $\T$-equivariant diffeomorphism $\Phi\colon M\to M$ such that $\Phi^*(F_1)=F_0$. Then we have $J, \Phi^*(J)\in \AGK_{F_0}^\T$, and we can apply Proposition~\ref{p:futkai_symplectic} to conclude that
    \[
        \chi=
            -F^{-1}_0(d\Pi_{F_0}(\Gscal_{(F_0,J)}))=
            -F^{-1}_0(d\Pi_{F_0}(\Gscal_{(F_0,\Phi^*(J))}))
    \]
    Since $\Phi$ is $\T$-equivariant, we observe that $(\Phi^*)\chi=\chi$ so that applying $(\Phi^*)^{-1}$ to the last term we deduce
    \[
        -F^{-1}_0(d\Pi_{F_0}(\Gscal_{(F_0,J)}))=
        -F^{-1}_1(d\Pi_{F_1}(\Gscal_{(F_1,J)}))
    \]
    as claimed.

    Finally, assume that $\T$ is a maximal torus in $\mathrm{Aut}_{\mathrm{red}}(J,\poiss_J)$ and $F\in\GK_{\poiss,\alpha}^\T$ is an extremal generalized K\"ahler structure with the extremal vector field
    \[
    \chi_{\mathrm{ext}}=-F(d\Gscal_{(F,J)}).
    \]
    Since $(F,J)$ and hence $\chi_{\mathrm{ext}}$ is invariant under  $\T$,  by the maximality of $\T$ and Theorem~\ref{t:Calabi-Lichne-Matsushima} we have $\chi_{\mathrm{ext}}\in\tor$, so that  $\Pi_{F}\Gscal_{(F,J)}=\Gscal_{(F,J)}$ and $\chi_{\mathrm{ext}}=\chi$. Conversely, if $\Pi_{F}\Gscal_{(F,J)}=\Gscal_{(F,J)}$, then $\chi=-F^{-1}(d\Gscal_{(F,J)})\in\tor\subset \hred(J,\poiss_J)$ is holomorphic, and $(F,J)$ is extremal.
\end{proof}

\subsection{Futaki character}

In this section we introduce a generalized K\"ahler analogue of the Futaki character on $\hred(J,\poiss_J)$. Recall that given a generalized K\"ahler structure $F\in\GK_{\poiss,\alpha}$ we have a decomposition (see~\ref{d:hred}) of a holomorphic vector field  $X\in\hred(J,\poiss_J)$
\[
X=F^{-1}(d\phi_{(X,F)})+F^{-1}(Id\psi_{(X,F)}),\quad \phi_{(X,F)},\psi_{(X,F)}\in \CCF.
\]

\begin{thm}[Futaki character]\label{t:futaki-character}
    Let $F\in\GK_{\poiss,\alpha}$ be a generalized K\"ahler structure on a $(M,J,\poiss_J)$. Define a homomorphism $\mathcal F_{(F,J)}\colon \hred(J,\poiss_J)\to \R$ by
    \begin{equation}\label{e:futaki-character}
        \mathcal F_{(F,J)}(X)=\int_M \psi_{(X,F)}\Gscal_{(F,J)}dV_F
        =-\boldsymbol{\tau}(\mathbf X_{\psi_{(X,F)}}).
    \end{equation}
    Then $\mathcal F_{(F,J)}$ is independent of $F\in\GK_{\poiss,\alpha}^0$ and vanishes on the commutator $[\hred(J,\poiss_J),\hred(J,\poiss_J)]$.
    In particular $\mathcal F_{(F,J)}$ is a character of $\hred(J,\poiss_J)$ and is identically zero if $\GK_{\poiss,\alpha}^0$ admits a cscGK metric.
\end{thm}

\begin{proof}
The infinitesimal action of the vector field $X\in\hred(J,\poiss_J)$ on $F\in\GK_{\poiss,\alpha}$ is given by $\mathcal L_XF=dd^c_I\psi_{(X,F)}$ thus it induces a vector field ${\mathbf X}$ on $\GK_{\poiss,\alpha}$, which at a point $F\in\GK_{\poiss,\alpha}$ is given by
\[
{\mathbf X}(F)={\mathbf X}_{\psi_{(X,F)}}.
\]
Now, the 1-form $\boldsymbol{\tau}$ on $\GK_{\poiss,\alpha}^0$ is invariant under the vector field $\mathbf{X}$, since it is induced by an action of $\mathrm{Diff}(M)$ on the entire structure $(F,J)$. Thus by Cartan's formula we have
\[
0=d(\boldsymbol{\tau}(\mathbf X))+d\boldsymbol\tau({\mathbf X},\cdot).
\]
Since $\boldsymbol{\tau}$ is closed, it follows that $\boldsymbol{\tau}(\mathbf X)$ is constant on $\GK_{\poiss,\alpha}^0$ which is equivalent to the statement of the constancy of $\mathcal F_{(F,J)}$ on $\GK_{\poiss,\alpha}^0$.

Now, if $Y\in\hred(J,\poiss_J)$ is another holomorphic vector field inducing a vector field $\mathbf Y$ on $\GK_{\poiss,\alpha}$, then
\[
\mathcal F_{(F,J)}([X,Y])=-\boldsymbol\tau([\mathbf X,\mathbf Y])=d\boldsymbol{\tau}(\mathbf X,\mathbf Y)+\mathbf{Y}\cdot \boldsymbol{\tau}(\mathbf X)-\mathbf{X}\cdot \boldsymbol{\tau}(\mathbf Y)=0,
\]
as claimed, so that $\mathcal F_{(F,J)}$ is a character of Lie algebra $\hred(J,\poiss_J)$.
\end{proof}

\subsection{Linearization of the normalized generalized scalar curvature} 

We present here the linearization of ${\Gscal}_{(F_t, J)}$ when $F_t$ varies within a given generalized K\"ahler class $\GK_{\poiss,\alpha}$, similar to \cite[Sec.\,6]{GotoLichne}.  Motivated by Theorem~\ref{t:GK-extremal} we will, more generally, fix a compact torus $\T \subset \Aut_{\mathrm{red}}(M, J, \poiss_J)$ and consider the space $\mathcal {GK}_{\poiss,\ga}^{\T}$ of $\T$-invariant elements of $\mathcal{GK}_{\poiss,\ga}$, and define a notion of $\T$-\emph{normalized} scalar curvature:

\begin{defn} For any $F\in {\GK}^{\T}_{\poiss,\alpha}$, denote by $\tor_F$ the vector space of smooth functions $f$ such that $-F^{-1}(df) \in \tor := {\rm Lie}(\T)$ and by $\Pi_{F}$ the $L^{2}(M, dV_F)$-orthogonal projection of $C^{\infty}(M,\R)$ to $\tor_F$. Then
the \emph{$\T$-normalized generalized scalar curvature} is
\[ \Gscal^{\T}_{(F, J)}:= \Gscal_{(F, J)} - \Pi_F\left(\Gscal_{(F, J)}\right).\]
\end{defn}

\begin{lemma}\label{GK-dot-scal} Let $F_t$ be a smooth path of generalized K\"ahler structures in ${\GK}^{\T}_{\poiss,\alpha}$, corresponding to a Hamiltonian  deformation with  a $\T$-invariant function $\phi \in \left(C^\infty(M,\R)\right)^\T$. Then, for any smooth function $\psi \in \left(C^{\infty}(M,\R)\right)^{\T}$
\[
\begin{split}
\int_M  \psi \left(\left. \frac{d}{dt} \right |_{t=0} \Gscal^{\T}_{(F_t, J)}\right)F_0^{[n]} =& -{\bf g}_{J}\left({\mathcal L}_{F_0^{-1}(d\phi)}J, {\mathcal L}_{F_0^{-1}(d\psi)}J\right) \\
& + \int_M \psi \Big(\tr_F\left(dd^c_J \phi  - dd^c_I \phi + d\Gscal^{\T}_{(F_0, J)} \wedge Jd\phi \right)\Big) F_0^{[n]} ,
\end{split}\]
where ${\bf g}_J$ is defined in Lemma~\ref{l:symplectic}.
\end{lemma}
\begin{proof} Notice that, by the assumption for $\T$ with respect to $F_0$,  for any $F\in {\GK}^{\T}_{\poiss,\alpha}$ we have a Lie algebra isomorphism
\[ \left( \tor_F /\R, \{\cdot, \cdot \}_F \right)  \simeq (\tor, [\cdot, \cdot ]), \]
i.e., each vector field $Y \in \tor$ is $F$-Hamiltonian.  Denote by $\mu_{F_0} : M \to \tor^*$ the momentum map of $\T$ with respect to $F_0$, with momentum image
a polytope $\Pol$.  Using a $\T$-equivariant Moser isotopy with respect to  $F\in {\GK}^{\T}_{\poiss,\alpha}$ (notice that $(1-t)F_0 + t F$ defines a $\T$-invariant isotopy of symplectic forms as they tame $J$ for all $t$), it follows that there is a uniquely determined  momentum map $\mu_F : M \to \Pol$.  Furthermore, the $\Pi_{F_0}$-projection of $\Gscal_{(F_0, J)}$ is of the form
$\mu_{F_0}^* \ell$, where $\ell$ is an affine-linear function on $\tor^*$. By the $\T$-equivariant Moser lemma and Theorem~\ref{t:GK-extremal}, we have
$\Pi_F(\Gscal_{(F, J)}) = \mu_F^*(\ell)$ for the same affine-linear function $\ell$.
We now use a slight modification of the computation in the  proof of Proposition~\ref{p:Mabuchi}: letting $\psi_t:= \Phi_t^*(\psi)$ and $J_t:= (\Phi_t)_* J (\Phi_t)_*^{-1}$ where $\Phi_t$ is the isotopy defined in Lemma~\ref{l:complex-orbit}, we have
\[
\begin{split}
\int_M & \left(\left. \frac{d}{dt} \right |_{t=0} \Gscal^{\T}_{(F_t, J)}\right) \psi F_0^{[n]}  = \left. \frac{d}{dt} \right |_{t=0} \int_M \Gscal^{\T}_{(F_t, J)} \psi F_t^{[n]} - \int_M  \Gscal^{\T}_{(F_0, J)}  \psi dd^c_I \phi \wedge F_0^{[n-1]}  \\
 =&\ \left. \frac{d}{dt} \right |_{t=0} \int_M \Gscal^{\T}_{(F_0, J_t)} \psi_t F_0^{[n]} - \int_M  \Gscal^{\T}_{(F_0, J)}  \psi dd^c_I \phi \wedge F_0^{[n-1]}  \\
 =&\ \left. \frac{d}{dt} \right |_{t=0} \int_M \left(\Gscal_{(F_0, J_t)}- \mu_{F_0}^*(\ell)\right) \psi_t F_0^{[n]} - \int_M  \Gscal^{\T}_{(F_0, J)}  \psi dd^c_I \phi \wedge F_0^{[n-1]} \\
 =&\ -{\bf g}_{J}\left({\mathcal L}_{F_0^{-1}(d\phi)}J, {\mathcal L}_{F_0^{-1}(d\psi)}J\right)  + \tfrac{1}{2} \int_M  \Gscal^{\T}_{(F_0, J)}\left\langle d\phi, J(I+J) d\psi \right\rangle_{g_0} F_0^{[n]} \\
 &\  - \int_M  \Gscal^{\T}_{(F_0, J)}  \psi dd^c_I \phi \wedge F_0^{[n-1]}.
 \end{split}
\]
Notice that  (cf.\,\eqref{a:deep1}, \eqref{a:identity})
\[\tfrac{1}{2}\left\langle d\phi, J(I+J) d\psi \right\rangle_{g_0} = \tfrac{1}{2}\left\langle d\psi, (I+J) J d\phi \right\rangle_{g_0} =-\tr_{F_0} (d\psi \wedge Jd\phi).\]
Substituting in the previous identity and integrating by parts yields the claim. \end{proof}

\begin{cor}\label{c:GK-dot-scal}  Suppose $(F_0, J)$ is an extremal generalized K\"ahler structure and let $\T$ be  a torus  in ${\rm Ham}(M, F_0)\cap \Aut(M, J)$ containing the exponential of the extremal vector field $\chi=-F_0^{-1}(d\Gscal_{(F,J)})$. Then the linearization of $\Gscal^{\T}_{(F, J)}$  at $F_0$ in the space ${\GK}^{\T}_{\poiss,\alpha}$ is a fourth-order  linear operator on $\left(C^{\infty}(M,\R)\right)^{\T}/\R$, whose kernel is the space of $F_0$-Hamiltonian holomorphic vector fields on $(M, J)$ commuting with $\chi$.
\end{cor}
\begin{proof}  By the definition of $\T$-normalized scalar curvature, if $(F_0, J)$ is extremal then $\Gscal^{\T}_{(F_0, \T)}=0$.
The first part of the claim then follows from Lemmas~\ref{GK-dot-scal}  and \ref{l:symplectic}, noticing that (after integrating by parts) the RHS of the expression in Lemma~\ref{GK-dot-scal} vanishes for $\psi=\mathrm{const}$.  Furthermore, to identify the kernel of the linearization, denoted here $T_0$, we note using \eqref{e:by-parts}
\[ \int_M T_0(\psi) \psi F_0^{n}= -{\bf g}_J \left({\mathcal L}_{F_0^{-1}(d\psi)}J, {\mathcal L}_{F_0^{-1}(d\psi)}J\right),\]
and the claim follows. \end{proof}

\section{Formal Riemannian structure on $\mathcal {GK}_{\pi,\ga}$ and uniqueness}
\subsection{The Riemannian metric on $\GK_{\poiss,\alpha}$ and geodesics}

We now introduce a formal Riemannian metric on the space $\GK_{\poiss,\alpha}$, generalizing the Mabuchi--Semmes--Donaldson Riemannian structure~\cite{Mabuchi,Semmes,donaldson-GIT} on $\KK_{\alpha}$.  Recall that at any $F\in \GK_{\poiss,\alpha}$, we showed that the tangent space ${\bf T}_F \left(\GK_{\poiss,\alpha}\right)$ is identified with $C^{\infty}(M,\R)/\R$, see Remark~\ref{d:O}.  For any given $F\in \GK_{\poiss,\alpha}$, we will further identify $C^{\infty}(M,\R)/\R$ with the space of \emph{F-normalized} smooth functions,
\[
{\bf T}_F\left(\GK_{\poiss,\alpha}\right) \simeq \CCF := \left\{ \phi \in C^{\infty}(M,\R) \, \big| \, \int_M \phi F^{[n]} =0\right\}.
\]
We can then define a (formal) Riemannian metric $\llangle \cdot, \cdot \rrangle$ on $\GK_{\poiss,\alpha}$, by letting
\[\big\llangle \phi_1, \phi_2 \big\rrangle_F := \int_M \phi_1\phi_2 F^{[n]}, \qquad \phi_1, \phi_2 \in \CCF.\]
We note by Stokes Theorem that at any fixed point $F$, the inner product $\llangle \cdot, \cdot \rrangle_F$ is ${\rm ad}$-invariant with respect the Poisson bracket of $F$, i.e., satisfies
\begin{equation}\label{ad-invariance}
\Big\llangle \{\phi_1, \phi_3\}_F, \phi_2 \Big\rrangle_F + \Big\llangle \phi_1, \{\phi_2, \phi_3\}_F \Big\rrangle_F =0, \qquad \forall \phi_1, \phi_2, \phi_3 \in \CCF.\end{equation}

\begin{lemma}\label{l:connection} Let $F\to \bxi_F \in {\bf T}_F \left(\GK_{\poiss,\alpha}\right)= \CCF$ be a formal vector field on $\GK_{\poiss,\alpha}$. Then the Riemannian structure $\llangle \cdot, \cdot \rrangle$ on ${\GK}$ admits a unique Levi--Civita connection ${\mathcal D}$, defined by
\begin{gather} \label{f:LC}
    ({\mathcal D}_{\phi} \bxi)_F = \dot{\bxi}_F (\phi) - {\rm tr}_F\left(d\phi\wedge \Jj d\bxi_F\right), \qquad \phi \in \CCF={\bf T}_F\left(\GK_{\poiss,\alpha}\right),
\end{gather}
where $\dot{\bxi}_F(\phi) := \left. \frac{d}{dt} \right |_{t=0}\bxi_{F_t}$, with $F_t$ being the Hamiltonian deformation of $F$ defined by $\phi$.
\end{lemma}

\begin{proof}
We will first establish (\ref{f:LC}) for a fundamental vector field $\bxi = {\bf X}_{\psi}$ (cf.\,Definition~\ref{d:fundamental-field}). In what follows, we assume (without loss) that $\int_M F^{[n]}=1$ for  any $F \in \GK_{\poiss,\alpha}$. We will identify $\mathbf{X}_\psi$ at $F\in \GK_{\poiss,\alpha}$ with the normalized function given by $\bxi_F = \psi - \int_M \psi F^{[n]}$

Letting ${\mathcal D}$ denote the formal connection which preserves $\llangle \cdot, \cdot \rrangle$ and is torsion free and we will compute it using the Koszul's formula for fundamental vector fields $\bxi_i = {\bf X}_{\phi_i}, \, i=1,2,3$. Applying Lemma~\ref{l:commuting} we have
\begin{equation}\label{e:koszul}
\begin{split}
2\llangle {\mathcal D}_{\bxi_1} \bxi _2, \bxi_3 \rrangle =&\ {\bf X}_{\phi_1} \cdot \llangle {\bf X}_{\phi_2}, {\bf X}_{\phi_3} \rrangle  + {\bf X}_{\phi_2}\cdot \llangle {\bf X}_{\phi_1}, {\bf X}_{\phi_3}\rangle - {\bf X}_{\phi_3}\cdot \llangle {\bf X}_{\phi_1},  {\bf X}_{\phi_2}\rrangle \\
&\ - \llangle {\bf X}_{\{\phi_1, \phi_2\}_{\nonhalf\poiss}}, {\bf X}_{\phi_3}\rrangle + \llangle {\bf X}_{\{\phi_2, \phi_3\}_{\nonhalf\poiss}}, {\bf X}_{\phi_1}\rrangle
+ \llangle {\bf X}_{\{\phi_1, \phi_3\}_{\nonhalf\poiss}}, {\bf X}_{\phi_2}\rrangle.
\end{split}\end{equation}
We compute at $F$ (assuming without loss that $\int_M \phi_i F^{[n]}=0$):
\[
\begin{split}
{\bf X}_{\phi_1} \cdot \llangle {\bf X}_{\phi_2}, {\bf X}_{\phi_3} \rrangle &= \int_M \phi_2\phi_3 (dd^c_I \phi_1)\wedge F^{[n-1]}, \\
{\bf X}_{\phi_2}\cdot \llangle {\bf X}_{\phi_1}, {\bf X}_{\phi_3} \rrangle  &= \int_M \phi_1 \phi_3 (dd^c_I \phi_2)\wedge F^{[n-1]}, \\
{\bf X}_{\phi_3}\cdot \llangle {\bf X}_{\phi_1},  {\bf X}_{\phi_2}\rrangle  &=  \int_M \phi_1\phi_2 (dd^c_I \phi_3)\wedge F^{[n-1]} = \int_M \phi_3 \left(dd^c_\Jj (\phi_1\phi_2)\right) \wedge F^{[n-1]},
\end{split}\]
where for the last line we have used \eqref{e:by-parts}. It follows that
\[
\begin{split}
&{\bf X}_{\phi_1} \cdot \llangle {\bf X}_{\phi_2}, {\bf X}_{\phi_3} \rrangle   + {\bf X}_{\phi_2}\cdot \llangle {\bf X}_{\phi_1}, {\bf X}_{\phi_3}\rrangle - {\bf X}_{\phi_3}\cdot \llangle {\bf X}_{\phi_1},  {\bf X}_{\phi_2}\rrangle   \\
&= \int_M \phi_3\Big(\phi_2 dd^c_\Ii \phi_1 + \phi_1 dd^c_\Ii \phi_2 -dd^c_\Ii(\phi_1\phi_2) +  d(\Ii-\Jj) d (\phi_1\phi_2)\Big)F^{[n]} \\
&=\int_M \phi_3\Big(-d\phi_2 \wedge \Ii d\phi_1 - d\phi_1 \wedge \Ii d\phi_2 + d(\Ii-\Jj)d(\phi_1\phi_2)\Big)\wedge F^{[n-1]} \\
& =\int_M \phi_3\Big((I+ J) d\phi_1 \wedge d\phi_2 + d(\Ii-\Jj)d(\phi_1\phi_2)\Big)\wedge F^{[n-1]} \\
&=\int_M \phi_3\Big((I+ J) d\phi_1 \wedge d\phi_2 + \phi_2 d(\Ii-\Jj)(d\phi_1) + \phi_1 d(\Ii-\Jj) d\phi_2 \Big)\wedge F^{[n-1]},
\end{split}\]
where for passing from the 4th line to the 5th and from the 5th to the 6th we have used \eqref{f:Fvolumeidentity}.

Similarly, using \eqref{f:Fvolumeidentity} and \eqref{a:poisson} we compute
\[
\begin{split}
\big\llangle {\bf X}_{\{\phi_1, \phi_2\}_{\nonhalf\poiss}}, {\bf X}_{\phi_3}\big\rrangle &= \int_M \phi_3\Big((\Jj-\Ii) d\phi_1 \wedge d\phi_2\Big)\wedge F^{[n-1]}, \\
\big\llangle {\bf X}_{\{\phi_2, \phi_3\}_{\nonhalf\poiss}}, {\bf X}_{\phi_1}\big\rrangle &= \int_M \phi_1\Big((\Jj-\Ii) d\phi_2 \wedge d\phi_3\Big)\wedge F^{[n-1]} \\
&= - \int_M \phi_3\Big(d\phi_1 \wedge (\Ii-\Jj) d\phi_2 + \phi_1 d(\Ii-\Jj)(d\phi_2)\Big)\wedge F^{[n-1]} \\
\big\llangle {\bf X}_{\{\phi_1, \phi_3\}_{\nonhalf\poiss}}, {\bf X}_{\phi_2}\big\rrangle &= \sj \int_M \phi_3\Big(d\phi_2 \wedge (\Ii -\Jj) d\phi_1 + \phi_2 d(I-J) d\phi_1 \Big)\wedge F^{[n-1]}.
\end{split}\]
Using \eqref{f:Fvolumeidentity}, this yields
\[
\begin{split}
\big\llangle & {\bf X}_{\{\phi_1, \phi_2\}_{\nonhalf\poiss}}, {\bf X}_{\phi_3}\big\rrangle -\big\llangle {\bf X}_{\{\phi_2, \phi_3\}_{\nonhalf\poiss}}, {\bf X}_{\phi_1}\big\rrangle -\big\llangle {\bf X}_{\{\phi_1, \phi_3\}_{\nonhalf\poiss}}, {\bf X}_{\phi_2}\big\rrangle\\
&= \int_M \phi_3\Big((J - I)(d\phi_1)\wedge d\phi_2 + \phi_1 d(I-J) d\phi_2 + \phi_2 d(I-J d\phi_1\Big)\wedge F^{[n-1]}.
\end{split} \]
Substituting the above expressions back in \eqref{e:koszul} and using \eqref{f:Fvolumeidentity} and that $\int_M \phi_3 F^{[n]} = 0$, we get \begin{equation}\label{e:fundamental}
    \begin{split}
    \left({\mathcal D}_{\phi} \bxi\right)_F & = - {\rm tr}_F(d\phi\wedge \Jj d\psi) -\int_M \psi (dd^c_\Ii \phi)\wedge F^{[n-1]}  \\
    & = - (d\phi\wedge \Jj d\psi)\wedge F^{[n-1]}/ F^{[n]} -\int_M \psi (dd^c_\Ii \phi)\wedge F^{[n-1]},
    \end{split}
\end{equation}
which is equivalent to~\eqref{f:LC} for the fundamental vector fields. Notice that the RHS of~\eqref{e:fundamental} integrates to zero against $F^{[n]}$ by~\eqref{e:by-parts}.

We now show that \eqref{f:LC} holds for a general vector field $F \to \bxi_F$.  Using metric compatibility and \eqref{e:fundamental} we have 
\[
\begin{split}
\big\llangle {\mathcal D}_{\phi} \bxi, {\bf X}_{\psi}\big\rrangle_F &= \left({\bf X}_{\phi} \cdot \llangle \bxi, {\bf X}_{\psi}\rrangle\right)_F - \big\llangle \bxi_F, ({\mathcal D}_{\phi} {\bf X}_{\psi})_F\big\rrangle_F \\
&= \left. \frac{d}{dt} \right |_{t=0} \int_M \bxi_{F_t} \psi F_t^{[n]}  + \int_M \bxi_F\left(d\phi \wedge \Jj d\psi\right) \wedge F^{[n-1]} \\
&= \int_M \dot \bxi_F(\phi) \psi F^{[n]} + \int_M \bxi_F \psi (dd^c_\Ii \phi) \wedge F^{[n-1]} +\int_M \bxi_F\left(d\phi \wedge \Jj d\psi\right)\wedge F^{[n-1]}  \\
&=  \int_M \Big(\dot \bxi_F(\phi) - {\rm tr}_F\big(d\phi \wedge \Jj d\bxi_F\big) \Big) \psi F^{[n]}.
\end{split}\]
where for final line we applied \eqref{f:Fvolumeidentity} and integration by parts.
\end{proof}

From Lemma~\ref{l:connection}, we derive the corresponding geodesic equation which extends the expression found by Mabuchi~\cite{Mabuchi} in the K\"ahler case.

\begin{defn}\label{d:geodesic} Let $F_t$ be a smooth path in $\GK_{\poiss,\alpha}$, corresponding to a Hamiltonian deformation with respect to a path of smooth functions $\phi_t \in C^{\infty}_{0}(M,dV_{F_t})$.  We say that $F_t$ is a geodesic if $\phi_t$ satisfies
\begin{equation}\label{e:geodesic}
\begin{split}
0 =&\ \mathcal D_{\phi_t} \phi_t = \dot \phi_t - {\rm tr}_{F_t}(d\phi_t \wedge \Jj d\phi_t) = \dot{\phi}_t - \tfrac{1}{2}\Big(g_t(d\phi_t, d\phi_t) + g_t(\Ii_t d\phi_t, \Jj d\phi_t)\Big).
\end{split}
\end{equation}
\end{defn}

\begin{rmk}\label{r:geodesic-normalization} One can  more generally consider smooth solutions $\phi_t$ of \eqref{e:geodesic},  without assuming a priori that $\phi_t \in C^{\infty}_{0}(M,dV_{F_t})$.
It then follows using \eqref{e:by-parts} that
\[ \frac{d}{dt}\int_M \phi_t F_t^{[n]} = \int_M (d\phi_t \wedge \Jj d\phi_t) \wedge F^{[n-1]} + \int_M \phi_t d \Ii_t d\phi_t \wedge F^{[n-1]}= 0.\]
The above shows that $\phi_t \in C^{\infty}_0(M, dV_{F_t})$ as soon as $\phi_0 \in C^{\infty}_{0}(M, dV_{F_0})$.
\end{rmk}

\begin{rmk} \label{r:Semmes} Let $F_t$ be a smooth path in $\GK_{\poiss,\alpha}$, corresponding to a Hamiltonian deformation with respect to a path of smooth functions $\phi_t \in C^{\infty}_{0}(M,dV_{F_t})$, $t \in [0,1]$.  Let $\til{M} = M \times [0, 1] \times [0,1]$, and denote the generic point in $\til{M}$ as $(p, t, s)$.  Define
\begin{align*}
    \til{F}_{(p,t,s)} = F_{(p,t)} - d \phi_t \wedge ds + J d \phi_t \wedge dt + \dot{\phi}_t dt \wedge ds.
\end{align*}
Then the geodesic equation is equivalent to $\til{F}^{n+1} = 0$.
\end{rmk}

\begin{prop} \label{p:killing-geodesics} Let $(F_0, J)$ be a symplectic type generalized K\"ahler structure and $Y$ a vector field preserving $(F_0, J)$ which is also Hamiltonian with respect to $F_0$, i.e., there exists a smooth function $\phi_0$ such that
\[
    Y= - F_0^{-1} (d\phi_0)= \tfrac{1}{2}\left(\Ii_0\grad_{g_0} \phi_0 + \Jj \grad_{g_0} \phi_0\right), \qquad \int_M \phi_0 F_0^n =0.
\]
Then the flow $\Phi_t= \exp(-t\Jj Y)$ of $-\Jj Y$ defines a geodesic $F_t := \Phi_t^*(F_0)$ in $\GK_{\poiss,\alpha}$.
\begin{proof}
Indeed, we have that
\[
\begin{split}
\dot F_t  &= \Phi_t^*\left({\mathcal L}_{-\Jj Y}F_0 \right) = \Phi_t^*\left(d \left(F_0\Jj F_0^{-1}(d\phi_0)\right) \right)\\
              &=\Phi_t^*\left(-d \Ii_0^* (d\phi_0)\right) = \Phi_t^*\left(dd^c_{\Ii_0} \phi_0\right) = dd^c_{\Ii_t} \phi_t,
              \end{split}\]
where $\phi_t:= \Phi_t^*(\phi_0)$ satisfies $\int_M \phi_t F_t^{n}=0$. 
Furthermore, as $\Phi_t \cdot Y= Y$, we have $Y= - F_t^{-1}(d\phi_t)$, so we compute
\[
    \frac{d}{d t} \phi_t = {\mathcal L}_{-\Jj Y} \phi_t = \langle d\phi_t, -\Jj Y \rangle =   \langle -\Jj d\phi_t, F_t^{-1}(d\phi_t) \rangle= \tfrac{1}{2}\big(g_t(d\phi_t, d\phi_t) + g_t(\Jj d\phi_t, \Ii_t d\phi_t)\big),  \]
which is precisely the geodesic equation.
\end{proof}
\end{prop}

\subsection{Curvature}

We next compute the curvature of the formal Riemannian connection ${\mathcal D}$. To this end, it is enough to consider fundamental vector fields $\bxi_i= {\bf X}_{\phi_i}, \, i=1, 2, 3, 4$, see Definition~\ref{d:fundamental-field}, as they generate ${\bf T}_F \left(\GK_{\poiss,\alpha}\right)$ at any given point $F \in \GK_{\poiss,\alpha}$. The following is an extension of a result by Mabuchi~\cite{Mabuchi} to the symplectic type GK case.
\begin{thm}\label{p:Mabuchi-curvature} At any given point $F\in \GK_{\poiss,\alpha}$, the curvature tensor $\mathcal R$ of ${\mathcal D}$  is given by
\[\left( {\mathcal R}_{{\bf X}_{\phi_1}, {\bf X}_{\phi_2}} {\bf X}_{\phi_3}\right)_F =-\Big\{\big\{\phi_1,\phi_2\big\}_F, \phi_3\Big\}_F,  \]
where $\{\cdot, \cdot \}_{F}$ denotes the Poisson bracket of functions with respect to the symplectic form $F$. In particular,
\[\Big\llangle {\mathcal R}_{{\bf X}_{\phi_1}, {\bf X}_{\phi_2}} {\bf X}_{\phi_1}, {\bf X}_{\phi_2}\Big\rrangle_F =-\Big\llangle  \{\phi_1,\phi_2\big\}_F, \{\phi_1,\phi_2\big\}_F\Big\rrangle_F,\]
showing that $\llangle \cdot, \cdot \rrangle$ has nonpositive sectional curvature at any point.
\end{thm}
\begin{proof}  The second formula follows from the first by the ${\rm ad}$-invariance of $\big\llangle\cdot, \cdot \big\rrangle$ (cf.\,\eqref{ad-invariance}). Conversely, the first formula follows from the second as ${\mathcal R}$ is associated to a torsion-free Riemannian connection, and thus the sectional curvature determines the Riemannian curvature tensor.  It is thus enough to establish the second formula.

At a given point $F$, by Lemma~\ref{l:commuting} the curvature is given by
\begin{equation}\label{R}
\left( {\mathcal R}_{{\bf X}_{\phi_1}, {\bf X}_{\phi_2}} {\bf X}_{\phi_3} \right)_F = \left( -{\mathcal D}_{\phi_1}\left({\mathcal D}_{\phi_2} {\bf X}_{\phi_3}\right) + {\mathcal D}_{\phi_2}\left({\mathcal D}_{\phi_1} {\bf X}_{\phi_3}\right) \sj {\mathcal D}_{\{\phi_1,\phi_2\}_{\nonhalf\poiss}} {\bf X}_{\phi_3}\right)_F.\end{equation}
With a small abuse of notation, we will ignore the normalizing additive constants  of the smooth functions, as they do not contribute to the desired formula for ${\mathcal R}_{{\bf X}_{\phi_1}, {\bf X}_{\phi_2}} {\bf X}_{\phi_3}$.  We will also drop the dependence on the basepoint $F$ from the notation.  Thus from \eqref{e:fundamental} we can express the connection as\[ {\mathcal D}_{\phi}{\bf X}_{\psi} = -{\rm tr}_F(d\phi \wedge \Jj d\psi)=-\tfrac{1}{2}\Big(g(d\phi, d\psi) + g(Id\phi, \Jj d\psi)\Big),\]
where for the second equality we used the computation in \eqref{a:Fvolumeidentity}.
Letting $\dot{F} = dd^c_{\Ii} \phi_1$  be the derivative of $F$ in direction of ${\bf X}_{\phi_1}$, it follows  from Lemma~\ref{l:connection} and \eqref{f:Fvolumeidentity} that  (up to an additive constant)
\begin{equation}\label{e:intermidiate}
\begin{split}
{\mathcal D}_{\phi_1}\left({\mathcal D}_{\phi_2} {\bf X}_{\phi_3}\right) &=  \tfrac{1}{2}{\rm tr} \left(F^{-1} \dot{F}F^{-1} (d\phi_2 \wedge \Jj d\phi_3)\right)  + \tr_F\Big(d\phi_1 \wedge \Jj d(\tr_F(d\phi_2 \wedge \Jj d\phi_3))\Big) \\
&= \tfrac{1}{2}{\rm tr} \left(F^{-1} (dd^c_\Ii \phi_1) F^{-1} (d\phi_2 \wedge \Jj d\phi_3)\right)+ \tr_F\Big(d(\tr_F(d\phi_2\wedge \Jj d\phi_3)\wedge \Ii d\phi_1)\Big).
\end{split}\end{equation}
We will use the following algebraic identity which can be deduced easily from Schur's lemma and holds for any $2$-forms $\Phi, \Psi$: \begin{equation}\label{e:symplectic-schur} \Phi \wedge \Psi \wedge F^{[n-2]}= -\tfrac{1}{2}{\rm tr}\left(F^{-1}\Phi F^{-1}\Psi\right) F^{[n]}  + \left({\rm tr}_F \Phi\right)\left({\rm tr}_F \Psi\right)F^{[n]}.\end{equation}
By \eqref{e:symplectic-schur}, we get from \eqref{e:intermidiate} and integrating by parts
\[\begin{split}
\Big\llangle{\mathcal D}_{\phi_1}\left({\mathcal D}_{\phi_2} {\bf X}_{\phi_3}\right), {\bf X}_{\phi_4}\Big\rrangle
 =&\ -\int_M \phi_4\Big(dd^c_\Ii \phi_1 \wedge d\phi_2 \wedge \Jj d\phi_3\Big) \wedge F^{[n-2]} + \int_M \phi_4\Big(\tr_F(d\phi_2\wedge \Jj d\phi_3)\Big) dd^c_\Ii \phi_1 \wedge F^{[n-1]}\\
    &\ + \int_M\phi_4\Big(d(\tr_F(d\phi_2\wedge \Jj d\phi_3))\wedge \Ii d\phi_1\Big) \wedge F^{[n-1]} \\
   =&\ \int_M \Big(d\phi_4\wedge \Ii d \phi_1 \wedge d\phi_2 \wedge \Jj d\phi_3\Big) \wedge F^{[n-2]} + \int_M \phi_4\Big( \Ii d\phi_1 \wedge d\phi_2 \wedge dd^c_\Jj \phi_3 \Big)\wedge F^{[n-2]} \\
    &- \int _M \Big(\tr_F(d\phi_2\wedge \Jj d\phi_3) \tr_F(d\phi_4 \wedge \Ii d\phi_1)\Big) F^{[n]}.
\end{split}\]
It follows that
\begin{equation*}\label{R1}
\begin{split}
\Big\llangle -{\mathcal D}_{\phi_1}&  \left({\mathcal D}_{\phi_2} {\bf X}_{\phi_1}\right) + {\mathcal D}_{\phi_2}\left({\mathcal D}_{\phi_1} {\bf X}_{\phi_1}\right) , {\bf X}_{\phi_2}\Big\rrangle\\
=& \int_M \Big(d\phi_2 \wedge \Ii d\phi_2 \wedge d\phi_1 \wedge \Jj d\phi_1\Big)\wedge F^{[n-2]} - \int_M\Big(\tr_F(d\phi_1\wedge \Jj d\phi_1)\Big)\Big(\tr_F(d\phi_2 \wedge \Ii d\phi_2) \Big)F^{[n]} \\
&-\int_M \phi_2\Big( d\phi_1 \wedge \Ii d\phi_2 + \Ii d\phi_1 \wedge d\phi_2\Big)\wedge dd^c_\Jj \phi_1\wedge  F^{[n-2]} +\int_M\Big(\tr_F(d\phi_2\wedge \Jj d\phi_1)\Big)\Big( \tr_F(d\phi_2 \wedge \Ii d\phi_1)\Big) F^{[n]}. \end{split}
 \end{equation*}
Noting that from \eqref{a:poisson} we have
\[ \{\phi_1,\phi_2\}_{\nonhalf\poiss}  = \sj \tr_F(d\phi_1 \wedge \Ii d\phi_2) \msj \tr_F(d\phi_1 \wedge \Jj d\phi_2), \]
we further compute, after integrating by parts and using \eqref{a:Fvolumeidentity} for the last equality,
\begin{equation*}\label{R2}
\begin{split}
 \Big\llangle {\mathcal D}_{\{\phi_1,\phi_2\}_{\nonhalf\poiss}} {\bf X}_{\phi_1}, {\bf X}_{\phi_2}\Big\rrangle =& -\int_M \phi_2\Big(d\{\phi_1,\phi_2\}_{\nonhalf\poiss} \wedge \Jj d\phi_1\Big) \wedge F^{[n-1]} \\
 =  & \int_M \{\phi_1,\phi_2\}_{\nonhalf\poiss}  d\phi_2 \wedge \Jj d\phi_1 \wedge F^{[n-1]}  + \int_M \phi_2 \{\phi_1,\phi_2\}_{\nonhalf\poiss}  dd^c_\Jj \phi_1 \wedge F^{[n-1]} \\
 =&\sj \int_M\Big(\tr_F(d\phi_1\wedge \Ii d\phi_2) - \tr_F(d\phi_1\wedge \Jj d\phi_2)\Big)\tr_F(d\phi_2\wedge \Jj d\phi_1) F^{[n]} \\
 &\sj \int_M \phi_2\Big(\tr_F(d\phi_1\wedge \Ii d\phi_2) - \tr_F(d\phi_1\wedge \Jj d\phi_2)\Big)\tr_F(dd^c_\Jj\phi_1) F^{[n]} \\
=&\sj \int_M\Big(\tr_F(d\phi_2\wedge \Jj d\phi_1) - \tr_F(d\phi_2\wedge \Ii d\phi_1)\Big)\tr_F(d\phi_2\wedge \Jj d\phi_1) F^{[n]} \\
 &\sj \int_M \phi_2\Big(\tr_F(d\phi_1\wedge \Ii d\phi_2) + \tr_F(\Ii d\phi_1\wedge d\phi_2)\Big)\tr_F(dd^c_\Jj\phi_1) F^{[n]}.
\end{split}
\end{equation*}
Substituting the latter two expressions above back in \eqref{R},  and regrouping the terms, we get
\[
\begin{split}
&\Big\llangle {\mathcal R}_{{\bf X}_{\phi_1}, {\bf X}_{\phi_2}} {\bf X}_{\phi_1}, {\bf X}_{\phi_2}\Big\rrangle =  \int_M \Big(\tr_F(d\phi_2 \wedge \Jj d\phi_1)\Big)^2F^{[n]}\\
&+ \int_M \Big(d\phi_2 \wedge \Ii d\phi_2 \wedge d\phi_1 \wedge \Jj d\phi_1\Big)\wedge F^{[n-2]} - \int_M\Big(\tr_F(d\phi_1\wedge \Jj d\phi_1)\Big)\Big(\tr_F(d\phi_2 \wedge \Ii d\phi_2) \Big)F^{[n]} \\
&-\int_M \phi_2\Big(d\phi_1 \wedge \Ii d\phi_2 \wedge dd^c_\Jj \phi_1\Big)\wedge  F^{[n-2]} + \int_M \phi_2\big(\tr_F(d\phi_1 \wedge \Ii d\phi_2)\tr_F(dd^c_\Jj\phi_1)\Big)F^{[n]}\\
&-\int_M \phi_2\Big( \Ii d\phi_1 \wedge d\phi_2\wedge dd^c_\Jj \phi_1\Big)\wedge  F^{[n-2]} + \int_M \phi_2\Big(\tr_F(\Ii d\phi_1 \wedge d\phi_2)\Big)\tr_F(dd^c_\Jj\phi_1)F^{[n]}.\\
\end{split}
\]
We now apply \eqref{e:symplectic-schur} to each of the sums on the last three lines. The algebraic identity \eqref{a:deep4} shows that the last two lines cancel out whereas \eqref{a:deep5} allows us to simplify the expressions at the first and second lines to
\[
\begin{split}
&\Big\llangle {\mathcal R}_{{\bf X}_{\phi_1}, {\bf X}_{\phi_2}} {\bf X}_{\phi_1}, {\bf X}_{\phi_2}\Big\rrangle =  \int_M \Big(\tr_F(d\phi_2 \wedge \Jj d\phi_1)\Big)^2F^{[n]}\\
&+ \int_M \Big(d\phi_2 \wedge \Ii d\phi_2 \wedge d\phi_1 \wedge \Jj d\phi_1\Big)\wedge F^{[n-2]} - \int_M\Big(\tr_F(d\phi_1\wedge \Jj d\phi_1)\Big)\Big(\tr_F(d\phi_2 \wedge \Ii d\phi_2) \Big)F^{[n]} \\
&=\int_M \Big(\tr_F (d\phi_2 \wedge \Jj d\phi_1)\Big)^2F^{[n]} -\tfrac{1}{2}\int_M \tr \Big(F^{-1}(d\phi_2 \wedge \Ii d\phi_2) F^{-1}(d\phi_1 \wedge Jd\phi_1)\Big)F^{[n]} \\
&=\int_M \Big(\tr_F (d\phi_2 \wedge Jd\phi_1)\Big)^2F^{[n]}  -\int_M \Big(\tr_F(d\phi_1 \wedge d\phi_2)\Big)^2 F^{[n]} - \int_M\Big(\tr_F(d\phi_1 \wedge I d\phi_2)\Big)^2F^{[n]} \\
& = -\int_M \Big(\tr_F(d\phi_1 \wedge d\phi_2)\Big)^2 F^{[n]},
\end{split} \]
where we have used \eqref{a:Fvolumeidentity}  to obtain the last equality. The proposition is proved.
\end{proof}

\subsection{Formal uniqueness}

The formal moment map picture yields that if it exists, ${\mathbf M}_{F_0}$ is convex on geodesics.  More generally, the pairing of the $1$-form $\boldsymbol{\tau}$ with the velocity of a geodesic is monotone:

\begin{prop}\label{p:Mabuchi-convexity} Let $F_t$ be a geodesic in $(\GK_{\poiss,\alpha}, \llangle \cdot, \cdot \rrangle)$, corresponding to a time dependent smooth function $\phi_t$ satisfying \eqref{e:geodesic}.  Then
\[
    \frac{d}{dt} \boldsymbol{\tau} (\mathbf{X}_{\phi_t}) \geq 0,
\]
with equality for all $t$ if and only if $\phi_t$ is a geodesic induced by a Hamiltonian Killing field as in Proposition~\ref{p:killing-geodesics}.
\end{prop}
\begin{proof}
We suppose (without loss, see Remark~\ref{r:geodesic-normalization}) that $\phi_t$ is $dV_{F_t}$-normalized, i.e., $\int_M \phi_t F_t^{[n]} =0$. We then have to compute
\[-\frac{d}{dt} \int_M {\Gscal}_{(F_t, J)} \phi_t F_t^{[n]}.\]
It is enough to establish the positivity of the derivative at $t=0$. The computation at time $t_0$ will follow from the latter by considering the re-parametrized geodesic $\phi_{t+ t_0}$.  To this end, we apply the Moser isotopy $\Phi_t$ defined in Lemma~\ref{l:complex-orbit}.  In particular, we first claim that $\psi_t := \Phi_t^* (\phi_t) = \phi_0$.  Using the notation of Lemma \ref{l:complex-orbit} we compute
\begin{gather*}
\begin{split}
\frac{d}{dt} \psi_t =&\ \Phi_t^*\Big({\mathcal L}_{Z_t} \phi_t + \dot \phi_t \Big) =\Phi_t^*\Big(\big\langle -F_t^{-1}(I_td\phi_t), d\phi_t \big\rangle + \dot \phi_t\Big)\\
 =&\  \Phi_t^*\Big(\tr_{F_t} (I_t d\phi_t  \wedge d\phi_t) + \dot \phi_t\Big) =\Phi_t^*\Big(-\tr_{F_t} (d\phi_t  \wedge Jd\phi_t) + \dot \phi_t\Big)=0. \end{split}
 \end{gather*}
Using this fact we have
\[ \begin{split}
-\left. \frac{d}{dt} \right |_{t=0}\int_M {\Gscal}_{(F_t, J)} \phi_t F_t^{[n]} &= -\left. \frac{d}{dt} \right |_{t=0}\int_M  \Phi_t^*\Big({\Gscal}_{(F_t, J)} \phi_t F_t^{[n]}\Big) \\
&=- \left. \frac{d}{dt} \right |_{t=0} \int_M {\Gscal}_{(F_0, J_t)} \phi_0 F_0^{[n]} \\
&=- {\boldsymbol{\Omega}}_{J}\Big(\big(-{\mathcal L}_{Y_{\phi_0}} J\big), \dot{J} \Big) \\
&= {\boldsymbol{\Omega}}_{J}\Big(\big(-{\mathcal L}_{Y_{\phi_0}} J\big), J \big(-{\mathcal L}_{Y_{\phi_0}} J\big)\Big) \\
&= \Big\|{\mathcal L}_{Y_{\phi_0}} J\Big\|^2_{\boldsymbol{g}_J}.
\end{split}\]
In the above equalities, $Y_{\phi_0}:= - F_0^{-1}(d\phi_0)$  and for passing to the third line we have used that ${\Gscal}_{(F_0, J_t)}$ is a moment map for the action of ${\rm Ham}(M, F_0)$ on $\left({\AGK}_F,\boldsymbol{\Omega}\right)$,  for passing to the fourth line we have used Lemma~\ref{l:complex-orbit} to identify $\dot J$, and for passing to the last line we have used Lemma~\ref{l:complex}.

This yields the claimed inequality, and for equality to hold for all $t$, one must have that the vector field  $Y_t:=-F_t^{-1}(d\phi_t)$ is $J$-holomorphic.  As it preserves $F_t$, $Y_t$ thus preserves the whole biHermitian structure $(g_t, b_t, I_t, J, F_t)$, i.e., it is a Killing field with potential $\phi_t$. To obtain the claim, it is enough to show that $Y_t=Y$ is time independent, as the $F_t$-normalized potential of the Killing field $Y$ with respect to $F_t$ is unique by the maximum principle, and is given by the geodesic determined by $Y$ via Proposition \ref{p:killing-geodesics}. Below we check that $Y_t$ is time independent, by using the geodesic equation \eqref{e:geodesic}, expressed as $\dot \phi = \langle Y, Id\phi \rangle$:
\[
\begin{split}
\dot{Y} &= F^{-1} \dot F F^{-1} (d\phi) - F^{-1}(d\dot \phi) = F^{-1}\Big(- \imath_Y dd^c_I \phi -  d \imath_Y I\phi\Big)\\
            &= F^{-1}\Big(-{\mathcal L}_Y (Id\phi) \Big) =  F^{-1}\Big(-I{\mathcal L}_Y d\phi \Big) =0,
\end{split}
\]
where in the last line we used that $Y$ is Killing (and thus ${\mathcal L}_Y I=0$) and ${\mathcal L}_Y \phi = -\big\langle F^{-1}(d\phi), d\phi \big\rangle =0$.
\end{proof}

\begin{cor}[Conditional uniqueness]\label{c:conditional-uniqueness}
Suppose $F_0, F_1 \in \GK_{\poiss,\alpha}$ are cscGK structures connected by a smooth geodesic $F_t$.  Then there exists $Y\in\hred(J,\poiss_J)$ such that $F_t=\Phi_t^* F_0$, where $\Phi_t=\exp(-t J Y)\in \mathrm{Aut}_{\mathrm{red}}(J,\poiss_J)$.
\end{cor}
\begin{proof}
    Let $\phi_t\in C^\infty_0(M,\R), 0\leq t\leq 1$ be a path inducing a geodesic between $F_0$ and $F_1$. Since $F_0$ and $F_1$ are cscGK, we have $\boldsymbol{\tau}\big|_{F_0}=\boldsymbol{\tau}\big|_{F_1}=0$. By Proposition~\ref{p:Mabuchi-convexity} function $t\mapsto \boldsymbol{\tau}(\mathbf{X}_{\phi_t})$ is non-decreasing, and being 0 at $t=0$ and $t=1$, it must vanish identically. 
    Therefore by the second part of Proposition~\ref{p:Mabuchi-convexity} there exists $Y\in\hred(J,\poiss_J)\cap \mathfrak{ham}(M,F_0)$ such that the flow of $-JY$ induces the path $F_t$.
\end{proof}

\section{The toric case}~\label{s:toric}

In this section we consider the case when $(M, J, \T^{\C})$ is a (projective) smooth toric variety under the effective action of a complex  $n$-dimensional complex torus $\T^{\C}\simeq (\C^*)^{n}$. We denote by $\T$ the corresponding compact (real) $n$-dimensional torus and by $\tor$ its Lie algebra.   We will assume that $\poiss_J$ is a $\T^{\C}$-invariant holomorphic Poisson tensor on $(M, J)$, i.e., given by an element of $\Wedge^{2}(\tor\otimes \C)$ (see \cite[Prop.~2.14]{Hong}).

\subsection{Reduction to invariant structures} 

As a corollary of Theorem~\ref{t:Calabi-Lichne-Matsushima}  we have
\begin{cor}\label{c:T-reduction} Suppose $(M, J, \T^{\C})$ is a smooth projective toric variety and $\poiss_J$ a $\T^{\C}$-invariant holomorphic Poisson tensor.  Let $F_0$ be  a  $\T$-invariant symplectic type generalized K\"ahler structure in ${\GK}_{\poiss}$. Then, up to isometry in $\Aut_0(J, \poiss_J)$, any extremal generalized K\"ahler structure $F\in\GK_{\poiss}$ is $\T$-invariant.
\end{cor}
\begin{proof} By the toric assumption on $(M, J)$, we know that $b_1(M)=0$ and thus  $\Aut_{0}(J, \poiss_J)= \Aut_{\rm red}(J,\poiss_J)$ by Remark~\ref{r:b_1=0}. Note that $\T$ is a maximal compact torus in $\Aut_{0}(J, \poiss_J)$. By Theorem~\ref{t:Calabi-Lichne-Matsushima}, $(F, J)$ is invariant under a maximal torus $\T'$ in $\Aut_0(J, \poiss_J)$. As any two such tori are conjugated inside $\Aut_0(J, \poiss_J)$,  after acting with an element of $\Aut_0(J, \poiss_J)$, we can assume that $(F, J)$ is $\T$-invariant.
\end{proof}

\subsection{Abreu--Guillemin type description of toric structures} 

Because of  Corollary~\ref{c:T-reduction}, we focus from now on $\T$-invariant symplectic type generalized K\"ahler structures. The corresponding spaces will be denoted by upper script $\T$:   ${\GK}^{\T}_{\poiss}, (\GK_{\poiss,\alpha}^{\T})^{0}$,  ${\GK}_{\poiss, \alpha}^{\T}$, etc.   As $b_1(M)=0$,  for any $F\in {\GK}^{\T}_{\poiss}$ the action of $\T$ is $F$-Hamiltonian,  i.e., $(F, J)$ is  toric in the sense of  \cite{ASU,W2}. Toric symplectic type  generalized K\"ahler structures were extensively studied in \cite{boulanger,W1,W2},  where an Abreu--Guillemin type description~\cite{Abreu, guillemin} is obtained. We follow the notation of our previous work~\cite{ASU} where we have recast the classification of \cite{boulanger,W1,W2} in a  formalism  compatible with the one used in the current article. The following discussion is a brief summary of \cite[Sec.4]{ASU}. We refer the reader to this work for further details and literature review.

We thus assume that  $(M, g, J, I, F)$ is a compact symplectic type generalized K\"ahler structure which admits an effective isometric and $F$-Hamiltonian action of a compact torus $\T$ with $2\, {\rm dim}_{\R} \T ={\rm dim}_{\R} M=2n$. We denote by  $(\Pol, \Lab)$ the (labeled) Delzant polytope of $(M, F)$ in $\tor^*\simeq \R^n$,  and by $\mu=(\mu_1, \ldots, \mu_n) : M \to \Pol$  the corresponding momenta.  Here $\tor = {\rm Lie}(\T)$  is the Lie algebra of $\T$ and we use a lattice basis (and its dual basis) to identify respectively  $\tor$  and its dual vector space $\tor^*$ with $\R^n$;  in particular,   $\T = \R^n/2\pi {\mathbb Z}^n$. In this set up, it is observed by Guillemin~\cite{guillemin} that there are canonical \emph{angular coordinates}  $d\ang=\{d\ang_1, \ldots, d\ang_n\}$ defined on the dense open subset  $\mathring{M}:= \mu^{-1}(\mathring{\Pol})$ (where $\mathring{\Pol}$ is the interior of $\Pol$).
Furthermore, with respect to the coordinate system $(\mu, \theta)$,
the $\T$-invariant $2$-form on $\mathring{M}$
\[\omega := \langle d\mu, d\ang \rangle = \sum_{i=1}^n d\mu_i \wedge d\ang_i\]
is smoothly extendable to $M$ and defines a symplectic structure (still denoted by $\omega$). In general, $F$ and $\omega$ are different symplectic forms, even though they are $\T$-equivariantly symplectomorphic,  belong to the same deRham class $[F] =[\omega]$, and share the same momentum coordinates and Delzant polytope $(\Pol, \Lab)$.
The main conclusion in \cite{boulanger,W2} is that, up to a $\T$-equivariant isometry, $(F, J)$ is obtained as follows:
\begin{itemize}
\item There exists a unique,  up to the addition of an affine-linear function on $\tor^*$,  convex smooth function $u(x)$ defined on  the interior $\mathring{P}$ such that
\begin{equation*}\label{g-u}
g_{u} := \sum_{i,j=1}^n  u_{,ij}(\mu) d\mu_i d\mu_j  + u^{,ij}(\mu) d\ang_i d\ang_j \end{equation*}
is a $\omega$-compatible K\"ahler structure (defined on $\mathring{M}$ and extendable to $M$) whose complex structure is \[J_u: = g_u^{-1}\omega = \sum_{i,j=1}^n\left(u_{,ij}d\mu_i \otimes \frac{\partial}{\partial \ang_j} - u^{,ij}  d\ang_i\otimes \frac{\partial}{\partial \mu_j}\right). \]\
\item There exist unique elements $A, B \in \Wedge^2 \tor$ such that $(F, J)$ is obtained from $(\omega, J_u)$ by deformations of type A and B, i.e.
\[
\begin{split}
J &= \sum_{i,j=1}^n\left(\big(u_{,ij} + A_{ij}\big) d\mu_i \otimes \frac{\partial}{\partial \ang_j} - \big(u_{,ij} + A_{ij}\big)^{-1}_{ij} d\ang_i\otimes \frac{\partial}{\partial \mu_j}\right), \\
F &= \sum_{i=1}^n d\mu_i \wedge d\ang_i + \sum_{i,j=1}^n B_{ij} d\mu_i \wedge d\mu_j,
\end{split}\]
where in the above formulae  $A=(A_{ij}), B=(B_{ij})$ using the chosen basis of $\tor$.\\
\item The corresponding Poisson tensor $\poiss_J$ is
\begin{equation}\label{p:toric-Poisson}
\poiss_J= \nontwice \sum_{i,j=1}^n \Big(A_{ij} + \i B_{ij} \Big)\Big(\frac{\partial}{\partial \ang_i} -\i J\frac{\partial}{\partial \ang_i}\Big) \wedge \Big(\frac{\partial}{\partial \ang_j} - \i J\frac{\partial}{\partial \ang_j}\Big).
\end{equation}
\end{itemize}

\begin{defn}[Abreu--Guillemin data]\label{d:abreu-guillemin} For a given toric symplectic type GK structure $(M, F, J, \T)$ with  Delzant polytope $(\Pol, \Lab)$, we will refer to  $(u, A, B)$ defined above as the corresponding \emph{Abreu--Guillemin data}. We note that $A, B \in \Wedge^2\tor$ determine and are determined by the holomorphic Poisson tensor $\poiss_J$ of $(F, J)$ via \eqref{p:toric-Poisson}. The convex function  $u(x)$ is defined only up to additive affine-linear function,  and the K\"ahler toric structure $(\omega, J_u, g_u)$ is referred to as \emph{the toric  K\"ahler reduction} of $(F, J)$.
\end{defn}

We now describe more precisely the space ${\mathcal S}_{A, B}(\Pol, \Lab)$ of smooth convex functions $u(x)$ on $\mathring{\Pol}$ appearing in Definition~\ref{d:abreu-guillemin}:
\begin{lemma}\label{l:boundary} A smooth strictly convex function $u(x)$ on $\mathring{\Pol}$   belongs to ${\mathcal S}_{A, B}(\Pol, \Lab)$ iff the following conditions hold:
\begin{enumerate}
\item[(i)] $u(x)$ is smooth and strictly convex on $\mathring{\Pol}$ such that $\left(\Hess(u) + \i B\right)$ is positive definite Hermitian form at any point of $\mathring{\Pol}$;
\item[(ii)] $u(x) = \tfrac{1}{2} \sum_{L \in \Lab} L(x) \log L(x)  + v(x)$ with  $v(x)$ smooth over  $\Pol$;
\item[(iii)] on the interior  $\mathring{{\mathrm F}}$ of any face  $\mathrm{F} \subset \Pol$,  $\left(\Hess(u) + \i B\right)$ is a smooth and positive definite Hermitian form on $\tor_{\mathrm{F}}^*\otimes \C$.
\end{enumerate}
The above conditions  (i)-(ii)-(iii) are equivalent to the conditions (i)-(ii)-(iii)' where
\begin{enumerate}
\item[(iii)'] $\left(\Hess(u) + \i B\right)^{-1} {\bf G}_0$  extends smoothly on $\Pol$, where ${\bf G}_0 = \Hess\left(\tfrac{1}{2} \sum_{L \in \Lab} L(x) \log L(x)\right)$.
\end{enumerate}
\end{lemma}
\begin{proof}
The conditions (i)-(ii)-(iii) extend the one obtained by Donaldson~\cite{donaldson-interior} in the K\"ahler case, and reflect the fact that $(F, J)$ induces a symplectic structure taming $J$ on the pre-image of each face ${\rm F}$ (which is a smooth toric sub-manifold of $(M, J)$). The proof is similar and left to Reader. One can alternatively  use a Taylor expansion  around a point of $\mathring{F}$ as in \cite[App.A2]{AAS} to show   that (i)-(ii)-(iii) are equivalent to (i)-(ii)-(iii)'. The latter are necessary and sufficient for the extension of $(F, J)$ to $M$ by \cite{W2} (which in turn uses \cite[Lemma~3 and Remark 4(ii)]{ACGT2}). \end{proof}

\begin{lemma}\label{l:S-convex} ${\mathcal S}_{A, B}(\Pol, \Lab) = {\mathcal S}_{0, B}(\Pol, \Lab)$ is a  linearly convex subset of $\mathcal{S}_{0,0}(\Pol, \Lab)$,  which is a Fr\'echet space modeled on $C^{\infty}(\Pol)$.
\end{lemma}
\begin{proof} The first claim follows from the description (i)-(ii)-(iii) of ${\mathcal S}_{A, B}(\Pol, \Lab)$ in Lemma~\ref{l:boundary} and the general relation of  Hermitian forms
$\left(\Hess(u) + \i B\right)_{\tor^*_{\mathrm{F}}} \le \Hess(u)_{\tor^*_{\mathrm{F}}}$. The convexity property also follows from (i)-(ii)-(iii). To observe the relevant Fr\'echet manifold structure of ${\mathcal S}_{A, B}(\Pol, \Lab)$, it is enough to show that  for any $u\in {\mathcal S}_{A, B}(\Pol, \Lab)$ and  $v\in C^{\infty}(\Pol)$, $u+tv \in {\mathcal S}_{A, B}(\Pol, \Lab)$  for $|t|<\epsilon(u, v)$. This follows easily from the description (i)-(ii)-(iii)' of  ${\mathcal S}_{A, B}(\Pol, \Lab)$ given in Lemma~\ref{l:boundary}. \end{proof}

\subsection{Invariant classes and geodesics}

The description of ${\mathcal S}_{A,B}(\Pol, \Lab)$ gives the following connectedness result in the toric case:
\begin{prop}\label{p:contractible} Let $(M, J_0, F_0)$ be a toric symplectic type generalized K\"ahler structure corresponding to Abreu--Guillemin data $(u_0, A, B), \, u_0\in {\mathcal S}_{A,B}(\Pol, \Lab)$.  Denote by $\alpha=[F_0]$ the corresponding deRham class in $H^2(M, \R)$. Then
${\GK}_{\poiss,\alpha}^{\T}$ is path-connected.
\end{prop}
\begin{proof} If $F \in {\GK}^{\T}_{\poiss, \alpha}$, we can use the $\T$-equivariant Moser's lemma to  send   the $\T$-invariant generalized K\"ahler structure $F$ (via an isotopy of $\T$-equivariant diffeomorphisms connecting the identity) to a $\T$-invariant  GK structure $(F_0, J)$, compatible with the symplectic form $F_0$. Acting by a further $\T$-equivariant $F$-symplectomorphism  in ${\rm Symp}_0^{\T}(M, F_0)$ if  necessary,  we can also assume that $(F_0, J_0)$ and $(F_0, J)$  correspond to Abreu--Guillemin data $(u_0, A, B), (u, A, B)$, $u, u_0 \in \mathcal{S}_{A,B}(\Pol, \Lab)$ and are respectively obtained by deformations of type A and B of two $\omega$-compatible K\"ahler structures $J_{u_0}$ and $J_{u}$ with the same momentum-angular coordinates. As the space $\mathcal{S}_{A,B}(\Pol, \Lab)$ is linearly convex, letting $u_t:= (1-t)u_0 + tu \in \mathcal{S}_{A,B}(\Pol, \Lab)$ defines a smooth path $J_{t}$  between $J_0$ and $J$ inside the space of $\T$-invariant  $F_0$-compatible generalized K\"ahler structures.  Now, by \cite[Lemma~3.6]{ASU}, we can find a $\T$-equivariant isotopy of diffeomorphisms, sending $(F_0, J_t)$ to $(F_t, J_0)$,  preserving $[F_t]=[F_0]=\alpha$.
\end{proof}

One of the key applications of the Abreu--Guillemin formalism for the theory of toric K\"ahler manifolds is that it transforms the rather complicated geodesic equation in the space ${\KK}^{\T}_{\alpha}$  into the linear equation $\ddot u=0$ for the corresponding symplectic potentials. This implies, in particular, that the space of $\T$-invariant K\"ahler metrics  ${\KK}^{\T}_{\alpha}$ is geodesically connected,  which in turn yields the uniqueness  (modulo $\T^{\C}$) of the constant scalar curvature metrics in that space (see \cite{Guan}).  We show below  that  these phenomena  persist in the generalized K\"ahler setting. For simplicity, we check these facts under the assumption $A=0$, which we can suppose without loss of generality if we study  extremal generalized K\"ahler structures (as we show in the next subsection).

\begin{prop}\label{l:geodesic-toric} Let $(F_0, J)$ be a toric symplectic type generalized  K\"ahler manifold with holomorphic Poisson tensor $\poiss_J$ given by \eqref{p:toric-Poisson} for $A=0$ and $B\in \Wedge^2 \tor$. Denote by $(\omega_0, J)$ the corresponding $\T$-invariant K\"ahler  reduction and by $\alpha := [F_0]=[\omega_0]$ the corresponding K\"ahler  class.  For any path $F_t \in {\GK}^{\T}_{\poiss,\alpha}$  let  $\kom_t$ be the unique K\"ahler metric in $\alpha$ which solves  $\kom_t^n = F_t^n$ and denote by $\phi_t$ the unique  $\T$-invariant smooth function such that
\[ \omega_t =  \omega_{0} + dd^c_J \phi_t, \qquad \int_M \dot{\phi_t} \omega_t^n = 0, \qquad \phi_0 =0.\]
Then $F_t$ is obtained from the Hamiltonian deformation of $F_0$ with respect to $\dot \phi_t$ and the following conditions are equivalent:
\begin{itemize}
\item $\phi_t$ is a geodesic in the space of $\T$-invariant $\omega_0$-relative K\"ahler potentials;
\item $\dot\phi_t$ is a geodesic  on $\GK_{\poiss,\alpha}^0$ in the sense of Definition~\ref{d:geodesic};
\item If $(u_t, 0, B)$ denotes the corresponding to Abreu--Guillemin data of $(F_t, J, \T)$, then $\ddot u_t =0$.
\end{itemize}
\end{prop}
\begin{proof} By \cite{guillemin} (see also \cite{ASU}), we can express, on $(\mathring{M}, J)$, $\omega_t= dd^c_{J} \varphi_t$ for  some $\T$-invariant smooth function $\varphi_t$. More precisely,   if $(y,  \, \ang)$ are the (exponential) $J$-pluriclosed coordinates defined on the open dense orbit  $\mathring{M} = \T^{\C} \cdot p_0$ of the complexified action $\T^{\C}$ (determined by fixing a base point $p_0\in \mathring{M}$), there is a unique,  up to the addition of an affine-linear function,  function $\varphi_t(y)$ with the above property. One can further require that $\phi_t(y):=\varphi_t (y) - \varphi_0(y)$ extends smoothly to $M$, which leaves only an additive constant in the definition $\varphi_t(y)$ (and $\phi_t$), once we have chosen $\varphi_0$.  This constant is further determined  via the normalization used in the lemma.

We now consider the Legendre transform, $u_t$,  of $\varphi_t$: letting $\mu_t := \nabla \varphi_t$, $(\mu_t, \ang)$ are momentum-angular coordinates of $\omega_t$ and  the functions $u_t(x)$ defined by
\[\varphi_t (y) + u_t(\mu_t)= \langle y, \mu_t \rangle\]
are elements of $\mathcal{S}_{0,B}(\Pol, \Lab)$, such that $(u_t, 0, B)$ are Abreu--Guillemin data of $(F_t, J)$ pulled back by  the diffeomorphism $\mu_t \to \mu_0, \ang \to \ang$, see \cite{ASU} for details. Using the basic property of the Legendre transform  ${\bf H} = \left({\rm Hess}(u)\right)^{-1} = \Hess(\varphi)$, we have
\[ \kom_{t} = \sum_{i,j=1}^m {\bf H}_{ij} d y_i \wedge d\ang_j, \qquad
        F_{t}= \sum_{i,j=1}^m{\bf H}_{ij}d y_{i}\wedge d\ang_j+({\bf H}B{\bf H})_{ij}d y_{i}\wedge d y_{j},
        \]
        where ${\bf H}_t(y) = \Hess(\varphi_t)(y)= \Hess(u_t)^{-1}(\mu_t)$.
Let $\dot\phi = \dot \varphi$ be a first order variation of $\phi_t$,  and $\dot{\kom}$, $\dot F$  and $\dot I$ denote the corresponding first order variations of  $\kom_{t}$, $F_{t}$ and $I_{t}$. We have
        \[
        \dot{\kom} = \sum_{i,j=1}^m \dot {\bf H}_{ij} d y_i \wedge d\ang_j, \qquad \dot F=\sum_{i,j=1}^m\dot{\bf H}_{ij}d y_i\wedge\ang_j+(\dot{\bf H}B{\bf H}+{\bf H}B\dot{\bf H})_{ij}d y_i\wedge d y_j.
        \]
        Using that $(y, \ang)$ and $(y, \bar{\ang}=\ang -2B\mu_t)$ with $\mu_t = \nabla \varphi$ (see \cite{ASU}) are respective pluriharmonic coordinates for $J$ and $I_t$, we compute
\begin{equation*}
        \begin{split}
        d d^c_J \dot \phi &= \sum_{i,j=1}^{m} {\dot \phi}_{,ij} d y_i \wedge d \ang_j= \sum_{i,j=1}^{m} {\dot {\bf H}}_{ij} d y_i \wedge d \ang_j, \\
    d d^c_I \dot{\phi}&= \sum_{i=1}^m d I\left(\dot{\phi}_{,i}d y_i\right)=d\left(\sum_{i=1}^m\dot{\phi}_{,i}d\ang_i-2\sum_{i,j,k=1}^m\mathbf{H}_{jk}B_{ki}\dot\phi_{,i}d y_j\right)\\
    &=\sum_{i,j=1}^m \left(\dot{\bf H}_{ij}d y_i\wedge d\ang_j+2({\bf H} B\dot{\bf H})_{ij}d y_i\wedge d y_j\right)\\
    &= \sum_{i,j=1}^m\left(\dot{\bf H}_{ij}d y_i\wedge d\ang_j+({\bf H} B\dot{\bf H}+\dot{\bf H} B{\bf H})_{ij}d y_i\wedge d y_j\right),
        \end{split}
        \end{equation*}
        which gives
        \[ \dot \omega = dd^c_J \dot \phi, \qquad \dot F = dd^c_I \dot \phi.\]
        The variation of $\dot I = -\tfrac{1}{2} \poiss (dd^c_I \dot \phi)$ then follows from  the general property of variations in ${\AGK}_{\poiss, \alpha}$, see Section~\ref{s:AGK-poiss}, and the $dd^c_I$-lemma. This shows that   $F_t$ is a Hamiltonian deformation of $F_0$ with function $\dot \phi_t$.

    The equivalence of the first and third statements is established in \cite{Guan}. The equivalence of the first and second statements follows from \eqref{e:geodesic} and the general relation
    \[ df\wedge Jdf \wedge F^{n-1} = df \wedge Jdf \wedge \omega^{n-1}, \qquad f=f(y),\]
    which can be checked from the above expressions for $\omega$ and $F$. \end{proof}

\subsection{The generalized K\"ahler  scalar curvature and extremal structures} In the above setting, a computation from \cite{W3} shows that Goto's scalar curvature of $(F, J, \T)$ is given by an Abreu-type formula (compare with \cite{Abreu} in the K\"ahler case):
\begin{equation}\label{GK-Abreu}
{\Gscal}_{(F, J)} = - \sum_{i,j=1}^n \frac{\partial^2}{\partial \mu_i \partial \mu_j} \left(\Hess(u) + \i B\right)_{ij}^{-1} =: \kappa(u, A, B).
\end{equation}
This can be also deduced from \cite{boulanger} (where the momentum map for the action of ${\rm Ham}^{\T}(M, F)$ on ${\AGK}^{\T}_F$ is identified with $\kappa$) and Theorem~\ref{t:GKscal-moment-map} above.

By Theorem~\ref{t:Calabi-Lichne-Matsushima}, for any $\T$-invariant \emph{extremal} generalized K\"ahler  structure  $(F, J)$ with Abreu--Guillemin data $(u, A, B), \, u \in {\mathcal S}_{A,B}(\Pol, \Lab)$, the corresponding generalized K\"ahler scalar curvature given by \eqref{GK-Abreu} is a pull-back by $\mu$ of an affine-linear function $\ell(x)$ on $\Pol$. It follows that the extremal equation is
\[ \kappa(u, A, B) = \ell (\mu), \qquad u\in {\mathcal S}_{A,B}(\Pol, \Lab). \]
An important feature is that the above expression is independent of $A$. Indeed, the LHS is manifestly independent of $A$ and we have ${\mathcal S}_{A,B}(\Pol, \Lab)= {\mathcal S}_{0,B}(\Pol, \Lab)$. Furthermore
\begin{lemma}\label{l:donaldson} If $(u, A, B), \, u \in {\mathcal S}_{A,B}(\Pol, \Lab)$ defines an extremal generalized K\"ahler structure with corresponding affine-linear function $\ell$, then $\ell=\ell_{\rm ext}$ is the extremal affine-linear function of the labeled $(\Pol, \Lab)$,  introduced in \cite{donaldsonJDG-02}.  In particular, $\ell_{\rm ext}$ is independent of $A, B$.
\end{lemma}
\begin{proof} Let
\begin{equation}\label{X}
{\bf X} := \Re e \left({\rm Hess}(u) + \i B\right)^{-1}.
\end{equation}
As  $ \Im m \left({\rm Hess}(u) + \i B\right)^{-1}$ is skew-symmetric, \eqref{GK-Abreu} becomes
\begin{equation}\label{AK-extremal}
{\Gscal}_{(F, J)}=- \sum_{i,j=1}^n {\bf X}_{ij, ij} = \ell.
\end{equation}
Notice that ${\bf X}=\left({\bf G}+ B{\bf H} B\right)^{-1}$ where ${\bf G} = {\rm Hess} (u)$ and ${\bf H}= {\bf G}^{-1}$. It then follows from the boundary conditions (i)-(ii)-(iii)'  in Lemma~\ref{l:boundary} that ${\bf X}^{-1} - {\bf G}_0$ is smooth and ${\bf X} {\bf G}_0$ is smooth and nondegenerate on $\Pol$. These in turn yield that ${\bf X}$ is smooth on $\Pol$ and satisfies first order boundary conditions at each facet of $\Pol$ (see  Lemma 2, Remark 4 and Proposition 1 in \cite{ACGT2}).   One can then obtain from \eqref{X}, by integrating by parts twice as in ~\cite{donaldsonJDG-02} (see also \cite[Lemma~3.2]{ apostolov-notes}) that for any smooth function $f$ on $\Pol$
\begin{equation}\label{toric-by-parts}
0= \int_{\Pol} f\left(\sum_{i,j=1}^n {\bf X}_{ij,ij} +\ell\right) dx = \int_{\Pol}f \ell dx -2\int_{\partial \Pol} f d\sigma_{\Lab}  + \int_{\Pol} \tr \Big({\bf X}\circ\Hess (f)\Big) dx,
\end{equation}
where $d x$ is the Lebesgue measure on $\tor^*$ associated to  the  chosen basis of $\tor$,  and $d\sigma_{\Lab}$ is the induced measure by $d x$ and the  inward normal $d L_j$ on each facet $\Pol_j \subset \partial \Pol$.  Specializing the above formula for $f$ affine-linear, we obtain that
\[- \int_{\Pol}  \ell(x) f(x)d x  + 2 \int_{\partial \Pol} f(x) d\sigma_{\Lab}=0, \qquad \forall \, f \, \textrm{affine-linear}.\]
This determines a unique affine-linear function $\ell$, denoted by $\ell_{\rm ext}$.
\end{proof}
\begin{cor}\label{c:A-independence} $(u, A, B) \in {\mathcal S}_{A,B}(\Pol, \Lab)$ defines an extremal $\T$-invariant generalized K\"ahler structure iff $(u, 0, B) \in {\mathcal S}_{0,B}(\Pol, \Lab)$ defines an extremal $\T$-invariant generalized K\"ahler structure.
\end{cor}
Following~\cite{donaldsonJDG-02}, for any affine-linear function $\ell$, we introduce a linear functional defined on the space of continuous functions on $\Pol$  by
\begin{equation}\label{Futaki}
{\bf F}_{\ell} (f) :=  - \int_{\Pol}  \ell(x) f(x)d x  + 2 \int_{\partial \Pol} f(x) d\sigma_{\Lab}.
 \end{equation}
\begin{defn}\label{d:l-ext} The \emph{extremal affine-linear function} $\ell_{\rm ext}$ of $(\Pol, \Lab)$ is the unique affine-linear function $\ell$ such that  ${\bf F}_{\ell}(f)=0$ for any affine-linear  function $f$. The corresponding functional ${\bf F}_{\ell_{\rm ext}}$ is called the \emph{relative Donaldson--Futaki invariant} of $(\Pol, \Lab)$.
\end{defn}
With the above remarks in mind, we now consider the following PDE problem which is a modified version of Abreu's
equation in \cite{Abreu}:
\begin{equation}\label{GKextremal}
- \sum_{i,j=1}^n \frac{\partial^2}{\partial \mu_i \partial \mu_j} \left(\Hess(u) + \i B\right)_{ij}^{-1} = \ell_{\rm ext} (\mu),  \qquad u \in \mathcal{S}_{0, B}(\Pol, \Lab),
\end{equation}
where $\ell_{\rm ext}$ is the extremal affine-linear function defined in Definition~\ref{d:l-ext}.

\subsection{The relative Mabuchi energy and the uniqueness of extremal generalized K\"ahler structures}

The identification of ${\GK}^{\T}_{\poiss,\alpha}$ with the convex space $\mathcal{S}_{A, B}(\Pol, \Lab)$ can be used to integrate the  formal closed 1-form $\boldsymbol{\tau}$  and obtain a well-defined  Mabuchi functional. This has been indeed obtained in \cite{ASU}, where the following  ``toric'' Mabuchi functional was introduced:
\[{\bf M}(u):= {\bf F}_{a}(u)  - \int_{\Pol} \log \det \big({\rm Hess}(u) + \i B\big) d x  + \int_{\Pol} \log \det \big({\rm Hess}(u_0)\big) d x,  \qquad  u \in \mathcal{S}_{A, B}(\Pol, \Lab). \]
In the above formula,   $u_0\in \mathcal{S}_{0,0}(\Pol, \Lab)$  is  the  convex function associated to some background toric K\"ahler metric compatible with $F$,  and  $a$ is the topological constant
\[ a:= 2\frac{\Vol\left({\partial \Pol}, d\sigma_{\Lab}\right)}{\Vol(\Pol, d x)} \]
which computes the average value of $\Gscal_{(F, J)}$ (see Remark~\ref{r:avescal}),  and  ${\bf F}_a$ denotes the linear functional \eqref{Futaki} defined with respect to the affine-linear function $\ell \equiv a$. Notice that  the difference of the last two terms is well-defined for $u\in \mathcal{S}_{A,B}(\Pol, \Lab)$, by virtue of  the conditions (i)-(ii)-(iii)' in Lemma~\ref{l:boundary}.  We can more generally introduce the \emph{relative} Mabuchi energy
\[ {\bf M}^{\ell_{\rm ext}} (u):= {\bf F}_{\ell_{\rm ext}}(u)  - \int_{\Pol} \log \det \big({\rm Hess}(u) + \i B\big) d x  + \int_{\Pol} \log \det \big({\rm Hess}(u_0)\big) d x,  \]
where, instead of the constant function $a$, we use the extremal affine-linear function $\ell_{\rm ext}$ in the linear term.

\begin{lemma}\label{l:Mabuchi-derivative} Given a one-parameter family of symplectic potentials $u_s = u + s \dot{u}$ and fixed  $A,B$, we have
\[ (\delta_{u} {\bf M}^{\ell_{\rm ext}})(\dot u) = \int_{\Pol} \dot u \left(\gk(u,A,B) - \ell_{\rm ext}\right) d x.\]
In particular, the critical points of ${\bf M}^{\ell_{\rm ext}}$ correspond to extremal toric generalized K\"ahler structures with Poisson tensor corresponding to $A, B$.  Furthermore,
\[(\delta^2_{u} {\bf M}^{\ell_{\rm ext}})(\dot u, \dot u) = \int_{\Pol}  {\rm tr}\Big[\Big(\big({\rm Hess}(u) + \i B\Big)^{-1}{\rm Hess}(\dot u)\Big)^2\Big] d x \ge 0.\]
\end{lemma}
\begin{proof} A direct computation shows
\[
\begin{split}
\frac{d}{ds}\left({\bf M}^{\ell_{\rm ext}}(u_s)\right) &= {\bf F}_{\ell_{\rm ext}}(\dot u) -\int_{\Pol} \tr\left(\Big({\rm Hess}(u) + \i B\Big)^{-1} \circ {\rm Hess}(\dot u)\right) dx \\
&= {\bf F}_{\ell_{\rm ext}}(\dot u) -\int_{\Pol} \tr\left({\bf X} \circ  {\rm Hess}(\dot u)\right) dx,
\end{split} \]
where ${\bf X} = \Re e \Big({\rm Hess}(u) + \i B\Big)^{-1}$. The first identity follows  from  the above and using the second equality in \eqref{toric-by-parts},  whereas the second identity follows by taking a second derivative in $s$. \end{proof}
\begin{thm}\label{t:uniqueness-toric} On a toric holomorphic Poisson manifold $(M, J, \poiss_J, \T)$ endowed with a generalized K\"ahler structure $F_0 \in {\GK}^{\T}_{\poiss}$, the extremal  generalized K\"ahler structures in $\GK_{\poiss,\alpha}^0$ are unique modulo $\Aut_0(J, \poiss_J)$. Furthermore, the  symplectic potential  in ${\mathcal S}_{A, B}(\Pol, \Lab)$ of  a $\T$-invariant extremal generalized K\"ahler structure  in ${\GK}^{\T}_{\poiss,\alpha}$ minimizes ${\bf M}^{\ell_{\rm ext}}$.
\end{thm}
\begin{proof} By Corollary~\ref{c:T-reduction}, it is enough to show uniqueness in $(\GK_{\poiss,\alpha}^{\T})^{0}$, modulo the action of $\T^{\C}$.  In this case, the problem is reduced to the uniqueness of \eqref{GK-Abreu} modulo affine-linear functions (see \cite{ASU}), and the latter  follows from the convexity of ${\bf M}^{\ell_{\rm ext}}$ established in Lemma~\ref{l:Mabuchi-derivative} and the fact that $\mathcal{S}_{A,B}(\Pol, \Lab)$ is linearly convex (see Lemma~\ref{l:S-convex}).   This also shows that  the symplectic potential of an extremal  generalized K\"ahler structures is a  global minimizer of ${\bf M}^{\ell_{\rm ext}}$. \end{proof}

\subsection{The obstruction theory}

\begin{defn}\cite{CLS, donaldsonJDG-02}\label{d:uni-K-stable} We say that $(\Pol, \Lab)$ is \emph{uniform relative K-stable} if, for a chosen point $x_0 \in \mathring{\Pol}$, there exist uniform positive constants $\lambda=\lambda(\Pol, \Lab, x_0), \delta=\delta(\Pol, \Lab, x_0)$ such that
 \begin{equation}\label{uniform-K-stable}
 {\bf F}_{\ell_{\rm ext}}(f) \ge \lambda \int_{\partial\Pol} f d \sigma_{\Lab}  -\delta,
 \end{equation}
for any continuous convex function $f(x)$ on ${\Pol}$, \emph{normalized}  (by adding an affine-linear function  $\ell$ which does not affect the RHS) so that  $f(x) \ge f(x_0)=0$.
\end{defn}
By the works \cite{CLS, donaldsonJDG-02},  if $(M, F, \T)$ admits a compatible extremal K\"ahler metric, then $(\Pol, \Lab)$ is necessarily uniform relative K-stable. The arguments readily extend to the generalized K\"ahler case:
\begin{thm}\label{c:easy} Let $(M, F, J, \T)$ an  extremal toric symplectic type generalized K\"ahler structure. Then the corresponding Delzant polytope $(\Pol, \Lab)$ is {uniform relative K-stable}.
 \end{thm}
\begin{proof}  Let $(u, A, B)$ be Abreu--Guillemin data corresponding to $(F, J, \T)$ and  ${\bf X}$ the corresponding  symmetric matrix valued function defined in \eqref{X}. As we have already observed in the proof of Lemma~\ref{l:donaldson}, ${\bf X}$ is smooth on $\Pol$,  positive definite  on $\mathring{\Pol}$, and satisfies the same first order boundary conditions at $\partial \Pol$ as  $\left({\rm Hess}(u_0)\right)^{-1}$ for any $u_0$ corresponding to a compatible toric K\"ahler structure.   Furthermore, the extremality of $(F, J)$ is equivalent to \eqref{AK-extremal}.  Notice that, using  \eqref{AK-extremal} and ${\bf X}>0$ on $\mathring{\Pol}$,
 \eqref{toric-by-parts} already shows that  for any smooth convex function $f$, ${\bf F}_{\ell_{\rm ext}} (f) \ge 0$ with equality iff $f$ is affine-linear. The improvement to uniform relative K-stability of $(\Pol, \Lab)$  is obtained in \cite[Thm.~4.4 and Prop. 4.6]{CLS}. The arguments of the proof of \cite[Theorem~4.4]{CLS} can be carried out by replacing the inverse Hessian $(v^{ij})= \left(\Hess(u_{\rm ext})\right)^{-1}$ with the matrix valued function ${\bf X}$ defined in \eqref{X}. \end{proof}

\begin{rmk}\label{r:AK} As we have seen in the course of the proof of Lemma~\ref{l:donaldson}, ${\bf X}$ is positive-definite on $\mathring{\Pol}$ and satisfies the same boundary conditions on $\partial \Pol$ as  ${\bf H}= \Hess(u)^{-1}$. By \cite{ACGT2}, ${\bf X}$  then defines an $\omega$-compatible,  $\T$-invariant Riemannian metric $\bar g$ on $\mathring{M}$ which is smoothly extendable to $M$,   by the formula
\[\bar g = \sum_{i,j=1}^n \left( {\bf X}^{-1}_{ij} d\mu_i d\mu_j + {\bf X}_{ij} d\ang_i d\ang_j\right).\]
As ${\bf X}$ satisfies \eqref{AK-extremal}, this is an instance of   \emph{extremal almost-K\"ahler} structure in the sense of \cite{lejmi}. E.\,Legendre observed in \cite{Legendre} that the proof of the uniform relative K-stability of $(\Pol, \Lab)$ in \cite{CLS} extends to the case of extremal toric almost-K\"ahler structures as above. The proof of Theorem~\ref{c:easy} can be alternatively deduced from her result.
\end{rmk}

Conversely, it is now established as a consequence of deep recent work by Chen--Cheng~\cite{CC} with a supplement by \cite{He},  and previous work by Donaldson~\cite{donaldsonJDG-02} and Zhou--Zhu~\cite{ZZ},   that any smooth toric variety with uniform relative K-stable Delzant polytope $(\Pol, {\Lab})$ in the sense of \eqref{uniform-K-stable} admits a compatible toric extremal K\"ahler metric (see \cite{apostolov-notes} for a survey of the proof).
\begin{thm}\cite{CC,He}\label{t:CC}  Let $(M, J, \T^{\C})$ be a smooth toric variety and $\alpha$ a K\"ahler deRham class on $(M, J)$, with corresponding  uniform relative K-stable labeled Delzant polytope $(\Pol, \Lab)$. Then  ${\KK}_{\alpha}^{\T}$ admits an extremal K\"ahler metric.
\end{thm}
Combining Theorems~\ref{c:easy} and \ref{t:CC}, we get the following
\begin{cor}\label{t:hard} Suppose $(M, F, J, \T)$ is an extremal toric symplectic type generalized K\"ahler structure. Then,   $\alpha:= [F]$ is a K\"ahler class which admits a $\T$-invariant extremal K\"ahler metric.
\end{cor}
\begin{proof} This follows from Theorems~\ref{c:easy} and \ref{t:CC}, noticing that $\alpha=[F]$ is a K\"ahler class with Delzant polytope $(\Pol, \Lab)$: we already saw that the K\"ahler reduction  $(\kom, J_{u})$  satisfies $\omega\in \alpha$ and $J_{u}$ biholomorphic to $J$, see \cite[Lemma~3.6]{ASU}. \end{proof}

\begin{rmk}
By \cite{donaldsonJDG-02, hisamoto}, the uniform relative stability of $(\Pol, \Lab)$ is equivalent to a notion of uniform relative K-stability of the underlying K\"ahler manifold $(M, J, \alpha)$. The above result strongly suggests that in general, there is a similar stability notion associated to an extremal generalized K\"ahler manifold of symplectic type.
\end{rmk}

\subsection{An existence result a la LeBrun--Simanca}

Theorem~\ref{t:CC} motivates us to ask:
\begin{qtn}\label{q:existence} If $B\neq 0$, are there further obstructions,  beyond the uniform relative  K-stability of $(\Pol, \Lab)$, to the existence of a solution of \eqref{GK-Abreu}?
\end{qtn}
We show below that the answer to Question~\ref{q:existence} is negative for sufficiently ``small'' $B$. This follows from  a straightforward adaptation of  the LeBrun--Simanca~\cite{LS} openness result (compare with \cite[Theorem~8.2]{GotoLichne} which requires trivial automorphism group).
\begin{thm}\label{t:LS} Suppose $(M, J, \T^{\C})$ is a smooth toric variety which admits an extremal $\T$-invariant K\"ahler metric $\omega_{0}$ in the deRham class $\alpha$. Let $\poiss_J$ be a $\T^{\C}$-invariant Poisson structure.  Then, there exists $\varepsilon = \varepsilon(\omega_0, \poiss_J)>0$, such that for any $t\in \R, \, |t| < \varepsilon$, there exists an extremal generalized K\"ahler  structure $F_t \in {\GK}^{\T}_{t\poiss, \alpha}$. \end{thm}
\begin{proof} Let $(\Pol, \Lab)$ be the Delzant polytope of $(M, \omega_0, \T)$.  By the $\T^{\C}$-invariance of  $\poiss_J$, it is of the form \eqref{p:toric-Poisson} for $A, B \in \Wedge^2 \tor$. The extremal K\"ahler metric $(J, \omega_0)$ then corresponds to a solution $u_{0} \in \mathcal{S}_{0,0}(\Pol, \Lab)$  at $t=0$ of the family of PDE's
\begin{equation}\label{GK-Abreu-t}
	 - \sum_{i,j=1}^n \frac{\partial^2}{\partial \mu_i \partial \mu_j} \left(\Hess(u_t) + \i tB\right)_{ij}^{-1} = \ell_{\rm ext}(\mu), \qquad u_t \in \mathcal{S}_{tA, tB}(\Pol, \Lab)=\mathcal{S}_{0,tB}(\Pol, \Lab).
\end{equation}
We first notice that \eqref{GK-Abreu-t} is well-defined for $|t|< \varepsilon_0$. Indeed, there exists $\varepsilon_0>0$ such that the closed $2$-form $F_t = \omega_0 + t\sum_{i,j=1}^{n} B_{ij} d\mu_i\wedge d\mu_j$ tames $J$ for any $|t|<\varepsilon_0$; the A-transform of $(F_t, J)$ with $A\in \Wedge^2\tor$ then defines a generalized K\"ahler structure with data  $(u_0, tA, tB)$ for any $|t|<\varepsilon_0$.  Thus, the spaces $\mathcal{S}_{tA, tB}(\Pol, \Lab)= \mathcal{S}_{0, tB}(\Pol, \Lab)$ are nonempty.

Up to an equivariant diffeomorphism (obtained by using Moser's lemma),  any element  of ${\GK}^{\T}_{t\poiss, \alpha}$ corresponds to  a $F$-compatible toric generalized K\"ahler structure with data $(u_t, tA, tB)$. It is thus enough to show that there exists $0<\varepsilon  < \varepsilon_0$  such that \eqref{GK-Abreu-t} has a solution $u_t\in \mathcal{S}_{tA, tB}(\Pol, \Lab)$ for any $|t|<\varepsilon$.  By  Corollary~\ref{c:A-independence}, it is sufficient to show the latter  holds for $A=0$, or equivalently, we can assume without loss of generality that the Poisson tensor $\poiss_J$ is of the form \eqref{p:toric-Poisson} with $A=0$.

In order to set up the problem in a form suitable for the application of the implicit function theorem, we will first apply the Legendre transform to \eqref{GK-Abreu-t} (assuming $A=0$), as we detailed in the proof of Proposition~\ref{l:geodesic-toric}: letting $y_t = \nabla u_t$ and $\varphi_t(y_t) + u_t(\mu)= (y_t, \mu)$, this introduces $\T$-invariant smooth functions $\phi_t(y_0)= \varphi_{u_t}(y_0) -\varphi_{u_0}(y_0)$ on $M$, such that
\begin{equation}\label{omega-t}
\begin{split}
\omega_{\phi_t} &:= \omega_0 + dd^c_J\phi_t>0,  \\
 F_{(\phi_t, tB)} &: = \omega_{\phi_t} + t\sum_{i,j=1}^n B_{ij} d\mu^{\phi_t}_i\wedge d\mu^{\phi_t}_j, \qquad \mu^{\phi_t}:= \mu + d^c_J \phi_t,
\end{split}
\end{equation}
is a generalized K\"ahler  structure in ${\GK}^{\T}_{t\poiss,\alpha}(M, J)$ with $J=J_{u_0}$,  $\T$-equivariantly isometric to the one corresponding to the data $(u_t, 0, tB)$.  The process is invertible, by applying Legendre transform with respect to  $\phi_t = \varphi_t + \phi_0$. Thus, finding solution of \eqref{GK-Abreu-t} is equivalent to finding $\phi_t\in C^{\infty}(M,\R)^{\T}$  such that
\begin{equation}\label{F-t}
{\Gscal}_{(F_{(\phi_t, tB)}, J)} = \ell_{\rm ext}(\mu^{\phi_t}).
\end{equation}
We define a map
\begin{align*}
W &:\ C^{\infty}(M,\mathbb R)^{\T} \times \Wedge^2 \tor \to C^{\infty}(M, \mathbb R)^{\T}\\
W(\phi,B) &= {\Gscal}_{(F_{(\phi, B)}, J)} - \ell_{\rm ext}(\mu^{\phi}).
\end{align*}
Note that $W$ is a nonlinear differential operator defined on an open subset of $0\in C^{\infty}(M,\R)^{\T}$.  It can be extended as a  $C^1$-map on suitable Sobolev spaces by standard theory.  We are now in a position to apply the implicit function theorem to solve ${\Gscal}_{(F_{(\phi, tB)}, J)}=0$, knowing that $\phi_0=0$ is a solution at $t=0$.

The differential of $W$ at $(0, 0)$, computed in the direction of $(\dot \phi, 0)$ is the linearization at $\phi=0$ of the \emph{normalized scalar curvature} ${\Scal}_{\omega_{\phi}}^{\T}:= \Scal_{\omega_{\phi}} - \ell_{\rm ext}(\mu^{\phi})$ with respect to $\T$ of the K\"ahler metric $\omega_{\phi}=\omega_0 + dd^c_J \phi$.  As $\omega_0$ is an extremal K\"ahler metric, $DW_{(0,0)}(\dot \phi,0)$ is given by (see for instance \cite[Lemma~5.2.9]{gauduchon-book} applied to $G=\T$ or  Corollary~\ref{c:GK-dot-scal}):
\[ DW_{(0,0)} (\dot \phi,0) = \mathbb{L}_{\omega_0}(\dot \phi), \]
where ${\mathbb L_{\gw_0}}$ is the Lichnerowicz Laplacian of $\omega_0$.  The kernel of $\mathbb L$ restricted to the space of $\T$-invariant smooth functions $C^{\infty}(M,\R)^{\T}$ therefore consists  of the $\T$-invariant $\omega_0$-Killing potentials.  Using the maximality of $\T$, these are precisely the pull-backs of affine-linear  functions on $\Pol$ by the moment map $\mu$. To remedy the nontriviality of $\Ker(\mathbb L_{\gw_0})$, LeBrun--Simanca~\cite{LS} proposed to consider the modified operator
\[ \left({\rm Id} - \Pi_0\right) \left({\Gscal}_{(F_{(\phi, B)}, J)} - \ell_{\rm ext}(\mu^{\phi})\right), \]
where $\Pi_0$ denotes the $L^2(M,dV_{\omega_0})$-orthogonal  projection to the vector space of pull-backs by $\mu$ of affine-linear functions.  This operator now acts on a neighborhood of $0\in \left({\rm Id} -\Pi_0\right)(C^{\infty}(M,\R)^{\T})$ and takes values in $\left({\rm Id} -\Pi_0\right)(C^{\infty}(M,\R)^{\T})$.  The corresponding linearization is $\left({\rm Id} -\Pi_{0}\right) \mathbb L_{\gw_0} = {\mathbb L}_{\omega_0}$ (as ${\mathbb L}_{\omega_0}$ is self-adjoint on $L^2(M, dV_{\omega_0})$) and it has a trivial kernel on $\left({\rm Id} -\Pi_0\right)(C^{\infty}(M,\R)^{\T})$. By the implicit function theorem, one can then find a family  $\phi_t, \, \phi_0=0$ (in a suitable Sobolev space embedded in $C^{4}(M)$) such that
\[\left({\rm Id} -\Pi_0\right)\left({\Gscal}_{(F_{(\phi_t, tB)}, J)} - \ell_{\rm ext}(\mu^{\phi_t})\right) =0. \]
Notice that, by the definition of $\ell_{\rm ext}$ (see \eqref{AK-extremal} and \eqref{toric-by-parts} with $f$ affine-linear),
\[\Pi_{t} \left({\Gscal}_{(F_{(\phi_t, tB)}, J)} - \ell_{\rm ext}(\mu^{\phi_t})\right) = 0,\] where $\Pi_{t}$ stands for the $L^2(M, dV_{F_{(\phi_t, tB)}})$ orthogonal projection to the space of  pull-backs by $\mu^{\phi_t}$ of affine-linear functions on $\Pol$. As ${\rm Ker}\left({\rm Id} - \Pi_0\right)  \cap {\rm Ker}\left(\Pi_t\right)=0$ for $t$ close to $0$, we conclude that \eqref{F-t} holds. The bootstrapping argument for $\phi_t$ uses the ellipticity of the linearization of $\Gscal_{(F_{\phi} tB)}$,  and the fact that the vector field $\omega_{\phi}^{-1}\left(d\ell_{\rm ext}(\mu^{\phi})\right)$ must be $J$-holomorphic (and thus smooth) as soon as $\phi \in C^{4}(M)$. \end{proof}

\begin{rmk}  Theorem~\ref{t:LS}  yields new  examples  of  cscGK structures  even on ${\bf CP}^2$ endowed with a general toric Poisson tensor $\poiss_J$. The existence results for such metrics obtained  in \cite{boulanger, Gotomoment, GotoLichne} apply only for special  toric Poisson tensors  which correspond to $B=0$ in our notation. In this case, \eqref{GK-Abreu} reduces to the usual Abreu equation and is  trivially solved by the symplectic potential of the Fubini--Study metric.

In general, when we fix the generalized K\"ahler class $\GK_{\poiss,\alpha}^0$ (and whence the cohomology class of the symplectic form and the Poisson tensor $\poiss_J$), the existence problem for extremal generalized K\"ahler structures  is not scale invariant. Our result above solves  Question~\ref{q:existence} only partially, as we have no control on how small the scale $t\poiss_J$ in Theorem~\ref{t:LS} is.
\end{rmk}

\bibliographystyle{hamsplain}

\end{document}